\documentclass[12pt]{article}
\usepackage{amsmath,amssymb,amsthm,amsfonts,amscd,euscript,verbatim, t1enc, newlfont, xypic}
\usepackage{hyperref}
\hypersetup{breaklinks}

 \theoremstyle{plain}
\newtheorem{theo}{Theorem}[subsection]

\newtheorem{pr}[theo]{Proposition}

 \newtheorem{lem}[theo]{Lemma}
 \newtheorem{coro}[theo]{Corollary}
  
\theoremstyle{remark}
\newtheorem{rema}[theo]{Remark}

\theoremstyle{definition}
\newtheorem{defi}[theo]{Definition}
\newtheorem*{notat}{Notation}

 \newcommand\lan{\langle}
\newcommand\ra{\rangle}

\newcommand\gd{\mathfrak{D}}
\newcommand\gdp{\mathfrak{D}'}
\newcommand\gds{\mathfrak{D}_s}
\newcommand\gdn{\mathfrak{D}^{naive}}
\newcommand\gdg{\mathfrak{D}^{gen}}
\newcommand\gdbr{\mathfrak{D}_{bir}}
\newcommand\gdsbr{\mathfrak{D}_{s,bir}}

\newcommand\ob{^{-1}}
\newcommand\smc{{SmCor}}
\newcommand\ssc{{Shv(SmCor)}}

\newcommand\dmk{{D^-(Shv(SmCor))}}

\newcommand\psc{{PreShv(SmCor)}}

\newcommand\dmge{DM^{eff}_{gm}{}}
\newcommand\dmgm{DM_{gm}}
\newcommand\dme{DM_-^{eff}{}}

\newcommand\dgm{\operatorname{DG-Mod}}
\newcommand\cha{\operatorname{char}}

\newcommand\mg{M_{gm}}
\newcommand\mgc{M_{gm}^c{}}
\newcommand\obj{Obj}
\newcommand\mo{Mor}
\newcommand\id{id}
\newcommand\cu{\underline{C}}
\newcommand\du{\underline{D}}
\newcommand\au{\underline{A}}
\newcommand\eu{\underline{E}}

\newcommand\cupr{\underline{C}'}

\newcommand\z{{\mathbb{Z}}}
\newcommand\q{{\mathbb{Q}}}
\newcommand\af{\mathbb{A}}

\newcommand\hhi{\mathbb{H}^i}
\newcommand\dd{\mathcal{D}}
\newcommand\dc{\mathcal{C}}
\newcommand\dk{\mathcal{K}}

\newcommand\pt{pt}

\newcommand\al{\alpha}
\newcommand\be{\beta}

\newcommand\de{\delta}

\newcommand\ns{\{0\}}

\DeclareMathOperator\prli{\varprojlim}
\DeclareMathOperator\inli{\varinjlim}

\DeclareMathOperator\res{\operatorname{Res}}

\newcommand\cihom{Co-{\underline{Hom}}}

\newcommand\chow{Chow}
\newcommand\chowe{Chow^{eff}}

\newcommand\ab{Ab}
\newcommand\var{Var}
\newcommand\sv{SmVar}
\newcommand\spv{SmPrVar}

\newcommand\spe{\operatorname{Spec}\,}

 \DeclareMathOperator\ke{\operatorname{Ker}}
 \DeclareMathOperator\cok{\operatorname{Coker}}
\DeclareMathOperator\imm{\operatorname{Im}}
\DeclareMathOperator\co{\operatorname{Cone}}

\DeclareMathOperator\exts{\operatorname{Ext}}

\DeclareMathOperator\adfu{\operatorname{AddFun}}

\newcommand\hrt{{\underline{Ht}}}
\newcommand\hw{{\underline{Hw}}}

\begin{document}

 \title{
 Motivically functorial coniveau spectral sequences for cohomology;
 direct summands of 
 cohomology of function fields}
 \author{M.V. Bondarko
   \thanks{ The author gratefully acknowledges the support from Deligne 2004
Balzan prize in mathematics. The work is also supported by RFBR
(grants no.
 08-01-00777a and 10-01-00287) and INTAS (grant no.
 05-1000008-8118).}}
 \maketitle

\begin{abstract}
The goal of this paper is to prove that coniveau spectral sequences
are motivically functorial for all cohomology theories that can  be
factorized through motives. To this end the motif of a smooth
variety over a countable field $k$
 is decomposed
(in the sense of Postnikov towers) into twisted (co)motives of its
points;
 this is generalized to arbitrary Voevodsky's motives.
 In order to study the functoriality of this construction, we use
 the formalism of weight structures (introduced in the previous paper).
 We also develop this  formalism (for general triangulated categories) further,
    and relate
  it with a  new notion of  {\it nice duality}  of
   triangulated categories (that's a sort of a pairing for two distinct categories); this piece
  of homological algebra could be interesting for itself. 

We construct a certain {\it Gersten} weight structure for a 
triangulated category of {\it comotives} that contains $DM^{eff}_{gm}$ as well as
(co)motives of function fields over $k$.
 It turns out that the corresponding {\it weight spectral
  sequences} generalize the classical coniveau ones
  (to cohomology of arbitrary motives). When a cohomological functor
   is represented
 by a $Y\in \operatorname{Obj} DM^{eff}_-$,  the corresponding coniveau spectral
 sequences can be expressed in  terms of the (homotopy)
$t$-truncations of $Y$; this extends to motives the seminal coniveau
spectral sequence computations of Bloch and Ogus.

 We also obtain
that the comotif of a smooth connected semi-local scheme is a direct summand
of the comotif of its generic point; comotives of function
fields contain twisted comotives of their residue fields (for all
geometric valuations). Hence similar results hold for any cohomology
of (semi-local) schemes mentioned.

\end{abstract}

\tableofcontents

 \section*{Introduction}

 Let $k$ be our perfect base field.

 We recall two very important statements concerning coniveau
 spectral sequences. 
The first one is the calculation
 of  $E_2$ of the coniveau spectral sequence for
cohomological theories
 that satisfy certain
 conditions; see \cite{blog} and \cite{suger}. It was
proved by Voevodsky that
 these conditions are fulfilled by any theory $H$ represented
by a motivic complex $C$
 (i.e. an object of $\dme$; see  \cite{1});
then the  $E_2$-terms of the spectral sequence
 could be calculated in terms of the (homotopy $t$-structure)
cohomology of $C$.
  This result implies the second one: $H$-cohomology
 of a smooth connected semi-local scheme (in the sense of \S4.4 of \cite{3}, i.e. actually an affine essentially smooth one) injects into the cohomology
of its generic point;
 the latter statement was extended to  all (smooth connected)
primitive schemes by M. Walker.

 The main goal of the present paper is to construct
   (motivically) functorial coniveau spectral sequences converging
  to  cohomology of arbitrary motives; there should
exist a description of these
  spectral sequences (starting from $E_2$) that is similar
 to the description
  for the case of cohomology of smooth varieties (mentioned above).

   A related objective is to clarify the nature of the injectivity  result
   mentioned; it turned our that  (in the case of a countable $k$) the
    cohomology of a smooth connected semi-local
   (more generally,  primitive) scheme is actually a direct summand of the cohomology
   of its generic point. Moreover, the (twisted) cohomology of
a residue
    field of a function field $K/k$
    (for any geometric valuation of $K$) is a direct summand of the
cohomology of $K$. We actually prove more in \S\ref{sext}.

Our main homological algebra tool is the theory of
 {\it weight structures} (in triangulated categories; we usually denote a weight structure by $w$)
 introduced in the previous paper \cite{bws}. In this article
 we develop it further; this part of the paper  could be interesting
  also to readers not acquainted with  motives (and could be read
  independently from the rest of the paper). 
  In particular, we study  {\it nice dualities} (certain pairings) of (two distinct)
   triangulated categories; it seems that this subject was not
    previously considered in the literature at all.
This allows us to  generalize the concept  of {\it adjacent} weight
and $t$-structures ($t$) in a triangulated category (developed in
\S4.4 of \cite{bws}): we introduce the notion  of {\it orthogonal
structures} in (two possibly distinct) triangulated categories. If
$\Phi$ is a  nice duality of triangulated $\cu,\du$, $X\in \obj
\cu,\ Y\in \obj \du$,  $t$ is orthogonal to $w$,  then the spectral
sequence $S$ converging to  $\Phi(X,Y)$ that comes from the
$t$-truncations of $Y$ is naturally isomorphic (starting from $E_2$)
to the {\it weight spectral sequence} $T$ for the functor $\Phi(-,Y)$.
$T$ comes from {\it weight truncations} of $X$
(note that those generalize stupid truncations 
for complexes). 
Our approach
yields an abstract alternative to the method of comparing similar
spectral sequences using filtered complexes (developed by Deligne
and Paranjape, and used in \cite{paran}, \cite{ndegl}, and
\cite{bws}). Note also that we relate $t$-truncations in $\du$ with
{\it virtual $t$-truncations} of cohomological functors on $\cu$.
Virtual $t$-truncations for
cohomological functors are defined for any $(\cu,w)$ 
(we do not need any triangulated 'categories of functors'
 or $t$-structures for them here); this notion was introduced
 in \S2.5 of \cite{bws} and is studied further in the current paper.

Now, we explain why we really need a certain new category of
{\it comotives} (containing Voevodsky's $\dmge$), and so
the theory of adjacent structures (i.e. orthogonal structures
 in the case $\cu=\du$, 
$\Phi=\cu(-,-)$) is not sufficient for our purposes.
It was already proved in \cite{bws} that  weight structures provide
   a powerful tool for constructing spectral sequences; they
also  relate
    the cohomology of objects of triangulated categories with
      $t$-structures adjacent to them.  Unfortunately, a
weight structure
      corresponding to  coniveau spectral sequences cannot exist on
      $\dme\supset \dmge$
            since these categories do not contain (any) motives
for function fields over $k$
      (as well as motives of other schemes not of finite
      type over $k$; still cf. Remark \ref{rdualn}(5)).
       Yet these motives should (co)generate the {\it heart}
      of this weight structure (since the objects of this heart
      should corepresent covariant exact functors from the category
      of homotopy invariant sheaves with transfers to $\ab$).

      So, we need a category that would contain
      certain  homotopy limits
      of objects of $\dmge$. We succeed in constructing a
triangulated category
       $\gd$ (of {\it comotives}) that allows us to reach
       the objectives listed.
 Unfortunately,
       in order to control morphisms between the homotopy
       limits mentioned we have to
       assume
       $k$ to be countable. In this case there exists a
       large enough
triangulated category
        $\gds$ ($\dmge\subset \gds\subset \gd$) 
        endowed with a
        certain {\it Gersten weight structure} $w$; its
        heart is 'cogenerated' by comotives of function fields. $w$ is
(left) orthogonal to the homotopy $t$-structure on $\dme$ and
(so) is closely connected
        with coniveau spectral sequences and Gersten
resolutions for sheaves.
Note still: we need $k$ to be countable only in order to construct the Gersten
weight structure. So those readers who would just want to have a category
that contains reasonable homotopy limits of geometric motives
(including comotives of function fields and of smooth semi-local schemes),
and consider cohomology theories for this category, may freely ignore
 this restriction. 
Moreover, for an arbitrary $k$ one can still pass to a countable
homotopy limit in the Gysin distinguished triangle (as in
Proposition \ref{pinfgy}). Yet for an uncountable $k$ countable
homotopy limits don't seem to be interesting; in particular, they
definitely do not allow to construct a Gersten weight structure
(in this case).
        
 So, we consider a certain triangulated category
  $\gd\supset \dmge$ that (roughly!) 'consists of' (covariant)
 homological functors
  $\dmge\to \ab$. In particular, objects of $\gd$ define
covariant functors
  $\sv\to \ab$ (whereas another 'big' motivic category $\dme$
 defined by Voevodsky is constructed
  from certain sheaves i.e.
  contravariant functors $\sv\to\ab$; this is also true for
all motivic homotopy categories of Voevodsky and Morel).
Besides, $\dmge$ yields a family of (weak) cocompact
cogenerators for $\gd$.
 This is why we call objects of
  $\gd$  comotives. 
 Yet note that the embedding $\dmge\to \gd$
 is covariant (actually, we invert the arrows in the corresponding 'category of functors' in order to make the
   Yoneda embedding functor covariant),
  as well as the functor that sends a smooth scheme $U$
(not necessarily of finite
  type over $k$) to its comotif (which coincides with
 its  motif if $U$
  is a smooth variety).

We also recall the {\it} Chow weight structure $w'_{\chow}$
introduced in \cite{bws}; the corresponding {\it Chow-weight}
spectral sequences are isomorphic to the classical (i.e.
Deligne's) weight spectral sequences when the latter are defined. $w'_{\chow}$ could be naturally extended to a weight structure $w_{\chow}$ for $\gd$.
We
always have a natural comparison morphism from the Chow-weight spectral
sequence for $(H,X)$ to the corresponding coniveau one; it is an isomorphism for
any birational cohomology theory. We consider the category of
birational comotives $\gdbr$  i.e. the localization of $\gd$ by
$\gd(1)$ (that contains the category of birational geometric motives introduced in \cite{kabir}; though some of the results of this unpublished preprint are erroneous, this makes no difference for the current paper). It turns our that $w$ and $w_{\chow}$ induce the same
weight structure $w'_{bir}$ on $\gdbr$. Conversely, starting from
$w'_{bir}$ one can 'glue' (from {\it slices}) the weight structures induced by
  $w$ and $w_{\chow}$ on $\gd/\gd(n)$ for all $n>0$.
 Moreover, these structures belong to an interesting family of weight
 structures indexed by a single integral parameter! It
 could be interesting to consider other members of this family.
We relate briefly these observations with those of A. Beilinson
(in \cite{beilnmot} he proposed a 'geometric' characterization
of the conjectural motivic $t$-structure).

 Now we describe the connection of our results with
related results
  of F. Deglise (see \cite{deggenmot}, \cite{de2}, and \cite{ndegl}).
  He considers a certain category of pro-motives whose
objects are  naive
  inverse limits of objects
   of $\dmge$ (this category is not triangulated, though
it is {\it pro-triangulated}
   in a certain sense).  This approach allows to obtain
(in a universal way) classical coniveau spectral
    sequences for
    cohomology of motives of smooth varieties; Deglise
also 
proves their relation with the homotopy
    $t$-truncations for cohomology represented by an object of
    $\dme$. Yet for cohomology theories not coming from motivic
    complexes,
    this method does not seem to extend
to (spectral sequences for
    cohomology of)  arbitrary motives; motivic
    functoriality is not obvious also. Moreover, Deglise
didn't prove that the pro-motif
    of a (smooth connected) semi-local scheme is a
    direct summand of the pro-motif of its generic point
(though this is true, at
    least in the case of a countable $k$). We will tell
much more about our strategy and on the relation of our
results with those of Deglise in \S\ref{sprom} below. Note also
that our methods are much more convenient for studying
functoriality (of coniveau spectral sequences) than the methods
 applied by M. Rost in the related context of cycle modules
 (see \cite{rostc} and \S4 of \cite{de2}).


The author would like to indicate the interdependencies of
the parts of this text (in order to simplify reading for
those who are not interested in all of it). Those readers
who are not (very much) interested in (coniveau) spectral
sequences, may avoid most of section \ref{swnew}
and read only \S\S\ref{srws}
--\ref{sbpws}
(Remark \ref{rfunct} could also be ignored).
Moreover, in order to prove our direct summands results
(i.e. Theorem \ref{tds1n}, Corollary \ref{tds2}, and Proposition
 \ref{cdscoh}) one needs only a small portion of the theory
 of weight structures; so a reader very reluctant to study this
  theory may try to derive them from the results of \S\ref{scomot}
   'by hand' without reading \S\ref{swnew} at all. Still,
   for motivic functoriality of coniveau spectral sequences
   and filtrations (see Proposition \ref{rwss} and 
   Remark \ref{rrwss}) one needs more of weight structures.
 On the other hand, those readers who are more interested in
 the (general) theory of triangulated categories may restrict
 their attention to \S\S\ref{dtst}-- \ref{sextkrau} and \S\ref{swnew};
 yet note that the rest of the paper describes in detail an
 important (and quite non-trivial) example of a weight
structure which is orthogonal to a $t$-structure with respect
to a nice duality (of triangulated categories).
Moreover, much of section \S\ref{sapcoh} could also be extended to
 a  general setting of a triangulated category satisfying properties
 similar to those listed in Proposition
\ref{pprop}; yet the author chose not to do this in order to make
the paper somewhat less abstract.

 Now we list the contents of the paper. More details could
 be found at the beginnings of 
 sections.

 We start \S1 with the recollection of
$t$-structures,  idempotent
  completions, and Postnikov towers for triangulated categories.
  We describe a method for extending cohomological functors from
  a full triangulated subcategory to the whole $\cu$
  (after H. Krause).
  Next we recall some results and definitions for
  Voevodsky's motives (this includes certain properties
of Tate twists for motives and
  cohomological functors).
  Lastly, we define pro-motives (following Deglise)
and compare them with
   our triangulated category $\gd$ of comotives. This
allows to explain our strategy
    step by step.

  \S\ref{swnew} is dedicated to weight structures.
  First we remind the basics of this theory
  (developed in \S\cite{bws}). Next we recall
 that a cohomological functor $H$ from an (arbitrary
 triangulated category) $\cu$ endowed with a weight
 structure $w$ could be 'truncated' as if it belonged
  to some triangulated category of functors (from $\cu$)
  that is endowed with a $t$-structure; we call the
   corresponding pieces of $H$ its {\it virtual $t$-truncations}.
 We recall the notion of   weight spectral sequence
 (introduces in ibid.). We prove that the derived exact
 couple for a weight spectral sequence  could be described
 in terms of virtual $t$-truncations.
 Next  we introduce the definition of
a {\it (nice) duality} $\Phi:\cu^{op}\times \du\to \au$
(here $\du$ is triangulated, $\au$ is abelian), and
of {\it orthogonal} weight and
 $t$-structures (with respect to $\Phi$). 
 If $w$ is  orthogonal to $t$, then the virtual
 $t$-truncations (corresponding to $w$)  of functors
 of the type $\Phi(-,Y),\ Y\in \obj\du$, are exactly the
 functors 'represented via $\Phi$' by the actual
 $t$-truncations of $Y$ (corresponding to $t$).
 Hence if $w$ and $t$ are orthogonal with respect to
  a nice duality, the weight spectral sequence converging
  to $\Phi(X,Y)$ (for $X\in \obj \cu,\ Y\in \obj \du$)
  is naturally isomorphic (starting from $E_2$) to the one
  coming from $t$-truncations of $Y$.  We also mention
  some alternatives and predecessors of  our results.
 Lastly we  compare weight decompositions, virtual
$t$-truncations, and weight spectral sequences corresponding to
distinct weight structures (in possibly distinct triangulated
categories).

In \S\ref{scomot} we describe the main properties of
$\gd\supset \dmge$.
The exact choice of $\gd$ is not important for
 most of this paper; so
 we just
list the main properties of $\gd$ (and its certain
{\it enhancement} $\gdp$) in \S\ref{comot}. We construct
 $\gd$ using the formalism of differential graded modules
  in \S\ref{rdg}
later. Next we define comotives for (certain) schemes and
ind-schemes of infinite type over $k$ (we call them pro-schemes). We
recall the notion of  primitive scheme. All (smooth) semi-local
pro-schemes are primitive; primitive schemes  have all nice
'motivic' properties of semi-local pro-schemes. We prove that there
are no $\gd$-morphisms
 of positive degrees between the comotives of primitive schemes
(and also between certain Tate
 twists of those). In \S\ref{scgersten} we prove that the Gysin distinguished triangle for motives of smooth varieties (in $\dmge$) could be naturally extended to comotives of pro-schemes. This allows to construct
certain
 Postnikov towers for the comotives of pro-schemes; these towers
are closely related
 with classical
 coniveau spectral sequences for cohomology.

\S\ref{sapcoh} is central in this paper.
 We introduce a certain {\it Gersten weight structure}
 for a certain triangulated category $\gds$
 ($\dmge\subset \gds\subset \gd$).  We prove that
 Postnikov towers constructed
 in \S\ref{scgersten} are actually
 {\it weight Postnikov towers} with respect to $w$.
We deduce our (interesting) results on
direct summands of the comotives of function fields. 
   We translate these results  to cohomology in the obvious way.

 Next we  prove that weight spectral sequences for
 the cohomology of $X$
(corresponding to
 the Gersten weight structure)
    are naturally isomorphic (starting from $E_2$) to  the
    classical coniveau spectral sequences
    if $X$
is the motif of a smooth variety; so we call these spectral sequence
coniveau ones in the general case also.
 We also prove that
 the Gersten weight structure $w$ (on $\gds$) is  orthogonal to
 the homotopy $t$-structure $t$ on $\dme$ (with respect to a
  certain  $\Phi$). It follows that for an
 arbitrary $X\in\obj \gds$, for a cohomology theory represented
 by $Y\in \obj \dme$ (any choice of)
 the coniveau spectral sequence that converges to $\Phi(X,Y)$
  could be described
in terms of the
 $t$-truncations of $Y$ (starting from $E_2$).

 We also define coniveau spectral sequences
for cohomology of motives over uncountable  base fields  as the limits of the corresponding coniveau spectral sequences over countable perfect subfields of definition.
This definition is compatible with the classical one; so we establish motivic functoriality of coniveau spectral sequences in this case also.

We also prove that the {\it Chow weight structure} for $\dmge$
(introduced in \S6 of \cite{bws}) could be extended to a weight structure
 $w_{\chow}$ on $\gd$. The corresponding {\it Chow-weight} spectral sequences are isomorphic to the
 classical (i.e. Deligne's) ones when the latter are defined
 (this was proved in \cite{bws} and \cite{mymot}).
 We compare Chow-weight spectral sequences with coniveau ones:
  we always have a comparison morphism; it is an isomorphism for a {\it birational} cohomology theory.
  We consider the category of birational comotives $\gdbr$  i.e. the localization of $\gd$ by $\gd(1)$. $w$ and $w_{\chow}$ induce the same weight structure $w'_{bir}$ on $\gdbr$; one almost can glue   $w$ and $w_{\chow}$ from copies of $w'_{bir}$ (one may say that these weight structures could almost be glued from the same slices with distinct shifts).

\S\ref{rdg} is dedicated to the construction of $\gd$ and
the proof of its
properties. We apply the formalism of differential graded
categories, modules
over them, and of the corresponding derived categories.
A reader not
interested in these details may 
skip (most of) this section. In fact, the author is not sure that
 there exists
only one $\gd$ suitable for our purposes; yet the choice of $\gd$
does not affect cohomology  of (the comotives of) pro-schemes
and 
of Voevodsky's motives.

 We also explain  how the differential graded modules
formalism can
 be used to define base
change  (extension and restriction of scalars) for comotives. This
allows to extend our results on direct summands of the comotives (and
cohomology) of function fields  to pro-schemes obtained from them
via base change. We also define tensoring of comotives by motives
(in particular, this yields Tate twist for $\gd$), as well as a
certain
 cointernal Hom (i.e. the corresponding left adjoint functor).

 \S\ref{ssupl}
 is dedicated to properties of comotives that are not (directly)
 related with the main
 results of the paper; we also make several comments.
 We recall the definition of the additive category $\gdg$
 of generic motives (studied in \cite{deggenmot}).  We  prove
that the exact conservative {\it weight complex}
 functor corresponding to $w$ (that exists by the general
theory of weight structures)
 could be modified to an exact conservative $WC:\gds\to K^b(\gdg)$.
Next we prove  that  a
cofunctor $\hw\to \ab$ is representable by a homotopy invariant
sheaf with transfers whenever is converts all products into direct
sums.

We also note that our theory could be easily extended to (co)motives
with coefficients in an arbitrary ring. Next we note (after B. Kahn)
that reasonable motives of pro-schemes with compact support
 do exist
in $\dme$; this observation could be used for the construction of an
alternative model for  $\gd$. Lastly we describe which parts of
our argument do not work (and which do work) in the case of an
uncountable $k$.

A caution: the notion of  weight structure is  quite a general
formalism for triangulated categories. In particular, one
triangulated category can support several distinct weight structures
(note that there is a similar situation with $t$-structures). In
fact, we  construct an example
 for such a situation
in this paper (certainly, much simpler examples exist): we define the
Gersten weight structure $w$ for $\gds$ and a Chow weight structure
$w_{\chow}$ for $\gd$. Moreover, we show in \S\ref{sgdbr} that these weight structures are compatible with certain weight structures defined on the localizations $\gd/\gd(n)$ (for all $n>0$). These two series of weight structures
are definitely distinct: note that $w$ yields coniveau spectral
sequences, whereas $w_{\chow}$ yields Chow-weight spectral
sequences, that generalize Deligne's weight spectral sequences for
\'etale and mixed Hodge cohomology (see \cite{bws} and
\cite{mymot}). Also, the weight complex functor constructed in
\cite{mymot} and \cite{bws}
 is quite distinct from the one considered in \S\ref{scompdegl} below
  (even the targets of the functors mentioned are completely distinct).

The author is deeply grateful to  prof. F. Deglise, prof. D. H\'ebert, prof. B. Kahn,  prof. M. Rovinsky, prof. A. Suslin,  prof. V. Voevodsky, and to the referee
for their interesting remarks.  The author gratefully acknowledges
the support from
   Deligne 2004
Balzan prize in mathematics. The work is also supported by RFBR
(grants no.
 08-01-00777a and 10-01-00287a). 
\begin{notat}

 For a category $C,\ A,B\in\obj C$, we denote by
$C(A,B)$ the set of  $C$-morphisms from  $A$ to $B$.
For categories $C,D$ we write 
$D\subset C$ if $D$ is a full 
subcategory of $C$.

 For additive $C,D$ we denote by $\adfu(C,D)$ the
category of additive functors from $C$ to $D$
(we will ignore set-theoretic difficulties here since
 they do not affect our arguments seriously). 

$\ab$ is the category of abelian groups.
For an additive $B$ we will
denote by $B^*$ the category $\adfu(B,\ab)$ and by $B_*$ the
category $\adfu(B^{op},\ab)$. Note that both of these are abelian.
Besides, Yoneda's lemma gives full embeddings of $B$ into $B_*$ and
of $B^{op}$ into $B^*$ (these send $X\in\obj B$ to $X_*=B(-,X)$ and
to $X^*=B(X,-)$, respectively).

For a category $C,\ X,Y\in\obj C$, we say that $X$ is a {\it
retract} of $Y$ if $\id_X$ can be factorized through $Y$. Note
that when $C$ is triangulated or abelian then $X$ is a  retract of
$Y$ if and only if $X$ is its direct summand. For any $D\subset C$
the subcategory $D$ is called {\it Karoubi-closed} in $C$ if it
contains all retracts of its objects in $C$. We will call the
smallest Karoubi-closed subcategory of $C$ containing $D$  the {\it
Karoubi-closure} of $D$ in $C$; sometimes we will use the same term
for the class of objects of the Karoubi-closure of a full subcategory
of $C$ (corresponding to some subclass of $\obj C$).

For a category $C$ we denote by $C^{op}$ its opposite category.

 For an additive $\cu$ an object $X\in \obj\cu$ is called cocompact if
$\cu(\prod_{i\in I} Y_i,X)=\bigoplus_{i\in I} \cu(Y_i,X)$ for any
set $I$ and any $Y_i\in\obj \cu$ such that the product exists (here
we don't need to demand all products to exist, though they actually
will exist below).

For $X,Y\in \obj \cu$ we will write $X\perp Y$ if $\cu(X,Y)=\ns$.
For $D,E\subset \obj \cu$ we will write $D\perp E$ if $X\perp Y$
 for all $X\in D,\ Y\in E$.
For $D\subset \cu$ we will denote by $D^\perp$ the class
$$\{Y\in \obj \cu:\ X\perp Y\ \forall X\in D\}.$$
Sometimes we will denote by $D^\perp$ the corresponding
 full subcategory of $\cu$. Dually, ${}^\perp{}D$ is the class
$\{Y\in \obj \cu:\ Y\perp X\ \forall X\in D\}$. This convention is opposite to the one of \S9.1 of \cite{neebook}.

In this paper all complexes will be cohomological i.e. the degree of
all differentials is $+1$; respectively, we will use cohomological
notation for their terms.

For an additive category $B$ we denote by $C(B)$ the category of
(unbounded) complexes over it. $K(B)$ will denote the homotopy
category of complexes. If $B$ is also abelian, we will denote by
$D(B)$ the derived category of $B$. We will also need certain bounded
analogues of these categories (i.e. $C^b(B)$, $K^b(B)$, $D^-(B)$).

$\cu$ and $\du$ will usually denote some triangulated categories.
 We will use the
term 'exact functor' for a functor of triangulated categories (i.e.
for a for a functor that preserves the structures of triangulated
categories).

$\au$ will usually denote some abelian category.
We will call a covariant additive functor $\cu\to \au$
for an abelian $\au$ {\it homological} if it converts distinguished
triangles into long exact sequences; homological functors
$\cu^{op}\to \au$ will be called {\it cohomological} when considered
as contravariant functors $\cu\to \au$.

$H:\cu^{op}\to \au$ will always be additive; it will usually
be cohomological.

For $f\in\cu (X,Y)$, $X,Y\in\obj\cu$, we will call the third vertex
of (any) distinguished triangle $X\stackrel{f}{\to}Y\to Z$ a cone of
$f$. Note that different choices of cones are connected by
non-unique isomorphisms, cf. IV.1.7 of \cite{gelman}. Besides, in
$C(B)$ we have canonical cones of morphisms (see section \S III.3 of
ibid.).

We will often specify a distinguished triangle by two of its
morphisms.

When dealing with triangulated categories we (mostly) use
conventions and auxiliary statements of \cite{gelman}. For a set of
objects $C_i\in\obj\cu$, $i\in I$, we will denote by $\lan C_i\ra$
the smallest strictly full triangulated subcategory containing all $C_i$; for
$D\subset \cu$ we will write $\lan D\ra$ instead of $\lan C:\ C\in\obj
D\ra$.  

We will say that  $C_i$ generate $\cu$ if $\cu$ equals $\lan
C_i\ra$. We will say that $C_i$ {\it weakly cogenerate} $\cu$ if for
$X\in\obj\cu$ we have $\cu(X,C_i[j])=\ns\ \forall i\in I,\
j\in\z\implies X=0$ (i.e. if ${}^\perp\{C_i[j]\}$ contains only
 zero objects).


We will call a partially ordered set $L$ a (filtered)
{\it projective system}
 if for any $x,y\in L$ there exists some maximum i.e. a $z\in L$
 such that $z\ge x$ and $z\ge y$. By abuse of notation, we will identify
  $L$ with the following category $D$: $\obj D=L$; $D(l',l)$ is empty
  whenever $l'< l$, and consists of a single morphism otherwise;
  the composition of morphisms is the only one possible.
If  $L$ is a projective system, $C$ is some category, $X:L\to C$ is
a covariant functor, we will denote $X(l)$ for $l\in L$ by $X_l$. We
will write $X=\prli_{l\in L} X_l$ for the limit of this functor (if it exists); we
will call it the inverse limit of $X_l$. We will denote the colimit
of a contravariant functor $Y:L\to C$ by $\inli_{l\in L}Y_l$ and
call it the direct limit. Besides, we will sometimes call the
categorical image of $L$ with respect to such an $Y$ an {\it
inductive system}.


Below $I,L$ will often be projective systems; we will
usually require $I$ to be countable.

A subsystem $L'$ of $L$ is a partially ordered subset in
which maximums exist
(we will also consider the corresponding full subcategory of $L$).
We will call $L'$ unbounded in $L$ if  for any $l\in L$
there exists an $l'\in L'$
such that $l'\ge l$.

$k$ will be our perfect base field. Below we will usually demand $k$
to be countable. Note: this yields that for any variety the set
of its closed (or open) subschemes is countable.

We also list central definitions and main notation of this paper.

 First we list the main
(general) homological algebra definitions. $t$-structures,
$t$-truncations, and Postnikov towers in triangulated categories are
defined in \S\ref{dtst}; weight structures, weight decompositions,
weight truncations, weight Postnikov towers, and weight complexes
are considered in \S\ref{srws}; 
virtual $t$-truncations and nice exact
complexes of functors are defined in \S\ref{svirtr}; weight spectral
sequences are studied in \S\ref{scovi}; (nice) dualities and
orthogonal weight and $t$-structures are defined in Definition
\ref{ddual}; right and left weight-exact functors are defined in
Definition \ref{dwefun}.

Now we list notation (and some definitions) for motives.
$\dmge\subset\dme $, $HI$ and the homotopy $t$-structure for $\dmge$
are defined in \S\ref{dvoev};  Tate twists are considered in
\S\ref{sttw};  $\gdn$ is defined in \S\ref{sprom}; comotives ($\gd$
and $\gdp$) are defined in \S\ref{comot}; in \S\ref{spsch} we
discuss pro-schemes and their comotives; 
in \S\ref{sprim} we recall the definition of a
primitive scheme;  in \S\ref{sdgersten} we define the Gersten weight
structure $w$ on
 a certain triangulated $\gds$; 
  we consider $w_{\chow}$ in \S\ref{stchow}; $\gdbr$ and $w'_{bir}$ are defined in \S\ref{sgdbr}; several
differential graded constructions (including extension and
restriction of scalars for comotives) are considered in \S\ref{rdg};
 we define
  $\gdg$ and $WC:\gds\to K^b(\gdg)$ in
\S\ref{scompdegl}.

\end{notat}

\section{Some preliminaries on  triangulated categories and
motives} \label{rtriang}

\S\ref{dtst} we
recall the notion of  $t$-structure
(and introduce some notation for it), recall
the
notion of the idempotent completion of an additive category;
we also recall that any small abelian category could be
faithfully embedded into $\ab$ (a well-known result by Mitchell).

In \S\ref{sextkrau} we describe (following H. Krause) a natural
method for extending cohomological
functors from a full triangulated $\cu'\subset\cu$ to $\cu$.

In \S\ref{dvoev}
 we recall some 
 definitions and results  of Voevodsky.

 In \S\ref{sttw} we recall the notion of  Tate twist;
 we study the properties of Tate twists for motives and
homotopy invariant sheaves.

 In \S\ref{sprom} we define pro-motives (following \cite{deggenmot}
and \cite{de2}). These are not necessary for our main result; yet they
allow to explain our methods step by step. We also describe in detail
the relation of our constructions and results with those of Deglise.

\subsection{
$t$-structures, Postnikov towers,
idempotent completions, and an embedding theorem of  Mitchell}
\label{dtst}

To fix the notation we recall the definition of a $t$-structure.

\begin{defi}\label{dtstr}

A pair of subclasses  $\cu^{t\ge 0},\cu^{t\le 0}\subset\obj \cu$
for a triangulated category $\cu$ will be said to define a
$t$-structure $t$ if $(\cu^{t\ge 0},\cu^{t\le 0})$  satisfy the
following conditions:

(i) $\cu^{t\ge 0},\cu^{t\le 0}$ are strict i.e. contain all
objects of $\cu$ isomorphic to their elements.

(ii) $\cu^{t\ge 0}\subset \cu^{t\ge 0}[1]$, $\cu^{t\le
0}[1]\subset \cu^{t\le 0}$.

(iii) {\bf Orthogonality}. $\cu^{t\le 0}[1]\perp
\cu^{t\ge 0}$.

(iv) {\bf $t$-decomposition}. For any $X\in\obj \cu$ there exists
a distinguished triangle
\begin{equation}\label{tdec}
A\to X\to B[-1]{\to} A[1]
\end{equation} such that $A\in \cu^{t\le 0}, B\in \cu^{t\ge 0}$.

\end{defi}

We will need some more notation for $t$-structures.

\begin{defi} \label{dt2}

1. A category $\hrt$ whose objects are $\cu^{t=0}=\cu^{t\ge 0}\cap
\cu^{t\le 0}$, $\hrt(X,Y)=\cu(X,Y)$ for $X,Y\in \cu^{t=0}$,
 will be called the {\it heart} of
$t$. Recall (cf. Theorem 1.3.6 of \cite{BBD}) that $\hrt$ is abelian
(short exact sequences in $\hrt$ come from distinguished triangles in $\cu$).

2. $\cu^{t\ge l}$ (resp. $\cu^{t\le l}$) will denote $\cu^{t\ge
0}[-l]$ (resp. $\cu^{t\le 0}[-l]$).

\end{defi}

\begin{rema}\label{rts}
1. The axiomatics of $t$-structures is self-dual: if $\du=\cu^{op}$
(so $\obj\cu=\obj\du$) then one can define the (opposite)  weight
structure $t'$ on $\du$ by taking $\du^{t'\le 0}=\cu^{t\ge 0}$ and
$\du^{t'\ge 0}=\cu^{t\le 0}$; see part (iii) of Examples 1.3.2 in
\cite{BBD}.

2. Recall (cf. Lemma IV.4.5 in \cite{gelman}) that (\ref{tdec})
defines additive functors $\cu\to \cu^{t\le 0}:X\to A$ and $C\to
\cu^{t\ge 0}:X\to B$. We will denote $A,B$ by $X^{t\le 0}$ and
$X^{t\ge 1}$, respectively.

3. (\ref{tdec}) will be called the {\it t-decomposition} of $X$. If
$X=Y[i]$ for some $Y\in \obj\cu$, $i\in \z$, then we will denote $A$
by $Y^{t\le i}$ (it belongs to $\cu^{t\le 0}$) and $B$ by $Y^{t\ge
i+1}$ (it belongs to  $\cu^{t\ge 0}$), respectively. Sometimes we
will denote $Y^{t^\le i}[-i]$ by $t_{\le i}Y$; $t_{\ge
i+1}Y=Y^{t^\ge i+1}[-i-1]$. Objects of the type $Y^{t^\le i}[j]$ and
$Y^{t^\ge i}[j]$ (for $i,j\in \z$) will be called {\it
$t$-truncations of $Y$}.

4. We denote by $X^{t=i}$ the $i$-th cohomology of $X$ with respect
to $t$ i.e. $(Y^{t\le i})^{t\ge 0}$ (cf. part 10 of \S IV.4 of
\cite{gelman}). 

5.  The following statements are obvious (and well-known): $\cu^{t\le 0}={}^\perp
\cu^{t\ge 1}$; $\cu^{t\ge 0}= \cu^{t\le -1}{}^\perp$.

\end{rema}

Now we recall the notion of idempotent completion.

\begin{defi}
An additive category $B$ is said to be {\it idempotent complete} if
for any $X\in\obj B$ and any idempotent $p\in B(X,X)$ there exists a
decomposition $X=Y\bigoplus Z$ such that $p=i\circ j$, where $i$ is the inclusion $Y\to Y\bigoplus Z$, $j$ is the projection $Y\bigoplus Z\to Y$.\end{defi}

Recall that any additive $B$ can be canonically idempotent
completed. Its idempotent completion is (by definition) the category
$B'$ whose objects are $(X,p)$ for $X\in\obj B$ and $p\in B(X,X):\
p^2=p$; we define
$$A'((X,p),(X',p'))=\{f\in B(X,X'):\ p'f=fp=f \}.$$ It
 can be easily checked
that this category is additive and idempotent complete, and for
any idempotent complete $C\supset B$ we have a natural full
embedding $B'\to C$.

The main result of \cite{ba} (Theorem 1.5) states that an
idempotent completion of a triangulated category $\cu$ has a
natural triangulation (with distinguished triangles being all
retracts of distinguished triangles of $\cu$).

Below we will need the notion of  Postnikov tower in
a triangulated category several times (cf. \S IV2 of \cite{gelman})).

\begin{defi}\label{dpoto}
Let $\cu$ be a triangulated category.

1. Let $l\le m\in \z$.

We will call a bounded Postnikov tower for $X\in\obj\cu$ the
following data: a sequence of $\cu$-morphisms $(0=)Y_l\to
Y_{l+1}\to\dots \to Y_{m}=X$ along with distinguished triangles
\begin{equation}\label{wdeck3}
 Y_{ i} \to Y_{i+1}\to X_{i}
\end{equation}
for some $X_i\in \obj \cu$;
here $l\le i<m$.

2. An unbounded Postnikov tower for $X$ is a collection of $Y_i$ for
$i\in\z$ that is equipped (for all $i\in\z$) with: connecting arrows
$Y_i\to Y_{i+1}$ (for $i\in\z$), morphisms $Y_i\to X$ such that all
the corresponding triangles commute, and distinguished triangles
(\ref{wdeck3}).

In both cases, we will denote  $X_{-p}[p]$ by $X^p$; we will call $X^p$ the {\it factors} of our Postnikov tower.

\end{defi}

\begin{rema}\label{rwcomp}
1. Composing (and shifting) arrows from   triangles (\ref{wdeck3})
for two subsequent $i$ one can construct a complex whose terms are
$X^p$ (it is easily seen that this is a complex indeed,
cf. Proposition 2.2.2 of \cite{bws}).  This observation will be important for us below when we will consider certain weight complex functors. 

2. Certainly, a bounded Postnikov tower could be easily completed to
an unbounded one. For example, one could take $Y_i=0$ for $i<l$,
$Y_i=X$ for $i>m$; then $X^i=0$ if $i<l$ or $i\ge m$.
\end{rema}

Lastly, we recall the following (well-known) result.

\begin{pr}\label{pmth}
For any small abelian category $\au$ there exists an exact
faithful functor $\au\to \ab$.
\end{pr}
\begin{proof}
By the Freyd-Mitchell's embedding theorem, any small $\au$ could be
fully faithfully embedded into $R-\operatorname{mod}$ for some
(associative unital) ring $R$. It remains to apply the forgetful
functor $R-\operatorname{mod}\to \ab$.
\end{proof}

\begin{rema}\label{rmit}

 1. We will need this statement below in order to assume that
 objects of $\au$ 'have elements'; this will considerably simplify
 diagram chase. Note that we can assume the existence
 of elements for a not necessarily small $\au$ in the case
 when a reasoning deals only with a finite number of objects
 of $\au$ at a time.

 2. In the proof it suffices to have a faithful
  embedding  $\au\to R-\operatorname{mod}$; this weaker
   assertion was also proved by Mitchell.
\end{rema}

\subsection{Extending cohomological functors from a
triangulated subcategory }\label{sextkrau}

We describe a method for extending cohomological
functors from a full triangulated $\cu'\subset\cu$ to $\cu$
(after H. Krause). Note that below we will
apply some of the results of \cite{krause} in the dual form.
The construction requires $\cu'$ to be
skeletally small i.e. there should exist a 
 subset (not just a subclass!) $D\subset \obj \cu'$ such that any object of $\cu'$ is isomorphic to some element of $D$.
For simplicity, we will sometimes (when writing  sums over
$\obj\cupr$) assume  that $\obj \cupr$ is a set itself.
Since the distinction between small and skeletally small
categories will not affect our  arguments and results,
we will ignore it in the rest of the paper.

If $\au$ is  an abelian category, then $\adfu(\cupr^{op},\au)$
 is abelian also; complexes
in it are exact whenever they are exact 
when applied to any object of $\cu'$.


Suppose that $\au$ satisfies AB5 i.e. it is closed with respect to all
small coproducts, and  filtered direct limits of exact sequences in
$\au$ are exact.

Let $H'\in \adfu(\cupr^{op},\au)$ be an additive functor
(it will usually be cohomological).

\begin{pr}\label{pextc}
I Let $\au,H'$ be fixed.

1. There exists an  extension of $H'$ 
to an additive functor $H:\cu\to \au$. It is cohomological
 whenever $H$ is. The correspondence $H'\to H$ defines an
 additive functor $\adfu(\cupr^{op},\au)\to \adfu(\cu^{op},\au)$.

2. Moreover, suppose that in $\cu$ we have a projective
system $X_l,\ l\in L$, equipped with a compatible system of
morphisms $X\to X_l$, such that the latter system for any
$Y\in \obj \cupr$ induces an isomorphism
$\cu(X,Y)\cong \inli \cu(X_l,Y)$. Then we have $H(X)\cong \inli H(X_l)$.

II Let $X\in \obj \cu$ be fixed.

1. One can choose a family of $X_l\in \obj \cu'$ and
$f_l\in \cu(X,X_l)$ such that $(f_l)$ induce a surjection
$\bigoplus H'(X_l)\to H(X)$ for any $H',\au$, and $H$ as in assertion I1.

2. Let $F'\stackrel{f'}{\to} G' \stackrel{g'}{\to} H'$
be a (three-term) complex in $\adfu(\cupr^{op},\au)$ that is exact in
the middle; suppose that $H'$ is cohomological. Then the
complex $F\stackrel{f}{\to} G \stackrel{g}{\to} H$
(here $F,G,H,f,g$ are the corresponding extensions) is exact
 in the middle also.


\end{pr}
\begin{proof}
I1. Following \S1.2 of \cite{krause} (and dualizing it),  we consider
the abelian category $C=\cupr{}^*=\adfu(\cupr,\ab)$ (this is
$\operatorname{Mod}\,\cupr{}^{op}$ in the notation of Krause).  
The definition easily implies that direct limits in $C$ are exactly
direct limits of functors (computed at each object of $\cu'$).  We have the Yoneda's
functor $i':\cu^{op}\to C$ 
that sends $X\in\obj \cu$ to the
functor $X^*=(Y\mapsto \cu(X,Y),\ Y\in \obj \cupr)$; it is obviously
cohomological.
 We denote by $i$ the restriction of $i'$ to $\cupr$
($i$ is opposite to a full embedding).

By Lemma 2.2 of \cite{krause} (applied to the category $\cupr^{op}$)
we obtain that there exists an exact functor $G:C\to \au$  that
preserves all small coproducts and satisfies $G\circ i=H'$.
It is constructed in the following way: if for $X\in \obj\cu$
we have an exact sequence (in $C$)
 \begin{equation}\label{eckrause} \bigoplus_{j\in J}X_j^*
 \to \bigoplus_{l\in L}X_l^* \to X^* \to 0\end{equation}
 for $X_j,X_l\in C'$,
 then we set \begin{equation}\label{ekrause}G(X)=\cok
 \bigoplus_{j\in J}H'(X_j) \to \bigoplus_{l\in L}H'(X_l).\end{equation}

We define  $H=G\circ i'$; it was proved in loc. cit. that we obtain a well-defined functor this way. 
As was also proved in loc. cit., the correspondence $H'\mapsto H$ yields a functor; $H$ is
cohomological if $H'$ is. 

2. The proof of loc. cit. shows (and mentions) that $G$
respects (small) filtered inverse
limits.
Now note that our assertions imply: $X^*=\inli X_l^*$
in $C$.

II 1. This is immediate from (\ref{ekrause}).

2. Note that the assertion is obviously valid if $X\in \obj \cupr$.
We reduce the general statement to this case.

Applying Yoneda's lemma to (\ref{eckrause}) is we obtain
(canonically) some morphisms $f_l:X \to X_l$ for all
$l\in L$ and $g_{lj}:X_l\to X_j$ for all $l\in L$, $j\in J$, such that:
for any $l\in L$ almost all $g_{lj}$ are $0$; for any $j\in J$
almost all $g_{lj}$ is $0$; for any $j\in J$ we have
$\sum_{l\in L} g_{lj}\circ f_l=0$.

Now, by Proposition \ref{pmth}, we may assume that $\au=\ab$
(see Remark \ref{rmit}).
We should check: if for $a\in G(X)$ we have $g_*(a)=0$,
 then $a=f_*(b)$ for some $b\in F(X)$.

Using additivity of $\cupr$ and $\cu$, we can gather finite sets
of $X_l$ and $X_j$ into single objects. Hence we can assume
that $a=G(f_{l_0})(c)$ for some $c\in G(X_l)\ (=G'(X_l)),\ l_0\in L$
 and that $g_*(c)\in H(g_{l_0j_0})(H(X_{j_0}))$ for some $j_0\in J$,
  whereas $g_{l_0j_0}\circ f_{l_0}=0$. We complete
  $X_{l_0}\to X_{j_0}$ to a distinguished triangle
  $Y\stackrel{\al}{\to}X_{l_0}\stackrel{g_{l_0j_0}}{\to} X_{j_0}$;
   we can assume that $B\in \obj \cupr$.
We obtain that $f_{l_0}$ could be presented as $\al\circ \be$
for some $\be\in \cu(X,Y)$. Since $H'$ is cohomological, we obtain
that $H(\al)(g_*(c))=0$. Since $Y\in \obj \cu$, the complex
$F(Y)\to G(Y)\to H(Y)$ is exact in the middle; hence
$G(\al)(c)=f_*(d)$ for some $d\in F(Y)$. Then we can take $b=F(\be)(d)$.

\end{proof}

\subsection{Some definitions of Voevodsky: reminder}\label{dvoev}

We use much notation from  \cite{1}. We recall (some of) it here for
the convenience of the reader, and introduce some notation of our own.

$\var\supset \sv\supset \spv$ will denote the class of all varieties
over $k$, resp. of smooth varieties, resp. of smooth projective
varieties.

We recall  that for categories of geometric origin (in particular,
 for 
$\smc$ defined below) the addition of objects is defined via
the disjoint union of varieties operation.

 We define the category $\smc$ of smooth
correspondences. $\obj \smc=\sv$, $\smc (X,Y)=\bigoplus_U\z$ for all
integral closed $U\subset X\times Y$  that are finite over $X$ and
dominant over a connected component of $X$; the composition of
correspondences is defined in the usual way via intersections (yet,
we do not need to consider correspondences up to an equivalence
relation).

We will write $\dots \to X^{i-1}\to X^{i}\to X^{i+1}\to \dots$, for
$X^l\in\sv$, for the corresponding complex over $\smc$.

$\psc$ will denote the (abelian) category  of  additive cofunctors
$\smc\to\ab$; its objects are usually called {\it presheaves with
transfers}.

$\ssc=\ssc_{Nis}\subset \psc$ is the abelian category  of  additive
cofunctors $\smc\to\ab$ that are sheaves in the Nisnevich topology
(when restricted to the category of smooth varieties); these sheaves
are usually called {\it sheaves with transfers}.

 $\dmk$ will be the bounded above
derived category of  $\ssc$.

For $Y\in \sv$ (more generally, for $Y\in \var$, see \S4.1 of
\cite{1}) we consider $L(Y)=\smc(-,Y)\in \ssc$. For a bounded
complex $X=(X^i)$ (as above) we will denote by $L(X)$ the complex
$\dots \to L(X^{i-1})\to L(X^{i})\to L(X^{i+1})\to \dots\in
C^b(\ssc)$.

$S\in \ssc$ is called homotopy invariant if for any $X\in \sv$ the
projection $\af^1\times X\to X$ gives an isomorphism $S(X)\to
S(\af^1\times X)$. We will denote the category of homotopy
invariant sheaves (with transfers) by $HI$; it is an exact abelian subcategory of $\smc$ by Proposition 3.1.13 of \cite{1}.

$\dme\subset \dmk$ is the full subcategory of complexes whose cohomology
sheaves are homotopy invariant; it is triangulated by loc. cit.
We will need the {\it homotopy} $t$-structure on $\dme$: it is the
restriction of the canonical $t$-structure on $\dmk$ to $\dme$. 
Below (when dealing with $\dme$) we will denote it by just by $t$.
We have $\hrt= HI$.

We recall the following results
of \cite{1}.

\begin{pr}\label{pvo}
1. There exists an exact functor $RC:\dmk\to \dme$ right adjoint to the
embedding $\dme\to \dmk$.

2. $\dme (\mg(Y)[-i],F)= \hhi(F)(Y)$ (the $i$-th Nisnevich
hypercohomology of $F$ computed in $Y$) for any $Y\in \sv$.

3. Denote  $RC\circ L$ by $\mg$. 
Then the corresponding functor $K^b(\smc)\to \dme$ could be
described as a certain localization of $K^b(\smc)$.

\end{pr}

\begin{proof} See \S3 of \cite{1}.
\end{proof}

\begin{rema}\label{rdmge}
1. In \cite{1} (Definition 2.1.1) the triangulated category $\dmge$
(of {\it effective geometric motives}) was defined as the idempotent
completion of a certain localization of $K^b(\smc)$. This definition
is compatible with a {\it differential graded enhancement} for $\dmge$;
cf. \S\ref{scgd} below. Yet in Theorem 3.2.6 of \cite{1} it was
shown that $\dmge$ is isomorphic to the idempotent completion of
(the categorical image) $\mg(C^b(\smc))$; this description of
$\dmge$ will be sufficient for us till \S\ref{rdg}.

2. In fact, $RC$ could be described in  terms of  so-called
Suslin complexes (see loc. cit.). We will not need this below. Instead, we will just note that $RC$ sends $D^-(\ssc)^{t\le 0}$ to $\dme^{t\le 0}$.
\end{rema}

\subsection{Some properties of Tate twists}\label{sttw}

Tate twisting in $\dme\supset \dmge$ is given by tensoring by the
object $\z(1)$ (it is
often denoted just by $-(1)$). Tate twist  has several descriptions and
nice properties. We will only need a few of
them; our main source is \S3.2  of \cite{1}; a
more detailed exposition could be found in \cite{vbook}
(see \S\S8--9). 

In order to calculate the tensor product of $X,Y\in \obj \dme$ one
should take any preimages $X',Y'$ of $X,Y$ in $\obj\dmk$ with
respect to $RC$ (for example, one could take $X'=X$, $Y'=Y$); next
one should resolve $X,Y$ by direct sums of $L(Z_i)$ for $Z_i\in\sv$;
lastly one should tensor these resolutions using the identity
$L(Z)\otimes L(T)=L(Z\times T)$ for $Z,T\in\sv$, and apply $RC$ to
the result. This tensor product is compatible with the natural
tensor product for $K^b(\smc)$.

We note that any object $\dmk^{t\le 0}$ has a resolution
concentrated in negative degrees (the canonical resolution of the
beginning of \S3.2 of \cite{1}). It follows that $\dme^{t\le
0}\otimes \dme^{t\le 0}\subset \dme^{t\le 0}$ (see Remark \ref{rdmge}(2); in fact, there is
an equality since $\z\in\obj HI$).

Next, we denote $\af^1\setminus \ns$ by $G_m$. The morphisms $\pt\to
G_m\to \pt$ (the point is mapped to $1$ in $G_m$) induce a splitting
$\mg(G_m)=\z\oplus \z(1)[1]$ for a certain ({\it Tate}) motif
$\z(1)$; see Definition 3.1 of \cite{vbook}.
For $X\in \obj \dme$ we denote $X\otimes\z(1)$ by $X(1)$.

One could also present
$\z(1)$ as $\co(\pt\to G_m)[-1]$; hence the Tate twist
functor $X\mapsto X(1)$ is compatible
with the functor  $-\otimes (\co(\pt\to G_m)[-1])$  on $C^b(\smc)$
via $\mg$.
We also obtain that
$\dme^{t\le 0}(1)\subset \dme^{t\le 1}$.

Now we define certain twists for functors.

\begin{defi}\label{dtw}
For an $G\in \adfu (\dmge,\ab)$, $n\ge 0$,  we define
$G_{-n}(X)=G(X(n)[n])$.
\end{defi}

Note that this definition is compatible with those of  \S3.1 of
\cite{3}. 
Indeed, for  $X\in \sv$ we have $G_{-1}(\mg(X))=G(\mg(X\times
G_m))/G(\mg(X))=\ke (G(\mg(X\times G_m))\to G(\mg(X)))$ (with respect
to natural morphisms $X\times\pt \to X\times G_m\to X\times \pt$);
$G_{-n}$ for larger $n$ could be defined by iterating $-_{-1}$.

Below we will extend this definition to (co)motives of pro-schemes.

For $F\in \obj\dme$ we will denote by $F_*$ the functor
$X\mapsto \dme(X,F):\dmge\to \ab$.

\begin{pr} \label{padj}

Let $X\in \sv$,  $n\ge 0$, $i\in\z$.

 1. For
any $F\in\obj\dme$ we have:
$F_{*-n}(\mg(X)[-i])$ is a retract of $\hhi(F)(X\times
G_m^{\times n})$ (which can be described explicitly).

2. There exists a $t$-exact functor $T_n:\dme\to\dme$ such that for
any $F\in\obj\dme$ we have
   $F_{*-n}\cong (T_n(F))_*$.
\end{pr}
\begin{proof}

1.
Proposition \ref{pvo} along with
our
description of $\z(1)$ yields the result.

2. For $F$ represented by a complex of $F^i\in \obj \ssc$ ($i\in
\z$) we define $T_n(F)$ as the complex of $T_n(F^i)$, where
$T_n:\psc\to \psc$ is defined similarly to $-_{-n}$ in Definition
\ref{dtw}. $T_n(F^i)$ are sheaves since $T_n(F_i)(X),\ X\in \sv$, is
a functorial retract of $F_i(X\times G_m^n)$.

In order to check that we actually obtain a well-defined a $t$-exact functor this way, it suffices to note that 
the restriction of $T_n$ to $\ssc$ is an exact functor by Proposition 3.4.3 of \cite{deggenmot}. 

Now, it suffices to check that $T_n$ defined satisfies the assertion
for $n=1$. In this case the statement follows easily from
Proposition 4.34 of \cite{3} (note that it is not important whether
we consider Zariski or Nisnevich topology by Theorem 5.7 of
ibid.).

\end{proof}

\subsection{Pro-motives vs. comotives; the description
of our strategy}\label{sprom}

Below we will embed $\dmge$ into a certain triangulated
category $\gd$ of {\it comotives}. Its construction
(and computations in it) is rather
 complicated; in fact, the author is not sure whether the main properties
 of $\gd$ (described below) specify it up to an isomorphism. So,
 before working with co-motives we will (following F. Deglise)
 describe a simpler category of {\it pro-motives}. The latter is
 not needed for our main results
(so the reader may skip this subsection);
yet the comparison of the
 categories mentioned would clarify the nature of our methods.

Following \S3.1 of \cite{deggenmot}, we define the category $\gdn$
as the additive  category of naive i.e. formal
(filtered)  pro-objects of $\dmge$.
This means that for any $X:L\to \dmge$, $Y:J\to \dmge$ we define
\begin{equation}\label{emorgdn}
\gdn (\prli_{l\in L}X_l,\prli_{j\in J} Y_j)
= \prli_{j\in J} (\inli_{l\in L} \dmge(X_l,Y_j)).\end{equation}

The main disadvantage of $\gdn$ is that it is not triangulated.
Still, one has the obvious shift for it; following Deglise, one can
define pro-distinguished triangles as (filtered) inverse limits of
distinguished triangles in $\dmge$.  This allows to construct a certain
motivic coniveau exact couple for a motif of a smooth variety
in \S4.2 of \cite{de2}
(see also \S5.3 of \cite{deggenmot}).
This construction is parallel to the
 classical construction of  coniveau spectral sequences (see \S1 of
 \cite{suger}).  One starts with certain 'geometric'
 Postnikov towers in $\dmge$ (Deglise calls them {\it triangulated
 exact couples}). For
 $Z\in \sv$ we consider filtrations
 $\varnothing=Z_{d+1}\subset Z_d\subset Z_{d-1}\subset
 \dots \subset Z_0=Z$;  $Z_{i}$ is everywhere
of codimension  $\ge i$ in $Z$ for
 all $i$.
  Then we have a system of distinguished
 triangles relating $\mg(Z\setminus Z_i)$ and
 $\mg(Z\setminus Z_i\to Z\setminus
 Z_{i+1})$; this yields a Postnikov tower.
  Then
 one passes to the inverse limit of these towers in $\gdn$
 (here the connecting morphisms
 are induced by the corresponding open embeddings). Lastly, the
 functorial form of the Gysin distinguished triangle for motives
 allows Deglise to identify
$X_i=\prli(\mg(Z\setminus Z_i\to Z\setminus
 Z_{i+1}))$ with the product of shifted Tate twists of
 pro-motives of all points of $Z$ of
 codimension $i$. Using the results of see \S5.2 of \cite{deggenmot}
(the relation of pro-motives with
 cycle modules of M. Rost, see \cite{rostc}) one can also
  compute the morphisms that
 connect $X^i$ with $X^{i+1}$.

 Next, for any cohomological
 $H:\dmge\to \au$, where $\au$ is an abelian category satisfying
 AB5, one can extend $H$ to $\gdn$ via the corresponding direct
 limits. Applying $H$ to the motivic coniveau exact couple one
 gets the classical coniveau spectral sequence (that converges
to the $H$-cohomology of
 $Z$). 
 This allows to extend the seminal results of \S6 of \cite{blog}
  to a comprehensive description of the coniveau spectral sequence
  in the case when $H$ is represented by  $Y\in \obj\dme$
  (in  terms of the homotopy $t$-truncations of $Y$; see
  Theorem 6.4 of \cite{ndegl}).

 Now suppose that one wants to apply a similar procedure for an
 arbitrary $X\in \obj \dmge$; say, $X=\mg(Z^1\stackrel{f}{\to} Z^2)$ for
 $Z^1,Z^2\in \sv$, $f\in \smc(Z^1,Z^2)$. One would expect that the
 desired exact couple for $X$  could be constructed from those for
 $Z^j$, $j=1,2$. This is indeed the case when
 $f$ satisfies certain codimension restrictions; cf. \S7.4 of
 \cite{bws}. Yet for a general $f$ it seems to be quite difficult to
 relate the filtrations of distinct $Z^j$ (by the corresponding
 $Z^j_i$).
On the other hand, the formalism of weight structures and
 weight spectral
 sequences (developed in \cite{bws}) allows to
 'glue' certain {\it weight} Postnikov
 towers for objects of a triangulated categories equipped with
a weight structure; see  Remark \ref{rcger}(3) below.

So, we construct a certain triangulated category
 $\gd$
 that is somewhat similar to $\gdn$.
 Certainly, we want distinguished triangles in $\gd$ to be
 compatible with inverse limits that come from 'geometry'. A
 well-known recipe for this is: one should consider some
  category $\gdp$ where
 (certain) cones of morphisms are functorial and pass to (inverse)
 limits in $\gdp$; $\gd$ should be a localization of $\gdp$.
 In fact,
$\gdp$ constructed in \S\ref{scgd} below could be endowed with  a
certain (Quillen) model structure
such that $\gd$ is its homotopy category. We will never use this
fact below; yet we will sometimes call inverse
 limits coming from $\gdp$
 homotopy limits (in $\gd$).

Now, in
Proposition \ref{cdscoh} below we  will prove that cohomological
functors $H:\dmge\to \au$ 
could be
extended to $\gd$
   in a  way that is compatible with homotopy limits
(those coming from  $\gdp$).
So one may say that objects of $\gd$ have the same
cohomology as those of $\gdn$.
On the other hand, we have to pay the price for $\gd$
being triangulated:
(\ref{emorgdn}) does not compute morphisms between
homotopy limits in $\gd$.
The 'difference' could be described in terms of certain
higher projective limits
(of the corresponding morphism groups in $\dmge$).

Unfortunately, the author does not know how to control the
corresponding
 $\prli^2$ (and higher ones) in the general case; this does
not allow to construct a weight structure  on a sufficiently large
triangulated subcategory of $\gd$ if $k$ is uncountable (yet see
\S\ref{unck}, especially the last paragraph of it). In the case of a
countable $k$ only $\prli^1$ is non-zero.
 In this case the morphisms between homotopy limits in $\gd$
are expressed
 by the formula (\ref{lmor}) below. This allows to prove that
there are no
morphisms of positive degrees between certain Tate twists of
the comotives of
function fields (over $k$). This immediately yields that one
can construct
a certain weight structure on the triangulated subcategory
$\gds$ of $\gd$
 generated by products of Tate twists of the comotives of
function fields
(in fact, we also idempotent complete $\gds$). Now, in
order to prove that
 $\gds$ contains $\dmge$ it suffices to prove that the
motif of any smooth
variety $X$ belongs to $\gds$. To this end it clearly suffices to
decompose $\mg(X)$ into a Postnikov tower whose factors
are products
 of Tate twists of the comotives of function fields. So, we lift
 the motivic
coniveau exact couple (constructed in \cite{de2}) from $\gdn$ to $\gd$.
Since cones in $\gdp$ are compatible with inverse limits, we
can construct
a tower whose terms are the homotopy limits of the
corresponding terms
of the geometric towers mentioned. In fact, this could
be done for an
uncountable $k$ also; the difficulty is to identify the
analogues of $X_i$
in $\gd$. If $k$ is countable, the homotopy limits corresponding to our
tower are countable also. Hence (by an easy well-known result)
the isomorphism
classes of these homotopy limits could be computed in terms
 of the corresponding
objects and morphisms in $\dmge$. This means: it suffices
to compute $X^i$
in $\gdn$ (as was done in \cite{de2}); this yields the result needed.
Note that we cannot (completely) compute the $\gd$-morphisms $X^i\to X^{i+1}$; yet we know how they act on cohomology.

The most interesting application of the results described is
the following one.
 We prove that there are no positive $\gd$-morphisms between
(certain) Tate twists
of the comotives of smooth semi-local schemes
(or {\it primitive schemes}, see below);
this generalizes the corresponding result for function fields.
 It follows that
these twists belong to the {\it heart} of the weight structure
 on $\gds$ mentioned.
 Therefore the comotives of (connected) primitive schemes are retracts of
the comotives of
their generic points. Hence the same is true for the
 cohomology of the comotives
mentioned and also for the corresponding pro-motives.
Also, the comotif of a function
field contains as retracts  the twisted comotives of its
 residue fields (for all geometric valuations); this
also implies the corresponding results for cohomology and pro-motives.

\begin{rema}\label{rdeg1}
 In fact, Deglise mostly considers pro-objects for
 Voevodsky's $\dmgm$ and of $\dme$; yet the distinctions
are not important since
 the full embeddings $\dmge\to \dmgm$ and $\dmge\to \dme$
obviously extend to  full
 embedding of the corresponding categories of pro-objects.
 Still, the embeddings mentioned allow Deglise to extend several
nice results for Voevodsky's motives to pro-motives.

 2. One of the advantages of the results of Deglise is
that he never requires
  $k$ to be countable. Besides, our construction of 
weight Postnikov towers
mentioned heavily relies on the functoriality of the
Gysin distinguished triangle
for motives (proved in \cite{de2}; see also Proposition 2.4.5 of \cite{deggenmot}).

\end{rema}

\section{Weight 
structures: reminder, truncations,  weight spectral sequences, and
duality with $t$-structures}\label{swnew}

In \S\ref{srws} we recall  basic definitions  of
 the theory of weight structures (it was developed in
\cite{bws}; the concept was also
 independently introduced in \cite{konk}).
 Note here that weight structures (usually denoted by  $w$)
 are natural counterparts of  $t$-structures. 
 Weight structures yield
 weight truncations; those (vastly) generalize stupid truncations
  in $K(B)$: in particular, they are not canonical,
  yet any morphism of objects could be extended (non-canonically)
  to a morphism of their weight truncations. 
 We recall several properties
of weight structures in \S\ref{sbpws}.

We recall {\it virtual $t$-truncations} for a (cohomological)
 functor  $H:\cu\to \au$ (for $\cu$ endowed with a weight structure)
in \S\ref{svirtr} (these truncations are defined in terms of  weight
truncations). Virtual $t$-truncations were introduced in \S2.5 of
\cite{bws}; they yield a way to present $H$ (canonically) as an extension of  a cohomological functor that is positive in a certain sense by a
'negative' one  (as if $H$ belonged to some
triangulated category of functors
 $\cu\to \au$ endowed with a $t$-structure).
 We study this notion further here, and prove that virtual
 $t$-truncations for a cohomological $H$ could be  characterized
   up to a unique isomorphism by their properties
   (see    Theorem \ref{trfun}(III4)). In order to give some
    characterization also for the 'dimension shift' (connecting
    the positive and the negative virtual $t$-truncations of
    $H$), we introduce the notion of  {\it nice (strongly exact)}
complex of functors. We prove that complexes of representable functors
coming from distinguished triangles in $\cu$ are nice, as well
as those complexes that could be obtained from  nice strongly exact
complexes of functors $\cupr\to \au$ for some small
triangulated $\cupr\subset \cu$ (via the extension procedure
given by Proposition \ref{pextc}).

In \S\ref{scovi} we consider weight spectral sequences
(introduced in \S\S2.3--2.4 of \cite{bws}). We prove that the derived
 exact couple for the weight spectral sequence $T(H)$ (for $H:\cu\to\au$)
  could be naturally described in terms of virtual $t$-truncations
  of $H$. So, one can express $T(H)$ starting from $E_2$
  (as well as the corresponding filtration of $H^*$) in these
  terms also. 
   This is an important result,
   since the basic definition of $T(H)$  is given in terms of
    {\it weight Postnikov towers} for objects of $\cu$, whereas
   the latter are not canonical. 
   In particular, this result  yields canonical functorial spectral
    sequences in classical situations (considered by Deligne;
    cf. Remark 2.4.3 of \cite{bws}; note that we do not need rational
    coefficients here).

In \S\ref{sdortstr} we introduce the definition
a {\it (nice) duality} $\Phi:\cu^{op}\times \du\to \au$, and
of (left) {\it orthogonal} weight and
 $t$-structures (with respect to $\Phi$). The latter definition
 generalizes the notion of
{\it adjacent} structures introduced in \S4.4 of \cite{bws}
(this is the case $\cu=\du$, $\au=\ab$, $\Phi=\cu(-,)$). If $w$ is
  orthogonal to $t$ then the virtual $t$-truncations
  (corresponding to $w$)  of functors of the type
  $\Phi(-,Y),\ Y\in \obj\du$, are exactly the functors
  'represented via $\Phi$' by the actual $t$-truncations of
   $Y$ (corresponding to $t$). We also prove that (nice) dualities
    could be extended from $\cu'$ to $\cu$ (using Proposition
    \ref{pextc}).
Note here that (to the knowledge of the author) this paper is the
first one which considers 'pairings' of triangulated categories.

In \S\ref{sconin} we prove: if $w$ and $t$ are orthogonal with
respect to a nice duality, the weight spectral sequence converging
to $\Phi(X,Y)$ (for $X\in \obj \cu,\ Y\in \obj \du$) is naturally
isomorphic (starting from $E_2$) to the one coming from
$t$-truncations of $Y$. Moreover even when the duality is not nice,
 all $E_r^{pq}$ for $r\ge 2$ and the filtrations corresponding to
 these spectral sequences
are still canonically isomorphic. Here niceness of a duality
(defined in \S\ref{sdortstr}) is a somewhat technical condition
(defined in terms of nice complexes of functors). Niceness
 generalizes to pairings ($\cu\times \du\to \au$) the axiom
TR3 (of triangulated categories: any commutative square in
 $\cu$ could be completed to a morphism of distinguished  triangles;
    note that this axiom could be described in
    terms of the functor $\cu(-,-): \cu\times \cu\to \ab$).
  We also discuss some
alternatives and predecessors of  our methods and results.

In \S\ref{schws}
 we compare weight decompositions, virtual $t$-truncations, and
weight spectral sequences corresponding to distinct weight
structures (in possibly distinct triangulated categories, connected by an exact functor).


\subsection{Weight structures: basic definitions}\label{srws}

We recall the definition of a weight structure (see \cite{bws};
in \cite{konk}
D. Pauksztello
introduced weight structures independently and called them
co-t-structures).

\begin{defi}[Definition of a weight structure]\label{dwstr}

A pair of subclasses $\cu^{w\le 0},\cu^{w\ge 0}\subset\obj \cu$ for
a triangulated category $\cu$ will be said to define a weight
structure $w$ for $\cu$ if 
they  satisfy the following conditions:

(i) $\cu^{w\ge 0},\cu^{w\le 0}$ are additive and Karoubi-closed
(i.e. contain all retracts of their objects that belong to
$\obj\cu$).

(ii) {\bf Semi-invariance with respect to translations.}

$\cu^{w\ge 0}\subset \cu^{w\ge 0}[1]$; $\cu^{w\le 0}[1]\subset
\cu^{w\le 0}$.

(iii) {\bf Orthogonality.}

$\cu^{w\ge 0}\perp \cu^{w\le 0}[1]$.

(iv) {\bf Weight decomposition}.

 For any $X\in\obj \cu$ there
exists a distinguished triangle

\begin{equation}\label{wd}
B[-1]\to X\to A\stackrel{f}{\to} B
\end{equation} such that $A\in \cu^{w\le 0}, B\in \cu^{w\ge 0}$.

\end{defi}

A simple  example of a category with a weight structure is $K(B)$
 for  any additive $B$: positive objects are complexes
 that are homotopy equivalent to those concentrated in positive degrees;
 negative
objects are complexes
that are homotopy equivalent to those concentrated in negative degrees.
 Here one could also consider  the subcategories of  complexes that
are bounded from above, below, or from both sides.

The triangle (\ref{wd}) will be called a {\it weight decomposition}
of $X$.  A weight decomposition is (almost) never canonical;
still we will sometimes denote any pair $(A,B)$ as in (\ref{wd}) by $X^{w\le 0}$
and $X^{w\ge 1}$. Besides, we will call objects of
 the type $(X[i])^{w\le 0}[j]$ and
$(X[i])^{w\ge 0}[j]$ (for $i,j\in\z$) {\it weight truncations} of $X$.
 A shift 
 of the distinguished triangle (\ref{wd}) by $[i]$ for any $i\in\z,\ X\in \obj \cu$ (as well as any its rotation) will sometimes be called a {\it shifted weight decomposition}.

In $K(B)$ (shifted) weight decompositions come from stupid
truncations of complexes.

We will also need the following definitions and notation.

\begin{defi}\label{d2} 
Let $X\in \obj \cu$.

\begin{enumerate}
    \item
 The category $\hw\subset \cu$ whose objects are
$\cu^{w=0}=\cu^{w\ge 0}\cap \cu^{w\le 0}$, $\hw(Z,T)=\cu(Z,T)$ for
$Z,T\in \cu^{w=0}$,
 will be called the {\it heart} of the weight structure
$w$.

    \item  $\cu^{w\ge l}$ (resp. $\cu^{w\le l}$, resp.
$\cu^{w= l}$) will denote $\cu^{w\ge
0}[-l]$ (resp. $\cu^{w\le 0}[-l]$, resp. $\cu^{w= 0}[-l]$).

\item We denote $\cu^{w\ge l}\cap \cu^{w\le i}$
by $\cu^{[l,i]}$. 

\item  $X^{w\le l}$ (resp. $X^{w\ge l}$) will denote $(X[l])^{w\le
0}$ (resp. $(X[l-1])^{w\ge 1}$).

\item $w_{\le i}X$ (resp. $w_{\ge i}X$) will denote $X^{w\le i}[-i]$
(resp.  $X^{w\ge i}[-i]$).

    \item  $w$ will be called {\it non-degenerate} if $$\cap_l \cu^{w\ge
l}=\cap_l \cu^{w\le l}=\ns.$$

    \item  We consider $\cu^b=(\cup_{i\in \z} \cu^{w\le i})\cap(\cup_{i\in \z} \cu^{w\ge
i})$ and call it  the class of {\it bounded} objects
of $\cu$.

For $X\in \cu^b$ we will usually take $w_{\le i}X=0$ for $i$ small
enough, $w_{\ge i}X=0$ for $i$ large enough.

We will also denote by $\cu^b$ the corresponding
 full subcategory of $\cu$.

\item We will say that $(\cu,w)$ is bounded if $\cu^b= \cu$.

\item \label{idwpost} We will call a Postnikov tower for $X$
(see Definition \ref{dpoto})
 a {\it weight Postnikov tower} if
all $Y_i$ are some choices for $w_{\ge 1-i}X$.
In this case we will call the complex whose terms are $X^p$
(see Remark \ref{rwcomp})
a {\it weight complex} for $X$.

We will call a weight Postnikov tower for $X$ {\it negative} if $X\in \cu^{w\le 0}$ and we choose
$w_{\ge j}X$ to be $0$ for all $j>0$ here.

    \item \label{idextst} $D\subset \obj \cu$ will be
called extension-stable
    if for any distinguished triangle $A\to B\to C$
in $\cu$ we have: $A,C\in D\implies
B\in D$.

We will also say that the corresponding full subcategory
is extension-stable.

    \item $D\subset \obj \cu$ will be called {\it negative}
if for any $i>0$ we have $D\perp D[i]$.

\end{enumerate}

\end{defi}

\begin{rema}\label{rpost}


1. One could also dualize our definition of a weight Postnikov tower
 i.e.  build a tower from $w_{\le l}X$ instead of $w_{\ge l}X$.
  Our definition of a weight Postnikov tower is more convenient for our
purposes since
in \S\ref{scgersten} below we will consider $Y_i=j(Z_0\setminus
               Z_i)$ instead of $=j(Z_0\setminus
               Z_i\to Z_0)[-1]$.
 Yet this does not make much difference; see \S1.5 of \cite{bws} and Theorem \ref{tbw}(\ref{ioct}) below.
In particular, our definition of the weight complex for $X$ coincides
with Definition 2.2.1 of ibid.  Note also, that
Definition 1.5.8 of ibid (of a weight Postnikov tower)
contained both 'our' part of the data and the dual part.

2. Weight Postnikov towers for objects of $\cu$ are far from
 being unique;
 their morphisms
(provided by   Theorem \ref{tbw}(\ref{impost}) below)
are not unique also
(cf. Remark 1.5.9 of \cite{bws}).
Yet the corresponding weight spectral sequences
 for cohomology are unique and functorial starting from $E_2$;
see Theorem 2.4.2 of
  ibid. and Theorem \ref{pssw} below for more detail.
In particular, all possible choices of a weight complex for $X$ are homotopy equivalent (see  Theorem 3.2.2(II) and  Remark 3.1.7(3) in  \cite{bws}).

\end{rema}

\subsection{Basic properties of weight structures}\label{sbpws}

Now we list some basic properties of notions defined. In the
theorem below we  will assume that $\cu$ is endowed with a fixed
weight structure $w$ everywhere except in assertions
\ref{igen} -- \ref{iwgene}.

\begin{theo} \label{tbw}

\begin{enumerate}

\item \label{idual}
The axiomatics of weight structures is self-dual: if $\du=\cu^{op}$
(so $\obj\cu=\obj\du$) then one can define the (opposite)  weight
structure $w'$ on $\du$ by taking $\du^{w'\le 0}=\cu^{w\ge 0}$ and
$\du^{w'\ge 0}=\cu^{w\le 0}$.

\item We have \label{iort} \begin{equation}\label{ewort}
\cu^{w\le 0}=\cu^{w\ge 1}{}^{\perp}\end{equation}
and \begin{equation}\label{ewort2}\cu^{w\ge 0}
={}^\perp\cu^{w\le -1}.\end{equation}

\item For any $i\in\z$, $X\in \obj \cu$ we have a distinguished triangle
 $w_{\ge i+1}X\to X\to w_{\le i}X$ (given by a shifted weight decomposition).

    \item\label{iext} $\cu^{w\le 0}$, $\cu^{w\ge 0}$, and $\cu^{w=0}$
are extension-stable.

    \item  All  $\cu^{w\le i}$ are closed with respect to
arbitrary (small)
 products (those, which exist in $\cu$); all $\cu^{w\ge i}$ and
$\cu^{w=i}$ are additive.

    \item For any weight decomposition of $X\in \cu^{w\ge 0}$ (see
(\ref{wd})) we have $A\in \cu^{w=0}$. 

    \item \label{isum}
    If $A\to B\to C\to A[1]$ is a distinguished triangle and
$A,C\in \cu^{w= 0}$, then $B\cong A\oplus C$.


\item \label{isump}
If we have a distinguished triangle $A\to B\to C$ for $B\in \cu^{w=0}$, $C\in \cu^{w\le -1}$, then $A\cong B\bigoplus C[-1]$.

    \item \label{idec0}
If $X\in \cu^{w=0}$, $X[-1]\to A\stackrel{f}{\to} B$ is a
    weight decomposition (of $X[-1]$), then $B\in \cu^{w=0}$;
 $B\cong A\oplus X$.

        \item\label{icompl} Let $l\le m\in \z$, $X,X'\in \obj \cu$; let
 weight decompositions
        of $X[m]$ and $X'[l]$ be fixed. Then  any morphism
$g:X\to X'$ can be
        completed to a morphism of  distinguished triangles
\begin{equation}\begin{CD} \label{d2x3}
w_{\ge m+1}X@>{}>>X @>{c}>> w_{\le m}X \\
@VV{a}V@VV{g}V @VV{b}V\\
w_{\ge l+1}X'@>{}>>X' @>{d}>> w_{\le l}X'
\end{CD}\end{equation}
This completion is unique if $l<m$.

\item\label{idowe}
Consider some completion of a commutative triangle $w_{\ge m+1}X\to w_{\ge l+1}X\to X$ (that is uniquely determined by the morphisms $w_{\ge m+1}X\to X$ and $w_{\ge l+1}X\to X$ coming from the corresponding shifted weight decompositions; see the previous assertion)  to an octahedral diagram:
$$\displaylines{
\label{doct} \xymatrix{ w_{\le l}X \ar[rd]^{[1]}\ar[dd]^{[1]} & & X \ar[ll] \\ & w_{\ge l+1}X \ar[ru]\ar[ld] & \\ w_{[l+1,m]}X\ar[rr]^{[1]} & & w_{\ge m+1}X\ar[lu]\ar[uu] }\cr
\xymatrix{ w_{\le l}X  \ar[dd]^{[1]} & & X \ar[ld] \ar[ll] \\ & w_{\le
    m}X \ar[lu] \ar[rd]^{[1]}   & \\ w_{[l+1,m]}X \ar[ru]
  \ar[rr]^{[1]} & & w_{\ge m+1}X\ar[uu] } 
}$$

 Then $w_{[l+1,m]}X\in \cu^{[l+1,m]}$; all the  distinguished triangles of this octahedron are  shifted  weight decompositions.

\item\label{ioct} For $X,X'\in \obj \cu$, $l,l',m,m'\in \z$, $l<m$, $l'<m'$, $l>l'$, $m>m'$,  consider two octahedral diagrams: (\ref{doct}) and a similar one corresponding to the commutative triangle  $w_{\ge m+1}X\to w_{\ge l+1}X\to X$ and $w_{\ge m'+1}X'\to w_{\ge l'+1}X\to X$ (i.e. we fix some choices of these diagrams). Then any $g\in \cu(X,X')$ could be uniquely 
 extended to a morphism of these diagrams. The corresponding morphism $h: w_{[l+1,m]}X\to w_{[l'+1,m']}X'$ is characterized uniquely by any of the following conditions: 
 
 (i) there exists a $\cu$-morphism $i$ that makes the squares  
\begin{equation}\label{ecd1}
\begin{CD}
w_{\ge l+1}X@>{}>>X\\
@VV{i}V@VV{g}V \\
w_{\ge l'+1}X'@>{}>>X'
\end{CD}\end{equation}
and 
\begin{equation}\label{ecd2}
\begin{CD}
w_{\ge l+1}X@>{}>>w_{[l+1,m]}X\\
@VV{i}V@VV{h}V \\
w_{\ge l'+1}X'@>{}>>w_{[l'+1,m']}X'
\end{CD}\end{equation}
commutative.

 (ii) there exists a $\cu$-morphism $j$ that makes the squares  
\begin{equation}\label{ecd3}
\begin{CD}
X @>{}>>w_{\le m}X\\
@VV{g}V@VV{j}V \\
X'@>{}>>w_{\le m'}X'
\end{CD}\end{equation}
and 
\begin{equation}\label{ecd4}
\begin{CD}
w_{[l+1,m]}X@>{}>>w_{\le m}X\\
@VV{h}V@VV{j}V \\
w_{[l'+1,m']}X'@>{}>>w_{\le m'}X'
\end{CD}\end{equation}
commutative.

    \item\label{iwpost} For any choice of $w_{\ge i}X$
    there exists  a weight Postnikov tower for $X$
(see  Definition \ref{d2}(\ref{idwpost})). For any weight Postnikov
tower we have
     $\co(Y_i\to X)\in \cu^{w\le -i}$;
    $X^i\in \cu^{w=0}$.

\item\label{iwpostc}
    Conversely, any bounded Postnikov tower (for $X$)
with $X^i\in \cu^{w=0}$ is a weight Postnikov
    tower for it.

    \item\label{impost} For $X,X'\in \obj \cu$ and arbitrary weight
    Postnikov towers for them, any $g\in \cu(X,X')$
can be extended to a morphism
     of Postnikov towers (i.e. there exist morphisms
  $Y_i\to Y'_i,\ X^i\to X'^i$,
     such that the corresponding squares commute).

\item\label{impostp} 
 For $X,X'\in \cu^{w\le 0}$, suppose that
 $f\in\cu(X,X')$ can be extended to a morphism of
 (some of) their negative Postnikov towers  that establishes an isomorphism
$X^0\to X'{}^0$. Suppose also that $X'\in \cu^{w=0}$.
Then $f$ yields a projection of $X$ onto $X'$ (i.e. $X'$ is a retract of $X$ via $f$).

    \item
    $\cu^b$ is a
 Karoubi-closed triangulated subcategory of $\cu$. $w$ induces
a non-degenerate  weight structure for it, whose
heart equals $\hw$. 

    \item \label{igen}

For a triangulated idempotent complete $\cu$ let  $D\subset \obj \cu$ be negative. Then
there exists a unique weight structure $w$ on the Karoubi-closure $T$
of $\lan D\ra$ in $\cu$ such that $D\subset T^{w=0}$. Its heart is
the Karoubi-closure of the closure of $D$ in $\cu$ with respect to
(finite) direct sums.

\item \label{iwgen} For the weight structure mentioned in the
    previous assertion, $T^{w\le 0}$ is the 
    smallest Karoubi-closed extension-stable subclass of $\obj\cu$
containing $\cup_{i\ge 0}D[i]$; $T^{w\ge 0}$ is the smallest Karoubi-closed extension-stable subclass of $\obj\cu$
containing $\cup_{i\le 0}D[i]$.

     \item \label{iwgene} For the weight structure mentioned in two
    previous assertions we also have $$T^{w\le 0}
=(\cup_{i<0}D[i])^{\perp};\ 
T^{w\ge 0}={}^\perp(\cup_{i>0}D[i]).$$

\end{enumerate}

\end{theo}

\begin{proof}
\begin{enumerate}

    \item
    Obvious; cf.  Remark \ref{rts} of \cite{bws}
(and Remark 1.1.2 of ibid.
for more detail).

\item These are parts 1 and 2 of Proposition 1.3.3 of ibid.

    \item Obvious (since $[i]$ is exact up to change of signs of morphisms); cf. Remark 1.2.2 of ibid.

    \item This is part 3 of Proposition 1.3.3 of ibid.

    \item Obvious  from the definition and parts 4 of loc. cit.

    \item This is part 6 of Proposition 1.3.3 of ibid.

    \item This is part 7  of loc. cit.

\item It suffices to note that $\cu(B,C)=\ns$, hence the triangle splits.

    \item This is part 8 of loc. cit.

    \item This is Lemma 1.5.1 of ibid.

\item The only non-trivial statement here is that $w_{[l+1,m]}X\in \cu^{[l+1,m]}$ (it easily implies: the left hand side of the lower cap in (\ref{doct}) also yields a shifted weight decomposition).
(\ref{doct}) yields distinguished triangles: $T_1=(w_{\ge l+1}X\to w_{[l+1,m]}X\to w_{\ge m+1}X[1])$ and $T_2=(w_{\le l}X\to w_{[l+1,m]}X[1]\to w_{\le m}X[1])$. Hence assertion \ref{iext} yields the result.


\item By assertion \ref{icompl}, $g$ extends uniquely to a morphism of the following distinguished triangles: from $T_3=(w_{\ge m+1}X\to X\to w_{\le m}X)$ to $T_3'=(w_{\ge m'+1}X'\to X'\to w_{\le m'}X)$, and from $T_4=(w_{\ge l+1}X\to X\to w_{\le l}X)$ to $T_4'=(w_{\ge l'+1}X'\to X'\to w_{\le l'}X)$; next we also obtain a unique morphism from $T_1$ (as defined in the proof of the previous assertion) to its analogue $T'_1$.   Putting all of this together: we obtain unique morphisms of all of the vertices of our octahedra, which are compatible with all the edges of the octahedra except (possibly) 
those that belong to $T_2$ (as defined above). 
We also obtain that there exists unique $i$ and $h$ that complete (\ref{ecd1}) and (\ref{ecd2}) to commutative squares.

Now, the morphism $w_{\le l}X\to w_{[l+1,m]}X$ could be decomposed into the composition of morphisms belonging to $T_1$ and $T_3$. 
Hence in order to verify that we have actually constructed a morphism of octahedral diagrams, it remains to verify the commutativity of the squares 
\begin{equation}\label{ecd5}
\begin{CD}
w_{\le m}X @>{}>>w_{\le l}X\\
@VV{g}V@VV{j}V \\
w_{\le m'}X'@>{}>>w_{\le l'}X'
\end{CD}\end{equation}
and (\ref{ecd4})
i.e. we should check that the two possible compositions of arrows for each of the squares are equal. Now,  assertion \ref{icompl}
implies: the compositions in question for  (\ref{ecd5}) both equal the only morphism $q$
that makes the square $$\begin{CD}X @>{}>>w_{\le m}X\\
@VV{g}V@VV{q}V \\
X'@>{}>>w_{\le l'}X'
\end{CD}$$ commutative.  
Similarly, the compositions  for (\ref{ecd4}) both equal the only morphism $r$
that makes the square $$\begin{CD}
w_{\ge l+1}X@>{}>>w_{[l+1,m]}X\\
@VV{
}V@VV{r}V \\X' @>{}>>w_{\le m'}X'
\end{CD}$$ commutative. Here we use the part of the octahedral axiom that says that the square $$\begin{CD}
w_{\ge l+1}X@>{}>>w_{[l+1,m]}X\\
@VV{}V@VV{}V \\X @>{}>>w_{\le m}X
\end{CD}$$
is commutative (as well as the corresponding square for $(X',l',m')$).

Lastly, as we have already noted, the condition (i) characterizes $h$ uniquely; for similar (actually, exactly dual) reasons the same is true for (ii). Since the morphism $w_{[l+1,m]}X\to w_{[l'+1,m']}X'$ coming from the morphism of the octahedra constructed satisfies both of these conditions, it is characterized by any of them uniquely.



\item Immediate from part 2 of (Proposition 1.5.6) of loc. cit. (and also from assertion \ref{idowe}).

\item Immediate  from 
 Remark 1.5.9(2) of ibid.

    \item Immediate from part 1 (of Remark 1.5.9)
of loc. cit.

    \item It suffices to prove that $\co f\in \cu^{w\le -1}$.
    Indeed, then the distinguished triangle $X\stackrel{f}{\to} X'\to \co f$ necessarily splits.
    
    We complete the commutative triangle $X^{w\le -1}\to X'^{w\le -1}\to X^0(=X'{}^0)$ to an octahedral diagram. Then we obtain $\co f\cong \co (X^{w\le -1}\to X'^{w\le -1})[1]$; hence  
  $\co f\in \cu^{w\le -1}$ indeed.

\item This is Proposition 1.3.6 of ibid.

    \item By    Theorem 4.3.2(II1) of 
    ibid., there exists a unique weight structure on $\lan D\ra$
such that
    $D\subset \lan D\ra^{w=0}$. Next, Proposition 5.2.2 of ibid. yields
that $w$ can be extended to the whole $T$; along with part
Theorem 4.3.2(II2) of loc. cit. it also allows to calculate $T^{w=0}$
 in this case.

\item Immediate from Proposition 5.2.2 of ibid. and the description of $\lan H \ra ^{w\le 0}$ and $\lan H \ra ^{w\ge 0}$  in the proof of  Theorem 4.3.2(II1) of ibid.

\item
If $X\in T^{w\le 0}$ then the orthogonality condition for $w$
immediately yields: $Y\perp X$ for any $Y\in \cup_{i<0}D[i]$.

Conversely, suppose that for some $X\in \obj T$ we have
$Y\perp X$ for all $Y\in \cup_{i<0}D[i]$. Then $Y\perp X$
also for all $Y$ belonging to the
    smallest extension-stable subclass of $\obj\cu$ containing
$\cup_{i<0}D[i]$.
 Hence this is also true for all
    $Y\in T^{w\ge 1}$ (see the previous assertion). Hence
(\ref{ewort}) yields:
$X\in T^{w\le 0}$.
We obtain the first part of the assertion.

The 
second part of the assertion is dual to 
the first one (and easy from (\ref{ewort2})).

\end{enumerate}

\end{proof}

\begin{rema}\label{rfunct}


\begin{enumerate}

\item\label{itriv} 
In the notation of assertion \ref{icompl},   
for any  $a$ (resp. $b$) such that the left (resp. right) hand square in (\ref{d2x3}) commutes
there exists some $b$  (resp. some $a$)  that makes (\ref{d2x3}) a morphism of distinguished triangles (this is just axiom TR3 of triangulated categories). Hence for  $l<m$ the left (resp. right) hand side of (\ref{d2x3}) characterizes $a$ (resp. $b$) uniquely.

\item\label{i1}\label{i2}\label{i3}  Assertions \ref{icompl} and \ref{ioct} 
yield  mighty tools for proving that a
 construction described in terms of weight decompositions is
 functorial (in a certain sense).  
 In particular, the proofs of functoriality of weight filtration and virtual $t$-truncations 
for cohomology (we will consider these notions below) in \cite{bws} were based on assertion \ref{icompl}.

 Now we explain what kind of functoriality could be obtained using  assertion loc. cit.  Actually, such an argument was already used in the proof of assertion \ref{ioct}.

In the notation of assertion \ref{icompl} we will say that $a$ and $b$ are compatible with $g$ (with respect to the corresponding weight decompositions). 
 Now suppose that 
for some $X''\in \obj\cu$, some $n\le l$, $g'\in \cu(X',X'')$, and a distinguished triangle 
$w_{\ge n+1}X''\to X' \to w_{\le n}X'$ we have morphisms $a':  w_{\ge l+1}X'\to w_{\ge n+1}X''$ and $b':w_{\le l}X'\to w_{\le n}X''$ compatible with $g'$. Then $a'\circ a$ and $b'\circ b$ are compatible with $g'\circ g$ (with respect to the corresponding weight decompositions)! Moreover, if $n<m$ then $(a'\circ a,b'\circ b)$ is exactly the (unique!) pair of morphisms compatible with $g'\circ g$. 

\item \label{idofunct} 
In the notation of assertion \ref{ioct} we will (also) say that $h:w_{[l+1,m]}X\to w_{[l'+1,m']}X'$ is compatible with $g$. Note that $h$ is uniquely characterized by  (i) (or (ii)) of loc. cit.; hence 
in order to characterize it uniquely it suffices to fix $g$ and all the 
rows in (\ref{ecd1}) and (\ref{ecd2}) (or in (\ref{ecd3}) and (\ref{ecd4})). 
Besides, we obtain that $h$ is functorial in a certain sense (cf. the reasoning above).

\item\label{i5} \label{i4} Assertion \ref{idowe} immediately implies: for any $l< m$
the class of all possible $w_{\le l}X$ coincides with the class of possible $w_{\le l}(w_{\le m}X)$, whereas the class of possible
 $w_{\ge m}X$ coincides with those of $w_{\ge m}(w_{\ge l}X)$.

Besides, assertion \ref{idowe} also allows to construct weight Postnikov towers (cf. \S1.5 of  \cite{bws}). Hence $w_{[i,i]}X$ is just $X^i[-i]$
 (for any $i\in \z,\ X\in \obj\cu$), and a weight complex for
 any $w_{[l+1,m]}X$ can be assumed to be the corresponding stupid
  truncation of the weight complex of $X$.

  \item\label{i7} Assertions \ref{icompl} and \ref{impost} will be generalized
  in \S\ref{schws} below to the situation when there are two
  distinct weight structures; this will also clarify the proofs
  of these statements. Besides, note  that our remarks on functoriality 
are also actual for this setting.

Some of the proofs in \S\ref{schws} may also help to understand the
concept of virtual $t$-truncations (that we will start to study
 just now) better.
\end{enumerate}


\end{rema}

\subsection{Virtual $t$-truncations of
(cohomological) functors}\label{svirtr}

Till the end of the section $\cu$ will be endowed with a fixed
weight structure $w$; $H:\cu\to \au$ ($\au$ is an abelian category) will be a contravariant
(usually, cohomological) functor. We will not
consider covariant (homological) functors here; yet certainly,
dualization is absolutely no problem.

Now we recall the results of \S2.5 of \cite{bws} and develop the theory
 further.

\begin{theo}\label{trfun}

 Let $H:\cu\to \au$ be a contravariant functor, $k\in\z$, $j>0$.

I The assignments
$H_1=H^{kj}_1:X\mapsto \imm (H(w_{\le k}X)\to H(w_{\le
k+j}X))$ and $H_2=H_2^{kj}:X\mapsto \imm (H(w_{\ge k}X)\to H(w_{\ge
k+j}X))$  define contravariant functors $\cu\to\au$ that do not
depend (up to a canonical isomorphism) from the choice of weight
decompositions. We have natural transformations $H_1\to H\to H_2$.

II Let $k'\in\z$, $j'>0$. Then there exist the following natural
 isomorphisms.

1. $(H_1^{kj})_1^{k'j'}\cong H_1^{\min(k,k'),\max(k+j,k'+j')-\min(k,k')}$.

2. $(H_2^{kj})_2^{k'j'}\cong H_2^{\min(k,k'),\max(k+j,k'+j')-\min(k,k')}$.

3. $(H_1^{kj})_2^{k'j'}
\cong (H_2^{k'j'})_1^{kj} \cong
\imm(H(w_{[k,k']}X)
\to H(w_{[k+j,k'+j']}X))$.
Here the last term is defined using the connection morphism $w_{[k+j,k'+j']}X\to w_{[k,k']}X$ that is compatible with $\id_X$ in the sense of Remark \ref{rfunct}(\ref{idofunct}); 
the last isomorphism is functorial in the  sense  described in loc. cit.

III Let $H$ be cohomological, $j=1$; let $k$ be fixed.

1.  $H_l$ ($l=1,2$) are also cohomological;
the transformations  $H_1\to H\to H_2$ extend canonically to a long
exact sequence of functors
\begin{equation}\label{elovir}\dots\to H_2\circ [1]\to H_1\to H\to H_2\to H_1\circ [-1]\to \dots\end{equation} 
(i.e. the sequence is exact when applied
 to any $X\in\obj\cu$).

2. $H_1\cong H$ whenever $H$ vanishes on $\cu^{w\ge k+1}$.

3. $H\cong H_2$ whenever $H$ vanishes on $\cu^{w\le k}$.

4. Let $H'\stackrel{f}{\to} H\stackrel{g}{\to} H''$ be a
(three-term) complex of functors exact in the middle such that:

(i) $H',H''$ are cohomological.

(ii) for any $X\in \obj\cu$  we have $\cok g(X)\cong \ke f(X[-1])$
(we do not fix these isomorphisms).

(iii) $H'$ vanishes on $\cu^{w\ge k+1}$; $H''$ vanishes on $\cu^{w\le k}$.

Then $H'\stackrel{f}{\to} H$ is canonically isomorphic to
$H_1\to H$; $H\stackrel{g}{\to} H''$ is canonically isomorphic
 to $H\to H_2$.

\end{theo}
\begin{proof}

I This is   Proposition 2.5.1(III1) of \cite{bws}. 

II Easily follows from Theorem \ref{tbw}, parts \ref{idowe} and \ref{ioct}; see Remark \ref{rfunct}. 

III1. This is Proposition 2.5.1(III2) of \cite{bws}.

2. If $H$ vanishes on $\cu^{w\ge k+1}$ then for any $X$ we have
$w_{\ge k+1}X=0$; hence $H_2$ vanishes.
 Therefore in the long exact sequence
 $\dots\to H_2(X [1])\to H_1\to H\to H_2(X)\to \dots $
 given by assertion II1
  we have $H_2(X [1])\cong 0\cong H_2(X)$; we obtain $H_1\cong H$.

Conversely, suppose that $H_1\cong H$. Let  $X\in \obj\cu^{w\ge k+1}$;
 we can assume that $w_{\le k}X=0$. Then we have
 $H(X)\cong H_1(X)=\imm H(w_{\le k}X)\to H(w_{\le
k+1}X))=0$. 

3. It suffices to apply assertion II1 to the dual functor
$\cu^{op}\to \au^{op}$; note that the axiomatics of abelian
 categories, triangulated categories, and weight structures
 are self-dual (see Remark \ref{rts}(1) and  Theorem \ref{tbw}(\ref{idual})).

4. We should check that in the diagram
$$\begin{CD} 
H'_1@>{g}>> H_1 \\
@VV{h}V@VV{}V \\
H'@>{}>>H
\end{CD}$$
$g$ and $h$ are isomorphisms. Then $g\circ h\ob$ will yield the
 first isomorphism desired, whereas  dualization will yield the
 remaining half of the statement.

Now, assertion III2 yields that $g$ in isomorphism.

Next, for an $X\in \obj \cu$ we choose some weight decompositions
 for $X[k]$ and $X[k+1]$ and consider the diagram
{\footnotesize
 $$\begin{CD} 
H''((w_{\le k}X)[1]) @>{}>> H'(w_{\le k}X) @>{l}>> H(w_{\le k}X) @>{}>> H''(w_{\le k}X)
 \\@.
@VV{a}V@VV{b}V@. \\ H''((w_{\le k+1}X)[1]) @>{}>> H'(w_{\le k+1}X) @>{m}>> H(w_{\le k+1}X) @>{}>> H''(w_{\le k+1}X).
 \end{CD}$$
}

 By our assumptions, $H''((w_{\le k}X)[1])\cong H''(w_{\le k}X)\cong H''((w_{\le k+1}X)[1])\cong 0$; hence $l$ is an isomorphism and $m$ is a monomorphism.  Hence the induced map $\imm a\to \imm b$ is an isomorphism; so $h$ is an isomorphism (since its application to any $X\in \obj \cu$ is an isomorphism).

\end{proof}

\begin{defi}\label{rtrfun}[virtual $t$-truncations of $H$]

Let $k,m\in \z$.
 For a (co)homological $H$ we will call $H^{k1}_l$, $l=1,2,\ k\in\z$,
{\it virtual $t$-truncations} of $H$. We will often denote them
simply by $H_l$; in this case we will assume $k=0$ unless    $k$
is specified explicitly.

We denote the following functors $\cu\to\au$: $H_1^{k1}$, $H_2^{k-1,1}$, $(H^{m1}_2)^{k1}_1$,  and $X\mapsto (H_1^{01})_2^{-11}(X[k])$  
 by $\tau_{\le k}H$, $\tau_{\ge k}H$, $\tau_{[m+1,k]}H$, and $H^{\tau=k}$,  respectively. Note that all of these functors are  cohomological if $H$ is.

\end{defi}

\begin{rema}\label{rdfunvirt}
1. Note that $H$ often lies in a certain triangulated 'category of
functors' $\du$ (whose objects are certain cohomological functors $\cu\to \au$).
We will axiomatize this below by introducing the notion of  duality
 $\Phi:\cu^{op}\times \du\to \au$: if $\Phi$ is a duality then for
 any $Y\in \obj \du$ we have a cohomological functor
 $\Phi(-,Y):\cu\to \au$. It is also often the case when the
 virtual $t$-truncations
defined are compatible with actual $t$-truncations with respect to  some
$t$-structure $t$ on $\du$ (see below).
Still, it is very amusing that these $t$-truncated functors as well
as their transformations corresponding to $t$-decompositions (see
Definition \ref{dtstr}) can be described without specifying any
$\du$ and $\Phi$!

2. Below we will need an explicit description of the connecting morphisms in (\ref{elovir}).
We give it here (following the proof of Proposition 2.5.1 of \cite{bws}).

The transformation $H_1\to H$ (resp. $H\to H_2$) for any  $k,j$ can be calculated by applying $H$ to any possible choice either of $X\to w_{\le k}X$  or of $X\to w_{\le k+j}X$ (resp. of $w_{\ge k}X\to X$ or of $w_{\ge k+j}X\to X$)
that comes from any possible choice the corresponding weight decomposition. The transformation $H_2\to H_1\circ [-1]$ for $j=1$ is given by applying $H$ to any possible choice 
either of the morphism
$w_{\le k+1}X\to w_{\ge k+2}X[1]$ or of the morphism
$w_{\le k}X\to w_{\ge k+1}X[1]$ that comes from any possible choice of a  weight decomposition of $X[k]$.

Here 
we use the following trivial observation: for  $\au$-morphisms $X_1\stackrel{f_1}{\to}Y_1$ and $X_2\stackrel{f_2}{\to}Y_2$ any $g:X_1\to X_2$ (resp. $h:Y_1\to Y_2$) is compatible with at most one morphism $i:\imm f_1\to \imm f_2$; if such an $i$ exists, we will say that it is induced by $g$ (resp. by $h$). Certainly, here $f_1$ could be equal to $\id_{X_1}$ or $f_2$ could be equal to $\id_{X_2}$.

3. For  any  $k,j$, and any $\cu$-morphism $g:X\to Y$ the morphism $H_1(X)\to H_1(Y)$ (resp. $H_2(X)\to H_2(Y)$) is induced by any choice of either of the morphism $w_{\le k}X\to w_{\le k}Y$ or of $w_{\le k+j}X\to w_{\le k+j}Y$ (resp. of the morphism $w_{\ge k}X\to w_{\ge k}Y$ or of $w_{\ge k+j}X\to w_{\ge k+j}Y$) that is compatible with $g$ with respect to the corresponding weight decomposition (in the sense of Remark \ref{rfunct}(\ref{i1})); see the proof of Proposition 2.5.1 of \cite{bws}.


\end{rema}

We would like to extend assertion III4 of Theorem \ref{trfun}
 to a statement on a (canonical) isomorphism of long exact sequences
 of functors. To this end we  need the following definition.

\begin{defi}\label{dnicesfun}

1. We will call a sequence of functors
 $C=\dots \to H''\circ [1] \stackrel{[1](h)}{\to}  H'\stackrel{f}{\to} H\stackrel{g}{\to} H'' \stackrel{h}{\to} H'\circ [-1]\to\dots $ of
contravariant functors $\cu\to\ab$ a {\it strongly exact complex}
if $H',H,H''$ are cohomological and
$C(X)$ is a long  exact sequence for any $X\in \obj\cu$;
here $[1](h)$ is the transformation
induced by $h$.

2. We will also say that a strongly exact complex $C$ is
 {\it nice  } in $H$ if the following condition is fulfilled:

For any distinguished triangle
$T=A\stackrel{l}{\to} B \stackrel{m}{\to} C\stackrel{n}{\to} A[1]$ in $\cu$  the natural morphism $p$:
\begin{equation}\begin{gathered}\label{ecompat}\ke ((H'(A)\bigoplus H(B) \bigoplus H''(C))\xrightarrow{\begin{pmatrix}
f(A) & -H(l) &0  \\
0& g(B) &-H''(m)  \\
- H'([-1](n)) & 0 &h(C)
\end{pmatrix}}
\\ (H(A) \bigoplus H''(B) \bigoplus H'(C[-1])))
 \stackrel{p}{\to} \ke ((H'(A)\bigoplus H(B))\\ \xrightarrow{f(A)\oplus -H(l)}
 H(A)) \text{ is epimorphic.} \end{gathered}
\end{equation}

\end{defi}

Now we describe the connection of (\ref{ecompat}) with truncated
realizations; our arguments will also somewhat  clarify the meaning of this
condition. 

\begin{theo}\label{tvinice}

1. Let $C$ be a strongly exact complex of functors that is nice
 in $H$; let $H'\stackrel{f}{\to} H\stackrel{g}{\to} H''$
 (a 'piece' of $C$)
 satisfy the conditions of assertion III4 of Theorem \ref{trfun}.
Then $C$ is canonically isomorphic to (\ref{elovir}).

2. Let $X\to Y\to Z$ be a distinguished triangle in $\cu$.
 Then $C=\dots\to \cu(-,X)\to \cu(-,Y) \to \cu(-,Z)\to\dots$
 is a strongly exact complex of functors $\cu\to\ab$; it is nice
 in $\cu(-,Y)$.

3. Let there exist a (skeletally) small full triangulated
$\cu'\subset \cu$ such that the restriction of a strongly exact
complex $C$ to  $\cu'$ is nice in $H$.
For  $D\in \obj \cu$ 
we
consider the projective system $L(D)$
whose elements are $(E,i):\ E\in \obj \cu',\
i\in \cu(D,E)$;
we set $(E,i)\ge (E',i')$ if
$(E,i)=(E'\bigoplus E'',i'\oplus i'')$ for some
$(E'',i'')\in L(D)$.

Suppose  that 
for any $D\in \cu$ and for $G=H'$ and $G=H$
we
 have \begin{equation}\label{eprop} \inli_{L(D)} (\imm G(i):G(E)\to G(D))=G(D); \end{equation} here we also assume that these limits exist. Then $C$ is nice on $\cu$ also.

4. Let $\cu'\subset \cu$ be a (skeletally) small triangulated subcategory,
 let $\au$ satisfy AB5. Let $C'=\dots\to H'\to H\to H''\to\dots$ be a
  strongly exact complex of functors $\cu'\to \au$.
We extend all its terms from $\cu'$ to $\cu$ by the method of
Proposition \ref{pextc} and denote the complex obtained by $C$; we
carry on the notation for the terms and arrows from $C'$ to $C$.
Then $C$ is a strongly exact complex also (and its terms are
cohomological functors).

It is nice in $H$ whenever $C'$ is.
\end{theo}
\begin{proof}

1. It suffices to check that the isomorphism provided by 
 Theorem \ref{trfun}(III4) is compatible with the
coboundaries if (\ref{ecompat}) is fulfilled.
We can assume $\au=\ab$; see Remark \ref{rmit}.
Then (\ref{ecompat}) transfers into: for any 
$(x,y):\ x\in H'(A),\ y\in H(B),\ f(A)(x)=H(l)(y)$ there exists a 
\begin{equation}\label{ecomptr} 
z\in H''(C)\text{ such that }  g(B)(y)=H''(z) \text{ and }
H([-1](n))(x)=h(C)(z).\end{equation}

We should prove: if the images of $x\in H_2(X)$ and of $y\in H''(X)$
in $H''_2(X)$ coincide, $w\in H_1(X[-1])$ and $t= H(X)(y)\in H'(X[-1])$
 are their coboundaries, then   $w$ and $t$ come from some (single)
  $u\in H'_1(X[-1])$.

We lift $x$ to some $x'\in H(w_{\ge k+1}X)$. Then (\ref{ecompat})
(if we substitute $w_{\ge k+1}$ for $A$ and  $X$ for $B$ in it)
implies the existence of some $v\in H'((w_{\le k}X)[-1])$ whose
image in $H'(X[-1])$ (resp. in $H(w_{\le k}X[-1])$)
coincides with 
$t$ (resp. with the coboundary of $x'$). Hence we can take $u$
being the image of $v$ (in $H'_1(X[-1])$).

2. Since the bi-functor $\cu(-,-)$ is (co)homological with  respect
to both arguments, $C$ is a strongly exact complex indeed. It
remains to note: (\ref{ecompat}) in this case just means that any
commutative square can be completed to a morphism of distinguished
triangles; so it follows from the corresponding axiom (TR3) of
triangulated categories.

3. First suppose that $\au=\ab$ (or any other abelian category
equipped with an exact
faithful functor $\au\to\ab$ that respects small direct limits;
note that below we will only need $\au=\ab$). Then we should
check (\ref{ecomptr}).

Now note: it suffices to prove that there exist $A',B'\in \obj
\cu',\ l'\in \cu(A',B')$,  $\al\in \cu(A,A'),\ \be\in \cu (B,B'),\
x'\in H'(A'),\ g'\in H(B')$ such that:
\begin{equation}\label{ebchoi} x=H'(\al)(x'),\ y=H(\be) (y'),\
l'\circ \al=\be\circ l,\ f(A')(x')= H(l')(y').\end{equation}
Indeed, denote $C'=\co(l')$; denote by $\gamma$ some element of
$\cu(C,C')$ that completes $$\begin{CD}
A@>{}>>B \\
@VV{}V@VV{}V \\
A'@>{}>>B'
\end{CD}$$ to a morphism of triangles. Let $z'\in H''(C')$
be some element satisfying the obvious analogue of (\ref{ecomptr}).
Then $h=H''(\gamma)(h')$ is easily seen to satisfy (\ref{ecomptr}).

Now we construct $A',B',\dots$ as desired. Note that 
 in this case the assumption
(\ref{eprop}) is equivalent to: for any $t\in G(D)$ there exist
$E\in \obj\cu'$, $s\in G(D)$, and  $r\in \cu(D,E)$, such that
$t=G(r)(s)$ (since $\cu'$ is additive). So, we can choose $A'\in
\obj\cu'$, $\al\in \cu(A,A')$, $x'\in H'(A')$ such that $x=H'(\al)
(x')$. We complete $q=\al\oplus l\in \cu(A,A'\bigoplus B)$
 to a distinguished triangle
 $A\to A'\bigoplus B\stackrel{p=p_1\oplus p_2}{\to} D$.
 Since $H(q)((- H'(f(A')(x'),y))=0$, there exists an $s\in H(D)$
 such that $H(p) (s)= (- H'(f(A')(x'),y)$ (recall that
 $H$ is cohomological on $\cu$). So, we have $H(p_2)(s)=y$,
  $-H(p_1)(s)=f(A')(X')$, $p_2\circ l=-p_1\circ \al$.

$D$  fits for $B'$ if it lies in $\obj \cu'$. In the general case
using (\ref{eprop}) again, we choose $B'\in \obj \cu'$,
 $\de\in \cu(D,B')$, $g'\in H(Y)$, such that $s=H(\de) (g')$.
 Then it is easily seen that taking $l'=-\de\circ p_1$,
 $\be=\de\circ p_2$, we complete the choice of a set of data
 satisfying (\ref{ebchoi}). 

This argument can be modified to work for a general $\au$.
To this end 
we separate those parts of the reasoning
where we used the fact that $H$ is cohomological from those where
we deal with limits; this allows us to 'work as if $\au=\ab$'.

We denote
$\ke (H'(A)\bigoplus H(B)) {\to} H(A))$ (with respect to the morphism
in (\ref{ecompat})  by $S(A,B)$, and
{\small$\ke (H'(A)\bigoplus H(B) \bigoplus H''(C)) {\to} H(A)
\bigoplus H''(B) \bigoplus H'(C[-1])$}
by $T(A,B,C)$.

Then we have a commutative diagram
$$\begin{CD}
\inli(\imm (T(A',B',C')\to T(A,B,C)))@>{t'}>>\inli(\imm (S(A',B')\to S(A,B)))\\
@VV{}V@VV{i}V \\
T(A,B,C)@>{t}>> S(A,B)
\end{CD}$$
here the first direct limit above is taken with respect to
morphisms of triangles
$(A\to B\to C)\to (A'\to B'\to C')$
for $A',B',C'\in \obj \cu'$ (the ordering is similar to those of
(\ref{eprop})); the second limit is taken similarly with respect
to  morphisms $(A\to B)\to (A'\to B')$
for $A',B'\in \obj \cu'$. Since the restriction of $C$ to $\cu'$
is nice in $H$, for all $A',B',C'$ the morphism
$T(A',B',C')\to S(A',B')$  is epimorphic; hence $t'$ is epimorphic.
Therefore, it suffices to prove that $i$ is epimorphic.

Now let us fix $A'=A_0$ and $\al=\al_0$. We use the notation
introduced above; denote the preimage of
 $\imm(H'(\al):H'(A')\to H'(A))$ with respect to the natural morphism
  $S(A,B)\to H'(A)$ by $J$. Then $J$ equals
$\imm (H'(A')\times H(D)\to S(A,B))$. Indeed, here we can
 apply Proposition
 \ref{pmth}  (see Remark \ref{rmit}) and then apply the reasoning
 'with elements' used above. 

In any $\au$ we obtain: since
 $\Phi(D,Y)=\inli (\imm(\Phi(B',Y)\to\Phi(D,Y)))$, we obtain that
  $G=\inli(\imm (S(A_0,B',X,Y)\to S(A,B,X,Y)))$. Here we use the
  following fact (valid in any abelian $\au$): if
  $J_i\subset J'\in \obj\au$, $\inli J_i=J$ (for some projective system),
   $u:J'\to J$ is an $\au$-epimorphism, then $\inli u(J_i)=J$.

Now, passing to the limit with respect to $(A_0,\al_0)$
(using (\ref{eprop})) finishes the proof.

4. $C$ is a complex indeed since the extension procedure
is functorial. 

By  Proposition \ref{pextc}(I1), all the terms of $C$ are
cohomological on $\cu$. Also, part II2 of loc. cit. immediately
implies that $C$ is  exact (i.e. $C(X)$ is exact for any $X\in \obj
\cu$). Hence $C$ is a strongly exact complex.


Obviously, if $C$ is nice in $H$ then $C'$ also is.

Conversely, let $C'$ be nice in $H$. Then   Proposition
 \ref{pextc}(II1) implies that $H'$ and $H$ satisfy (\ref{eprop})
 (for all $D$). Hence $C$ is nice in $H$  by assertion 3.

\end{proof}

\subsection{Weight spectral sequences and  filtrations;
relation with virtual $t$-truncations }\label{scovi}

\begin{defi}\label{dwfilf}
For an arbitrary $(\cu,w)$ let $H:\cu\to \au$ be a
cohomological
functor ($\au$ is any
abelian category). 

We define $W^i(H):\cu\to\au $ as $X\to \imm(H(w_{\le
i}X)\to H(X))$.
\end{defi}

By  Proposition 2.1.2(2) of \cite{bws}, $W^i(H)(X)$ does
not depend on  the choice of a weight
decomposition for $X[i]$; it also defines a (canonical) subfunctor
of $H(X)$.

Now  recall that Postnikov towers yield spectral sequences for
cohomology.  We will denote $H(X[-i])$ by $H^i(X)$ (for $X\in
\obj\cu$). We will also use the notation of Definition \ref{rtrfun}.

\begin{theo}\label{pssw} 
Let $k,m\in \z$.

I1. For any weight Postnikov tower for $X$
(see  Definition \ref{d2}(\ref{idwpost}))
 there exists a spectral sequence $T=T(H,X)$ with $E_1^{pq}(T)=
H^q(X^{-p})$ such that
the map $E_1^{pq}\to E_1^{p+1q}$ is induced by the
morphism $X^{-p-1}\to X^{-p}$ (coming from the tower).
We have $T(H,X)\implies H^{p+q}(X)$ for any $X\in \cu^b$.

 One can construct it using the following exact couple:
$E_1^{pq}=H^q(X^{-p})$,
$D_1^{pq}=H^q(X^{w\ge 1-p})$.

2. $T$ is (covariantly) functorial in $H$; it is contravariantly
$\cu$-functorial  in $X$ starting from $E_2$.

3.  Denote the step of filtration given by ($E_{\infty}^{l,m-l}:$ $l\ge -k$)
 on $H^{m}(X)$
by $F^{-k}H^{m}(X)$. Then $F^{-k}H^{m}(X)=(W^k H^{m})(X)$.

II 
The derived exact couple for $T(H,X)$ can be naturally calculated
in terms of virtual $t$-truncations of $H$ in the following way: 
$E_2^{pq}\cong E_2'^{pq}=(H^q)^{\tau=-p}(X)$,
$D_2^{pq}=D_2'^{pq}= (\tau_{\ge q}H)(X[1-p])$; 
the connecting morphisms of the couple $((E_2',D_2'))$ come from (\ref{elovir}).

III1. $F^{-k}H^{m}(X)=\imm ((\tau_{\le k}H^m)(X)\to H^m(X))$ (with respect to the connecting morphism mentioned in Theorem \ref{trfun}(I)). 

2. For any $r\ge 2,\ p,q\in\z$ there exists a  functorial
isomorphism $E^{pq}_r\cong (F^{p}(\tau_{[-p+2-r,-p+r-2]}H)^q)^{p}/F^{p+1}(\tau_{[-p+2-r,-p+r-2]}H)^q)^{p}$.

\end{theo}
\begin{proof}
I This is  
 Theorem 2.4.2 of \cite{bws};
see also Remark 2.4.1 of ibid. for the discussion of exact couples.

In fact, assertion 1 
follows easily from well known properties of Postnikov towers  and
of related spectral sequences.

II Since virtual $t$-truncations are functorial, the exact couple $(D'_2,E'_2)$ is functorial also.

The definitions of the derived exact couple and of the virtual
 $t$-truncations imply immediately that  $D_2^{pq}$   and
 their connecting maps are exactly $D_2'^{pq}$ (and their connecting
 morphisms) specified in the assertion. 

It remains to 
compare $E_2$ with $E'_2$, and also the connecting maps of exact couples starting and ending in $E_2$ with those for $E_2'$. It suffices to consider $p=q=0$.
Our strategy is the following one. First we construct an isomorphism $E_2^{00}\to E_2'^{00}$; our construction depends on some choices. Then we prove that the isomorphism constructed is actually natural 
(in particular, it does not depend on the choices made). Lastly we verify that the isomorphisms of the terms of the exact couples constructed is compatible with the connecting morphisms of these couples. Note that in this (last) part of the argument we can make those choices (of certain weight decompositions) that we like.  %

By the definition of the derived exact couple we have:
$E_2^{00}$ is the $0$-th cohomology of the complex $(H(X^{-j}))$ (for any choice of the weight complex $(X^i)$).
$E_2'^{00}$ is the image of $H(k)$ where $k\in \cu(w_{[0,1]}X, w_{[-1,0]}X)$ is any morphism that is compatible with $\id_X$ with respect to the corresponding weight decompositions (see see Theorem \ref{trfun}(II3) and Remark \ref{rfunct}(\ref{idofunct})). 
 So, we should compare a subfactor of $H(X^0)$ with a subobject of $H(w_{[0,1]}X)$.

Now suppose that we are given an octahedral diagram containing a commutative triangle $w_{[1,1]}X\to w_{[0,1]}X\to w_{[-1,1]}X$ (see Theorem \ref{tbw}(\ref{idowe})).
We could obtain it as follows:
fix some $w_{[-1,1]}X$; then choose certain $w_{[0,1]}X=w_{\ge 0}(w_{[-1,1]}X)$ and  $w_{[1,1]}X =w_{\ge 1}(w_{[-1,1]}X)$ (see Remark \ref{rfunct}(\ref{i4})). 
For any possible completion of  the commutative triangle $w_{[1,1]}X\to w_{[0,1]}X\to w_{[-1,1]}X$ to an octahedral diagram, the remaining vertices of the octahedron are certain
 $w_{[-1,0]}X$, $w_{[0,0]}X=X^0$, and $w_{[-1,-1]}X=X^{-1}[1]$ (by Theorem \ref{tbw}(\ref{idowe})).  
We 
obtain morphisms
 $w_{[0,1]}X\stackrel{i}{\to} X^0 \stackrel{j}{\to} w_{[-1,0]}X$
 such that  $k=j\circ i$.  
  Moreover, $\imm (H(X^{1})\to H(X^{0}))=\ke H(i)$. Hence $H(i)$ induces some monomorphism $\al: H(X^0)/\imm (H(X^{1})\to H(X^{0}))$ to $H(w_{[0,1]}X)$.
 Besides, $\ke (H(X^{0})\to H(X^{-1}))=\imm H(j)$; therefore the restriction of $\al$ to $\al\ob(\imm H(k))$  yields an
 isomorphism $\be: E_2^{00}\to E_2'^{00}$. 
 
Now we verify  that the isomorphism constructed is natural.

Note that it actually depends only on $w_{[0,1]}X\stackrel{i}{\to} X^0$ and $\imm H(k)$ (we used the remaining data only in order to verify that we actually obtain an isomorphism).
So, suppose that we have $X'\in \obj \cu$, $g\in \cu(X,X')$, and some choice of
$w_{\ge  0}X'$, $w_{\ge  1}X'$, and $w_{\ge  2}X'$. 
We have canonical connecting morphisms $w_{\ge  0}X'\to w_{\ge  1}X'\to w_{\ge  2}X'$
that are compatible with $\id_{X'}$ with respect to the morphisms $w_{\ge  l}X'\to X'$ ($l=0,1,2$). Applying
Theorem \ref{tbw}(\ref{idowe}), we obtain
a choice of $w_{[0,1]}X'\stackrel{i'}{\to} X'{}^0$. 
We also fix some choice of $H(k')$ (in order to do this we fix some choice of $w_{\le -1}X$ and of $w_{[-1,0]}X$). Note that all of these choices are necessarily compatible with some choice of 
the isomorphism $\be':E_2^{00}(X')\to E_2'^{00}(X')$ constructed as above (see \ref{rfunct}(\ref{i1})). 

 Now we choose some morphisms $g_l:w_{\ge  l}X\to w_{\ge  l}X'$, for $-1\le l\le 2$, compatible with $g$ (see Remark \ref{rfunct}(\ref{i3})). These choices could be extended to some morphisms $a: w_{[0,1]}X\to w_{[0,1]}X'$ and $b:X^0{\to} X'{}^0$ (by extending morphisms of arrows to morphism of distinguished triangles). 
 
 Now we verify the commutativity of the diagram
$$\begin{CD}
w_{[0,1]}X @>{i}>>X^0\\
@VV{a}V@VV{b}V \\
w_{[0,1]}X'@>{i'}>>X'{}^0
\end{CD}$$ 
It follows from Theorem \ref{tbw}(\ref{icompl}) applied to the morphism  $g_0: w_{\ge  0}X\to w_{\ge  0}X'$, $l=1$, $m=2$ (since both $b\circ i$ and $i'\circ a$ are compatible with $g_0$). Moreover, Remark \ref{rfunct}(\ref{idofunct}) yields that $H(a)$ sends $H(k)$ to $H(k')$. We obtain a commutative diagram
 $$\begin{CD}
E_2^{00}@>{\be}>>E_2'^{00}\\
@VV{}V@VV{}V \\
E_2^{00}(H,X')@>{\be'}>>E_2'^{00}(H,X')
\end{CD}$$
Since $E_2^{00}(H,-)$ and $E_2'^{00}(H,-)$ are $\cu^{op}$-functorial 
(and the vertical arrows in the diagram are exactly those that yield this functoriality; see Remark \ref{rdfunvirt}(3)), we obtain the naturality in question. 

Now it remains to prove that the isomorphisms of  terms of exact
couples constructed above is compatible with the (two remaining) connecting morphisms
 of these couples.

First consider the morphisms $E_2^{00}\to D_2^{10}$.  Recall (by the definition of the derived exact couple)
 that it is induced by any morphism $w_{\ge 0}X\to X^0$ that extends to a weight decomposition of $w_{\ge 0}X$ (here we consider $E_2^{00}$ as a subfactor of $H(X^0)$). On the other hand,  the morphism
 $E_2'^{00}\to D_2'^{10}=\imm(H(w_{\ge -1}X)\to H(w_{\ge 0}X))$ is induced by 
 any possible choice of a morphism $w_{\ge 0}X\to w_{[0,1]}X$ that yields  a weight decomposition of $w_{\ge 0}X[1]$ (by  Remark \ref{rdfunvirt}(2); see also Remark \ref{rfunct}(\ref{idofunct})). 
 Hence it suffices to note that the triangle $w_{\ge 0} X\to w_{[0,1]}X\stackrel{i}{\to} X^0$ is necessarily commutative by Remark \ref{rfunct}.

It remains consider the morphism $ D_2^{1,-1} \to E_2^{00}$.
It is induced by the 
morphism  $X^0\to w_{\ge 1}X$ (that yields a weight decomposition of $w_{\ge 0}X$). The
morphism $ D_2'^{1,-1}(=\imm (H(w_{\ge 1}X)[1])\to H(w_{\ge 2}X)[1]) ) \to E_2'^{00}$ is induced by 
the morphism $w_{[0,1]}X\to w_{\ge 2}X[1]$. 
 Hence it suffices to construct a  commutative  square
 $$\begin{CD}
w_{[0,1]}X@>{i}>>X^0\\
@VV{}V@VV{}V \\
w_{\ge 2}X[1]@>{}>>w_{\ge 1}X[1]
\end{CD}$$
 By applying Theorem \ref{tbw}(\ref{idowe}) to 
 the commutative triangle $w_{\ge 2}X\to w_{\ge 1}X\to w_{\ge 0}X$
 we obtain that there exists such a commutative square with a certain $i_0$ instead of $i$.
 Note that (by loc. cit.) $i_0$ yields a weight decomposition of $w_{[0,1]}X$. It suffices to verify that  we may take $i_0$ for $i$ i.e. that $i_0$ could be completed to an octahedral diagram one of whose faces yields some choice of the commutative triangle $w_{[1,1]}X\to w_{[0,1]}X\to w_{[-1,1]}X$. We take $w_{[1,1]}X=\co i_0[-1]$, choose some $w_{[-1,1]}X$ (coming from the same $w_{\le 1}X$ as $w_{[0,1]}X$). 
 By Remark \ref{rfunct}(\ref{i1}) we obtain a unique commutative triangle $w_{[1,1]}X\to w_{[0,1]}X\to w_{[-1,1]}X$ that is compatible with $\id_{w_{\le 1}X}$ respect to the corresponding weight decompositions. 
 It remains to apply Theorem \ref{tbw}(\ref{idowe}).

III We can assume $k=m=0$.

1. In the notation of Theorem
\ref{trfun} we consider the morphism of spectral sequences
$M:T(H_1,X)\to T(H,X)$ (induced by $H_1\to H$).
Part II of loc. cit. implies: $M$ is an isomorphism on $E_2^{pq}$ for
$p\ge -k$ and $E_2^{pq}(T(H_1,X))=0$ otherwise. The assertion follows
 immediately.

2. Similarly to  the previous reasoning, we have natural
isomorphisms: $E_2^{pq}(T(\tau_{[2-r,r-2]}H,X)\cong E_2^{pq}(T(H,X))$
for $2-r\le p\le r-2$ and $=0$ otherwise. It easily follows that
  $E^{pq}_{\infty}(T(\tau_{[2-r,r-2]}H,X)\cong E^{pq}_{r}(T(\tau_{[-p+2-r,-p+r-2]}H,X)$.
  The result follows immediately. 

\end{proof}

 \begin{rema}

 1. The dual of assertion II is: if we consider the
 alternative exact couple for our weight spectral
 sequence  (see Remark
  \ref{rpost})
 then the derived exact couple can also be described
  in  terms of virtual $t$-truncations (in a way that is
  dual in an appropriate sense to that of Theorem \ref{pssw}).

 2. Possibly, at least a part of (assertion II of)
  the theorem could be proved by
 studying the functoriality of the derived exact couple (and applying
   Theorem \ref{tvinice}(1)).
\end{rema}

\subsection{Dualities of triangulated categories; orthogonal weight
and $t$-structures}\label{sdortstr}

Let $\cu,\du$  be triangulated categories.
We study certain pairings of triangulated categories
$\cu^{op}\times \du\to \au$.
In the following definition we consider a general $\au$, yet
below we will mainly need $\au=\ab$. 

\begin{defi}\label{ddual}
1. We will call a (covariant) bi-functor
  $\Phi:\cu^{op}\times\du\to
\au$ a {\it duality} if  it is bi-additive, homological with respect
to both arguments; and is equipped with a (bi)natural bi-additive transformation
$\Phi(X,Y)\cong \Phi (X[1],Y[1])$.


2. We will say that $\Phi$ is {\it nice} if for any distinguished
triangle $X\to Y\to Z$ the corresponding (strongly exact)  complex
of functors
 \begin{equation}\label{ecofefu} \dots \to \Phi(-,X) \to \Phi(-,Y)
 \to \Phi(-,Z) \stackrel{f}{\to} \Phi([-1](-),X)\to\dots\end{equation}
  is nice in $\Phi(-,Y)$ (see Definition \ref{dnicesfun}); here $f$ is obtained from the
   natural morphism $\Phi(-,Z) {\to} \Phi(-,X[1])$ by applying the
   (bi)natural transformation mentioned above.

 3. Suppose that $\cu$ is endowed with a weight structure $w$,
 $\du$ is endowed with a $t$-structure $t$. Then we will say that $w$
 is (left) {\it orthogonal} to $t$ with respect to $\Phi$
 if the following
 {\it orthogonality condition} is fulfilled:
 \begin{equation}\label{edort}
 \Phi (X,Y)=0\text{ if: }X\in \cu^{w\le 0}
\text{ and }Y\in \du^{t \ge 1},\text{ or }X\in \cu^{w\ge 0}
\text{ and }Y\in \du^{t \le -1}.
 \end{equation}

4. If $w$ is defined on $\cu^{op}$, $t$ is defined on $\du^{op}$,
$w$ is left orthogonal to $t$ (with respect to some duality); then
we will say that the corresponding opposite weight structure on
$\cu$ is  {\it right orthogonal} to the opposite $t$-structure for
$\du$.
\end{defi}

\begin{rema}
1. The axioms of $\Phi$ immediately imply that (\ref{ecofefu}) is a
strongly exact complex of functors indeed
(whether $\Phi$ is nice or not).

2. Certainly, if $\Phi$ is nice then (\ref{ecofefu}) is nice at
any term (since we can 'rotate' distinguished triangles in $\du$).

\end{rema}


First we prove a statement that will simplify checking
the orthogonality
 of weight and $t$-structures.

\begin{pr}\label{pdual}
 Let $\Phi:\cu^{op}\times\du\to \au$ be some duality; let $(\cu,w)$
be bounded. Then $w$ is (left) orthogonal to $t$ whenever there exists
a $D\subset \cu^{w=0}$ such that any object of $\cu^{w=0}$ is a
retract of a finite direct sum of elements of $D$ and
\begin{equation}\label{orth0}
\Phi(X,Y)=0\ \forall\ X\in D,\  Y\in \du^{t\ge 1}\cup \du^{t\le -1}.
\end{equation}

\end{pr}
\begin{proof}
 If $w$ is is left orthogonal to $t$, then (\ref{orth0})
for $D=\cu^{w=0}$ follows
immediately from the orthogonality condition.

Conversely, let $D$ satisfy the assumptions of our assertion.
 Hence
we have: $\Phi(X,Y)=0$ if $X\in D[i],\ i\ge 0, Y\in \du^{t\ge 1}$,
or if $X\in D[i],\ i\le 0, Y\in \du^{t\le -1}$.

Now we  note: if for some $E,F\subset\obj\cu$ we have
$\Phi (X,Y)=0$ if
  $X\in E$ and $Y\in \du^{t \ge 1}$ (resp. $X\in F$
and $Y\in \du^{t \le
 -1}$), then we also have $\Phi (X,Y)=0$ if
  $X\in E'$ and $Y\in \du^{t \ge 1}$ (resp. $X\in F'$
and $Y\in \du^{t \le
 -1}$) where $E'$ (resp. $F'$) is the smallest Karoubi-closed extension-stable subclass of $\obj\cu$ containing $E$ (resp. $F$).

Now by Theorem \ref{tbw}(\ref{iwgen}), for  $E=\cup_{i\ge 0} D[i]$, 
$F=\cup_{i\le 0} D[i]$, we have $E'=\cu^{w\le 0}$, $F'=\cu^{w\ge 0}$. Hence we obtain the orthogonality desired.

\end{proof}

When (weight and $t$-) structures are orthogonal, virtual
$t$-truncations of $\Phi(-,Y)$ are given by $t$-truncations in $\du$.
We use the notation of Definition \ref{rtrfun}.

\begin{pr}\label{pvirdu}
1. Let $t$ be orthogonal to $w$ with respect to $\Phi$, $k\in\z$.
For  $Y\in \obj \du$ denote the functor $\Phi(-,Y):\cu\to \au$ by $H$.
Then we have an isomorphism of complexes
$(\tau_{\le k}H\to H\to \tau_{\ge k+1}H)\cong
(\Phi(-,t_{\le k}Y)\to H \to \Phi(-,t_{\ge k+1}Y))$
(where the connecting maps of the second complex are induced
by $t$-truncations); this isomorphism is natural in $Y$.

2. Suppose also that $\Phi$ is nice. Then the (strongly exact) complex
of functors that sends $X$ to
\begin{equation}\label{ettrk}\dots\to \Phi(X, t_{\le k}Y)\to \Phi(X,Y)\to \Phi(X,t_{\ge k+1}Y)\to \Phi(X[-1],t_{\le k}Y)\to\dots\end{equation} (constructed as in the definition of a nice duality) is naturally isomorphic to (\ref{elovir}).
\end{pr}
\begin{proof}
1. Since $t$ and $w$ orthogonal, $\Phi(-,t_{\le k}Y)$ vanishes on
 $\cu^{w\ge k+1}$, whereas $\Phi(-,t_{\ge k+1}Y)$ vanishes on
 $\cu^{w\le k}$. Moreover, (\ref{ettrk}) yields that
$H'=\Phi(-,t_{\le k}Y)$ and $H''=\Phi(-,t_{\ge k+1}Y)$ also satisfy
the condition (iii) of   Theorem \ref{trfun}(III4). Hence the theorem
 yields the claim.

2. Immediate from the previous assertion and Theorem
\ref{tvinice}(1).

\end{proof}

\begin{rema}
Note that we actually need quite a partial case of
the 'niceness condition' for $\Phi$ in order to prove  assertion 2.
Hence here (and so, in all the applications below) we will not need
the niceness condition in its full generality. Possibly, the
corresponding partial case of the condition is weaker than the whole
 assertion; yet checking it does not seem to be much easier.

 Also, it seems quite possible that for an arbitrary (not
 necessarily nice) duality there exists some isomorphism of
 (\ref{elovir}) with (\ref{ettrk})  if we modify the boundary maps
 of the second complex. Yet there seems to be no way to choose such
  a modification
canonically.
\end{rema}

'Natural' dualities are nice; we will justify this thesis now.

\begin{pr}\label{pnice}
1. If $\au=\ab$, $\du=\cu$, then $\Phi:(X,Y)\mapsto\cu(X,Y)$
is a nice duality.

2. For some duality $\Phi:\cu^{op}\times \du\to \au$ let there exist
a (skeletally) small full triangulated $\cu'\subset \cu$ such
that: the restriction of $\Phi$ to $\cu'^{op}\times \du$ is a nice
duality (of $\cu'$ with $\du$);
for any $X\in \obj\du$ the functor $G=\Phi(-,X),\ \cu^{op}\to \au$,
satisfies (\ref{eprop}). Then   $\Phi$ is nice also.

3. For $\du$, $\cu'\subset \cu$ as above, $\au$ satisfying AB5,
let $\Phi':\cu'^{op}\times \du\to \au$ be a duality.
For any $Y\in \obj \du$ we extend the functor $\Phi'(-,Y)$ from
$\cu'$ to
$\cu$ by the method of Proposition \ref{pextc}; we denote the functor obtained by
$\Phi(-,Y)$. Then the corresponding bi-functor $\Phi$ is a duality
($\cu^{op}\times \du\to \au$).
It is nice whenever $\Phi'$ is.

\end{pr}
\begin{proof}
 Immediate from parts 2--4 of Theorem \ref{tvinice}.


\end{proof}

\begin{rema}\label{rnice}

1. Proposition \ref{pnice}(1) yields an important family of nice
dualities; this case was thoroughly studied in \cite{bws}
(in sections 4 and 7). We will say that $w$ is left (resp.
right) {\it adjacent}  to $t$ if it is left (resp. right) orthogonal
to it with respect to $\Phi(X,Y)=\cu(X,Y)$. 
Note that for $w$ left (resp. right)
 adjacent to $t$ with respect to this definition
we necessarily have $\cu^{w\le 0}=\cu^{t\le 0}$ (resp. $\cu^{w\ge
0}=\cu^{t\ge 0}$) by  Theorem \ref{tbw}(\ref{iort}) and
Remark \ref{rts}(\ref{iort}); so this definition is actually compatible
 with
Definition 4.4.1 of \cite{bws}.

One can generalize this family as in \S8.3 of ibid.:
for $\au=\ab$ and an exact $F:\du\to \cu$ we define
 $\Phi(X,Y)=\cu(X,F(Y))$. Certainly, one could also
 dualize this construction (in a certain sense) and
 consider $F:\cu\to \du$
and $\Phi(X,Y)=\cu(F(X),Y)$.

2. Another (general) family of dualities is mentioned in 
Remark 6.4.1(2) of ibid. All the families of dualities mentioned
 can be expanded
using part 3 of the proposition. 

3. It is also easy to construct a duality that is not nice. To this
 end one can start with $\cu=\du$, $\Phi=\cu(-,-)$ and then modify
 the choice of distinguished triangles in $\du$ (without changing the
 shift in $\du$, and changing nothing in $\cu$) in a way that would not
  affect the properties of functors to be cohomological. The simplest
  way to do this is to proclaim a triangle
   $X\stackrel{f}{\to} Y \stackrel{g}{\to} Z \stackrel{h}{\to} X[1]$
to be distinguished in $\du$ if
$X\stackrel{-f}{\to} Y \stackrel{-g}{\to} Z \stackrel{-h}{\to} X[1]$
is distinguished in $\cu$.
Certainly, such a modification is not very 'serious'; in particular,
 one can 'fix the problem' by multiplying the isomorphism
 $\Phi(X,Y)\cong  \Phi(X[1],Y[1])$ by $-1$.

The author does not know whether any duality can be made nice
by modifying the choice of the class of distinguished triangles
(in $\du$), or by modifying the isomorphism mentioned. Note also
that the question whether there exists a $\du$ for which such a
modification can change the 'equivalence class' of triangulations
is well-known to be open. 
\end{rema}

\subsection{Comparison of weight spectral sequences with those
coming from (orthogonal) $t$-truncations}\label{sconin}

Now we describe the relation of weight spectral sequences
with orthogonal structures.


\begin{theo}\label{tdual}
Let $w$ for $\cu$ and $t$ for $\du$ be orthogonal with respect
 to a duality $\Phi$; 
let $i,j\in \z$, $X\in \obj \cu$,
  $Y\in \obj \du$. 


\begin{enumerate}

 \item\label{isposts}
 Consider  the spectral sequence $S$  coming from
    the following exact couple: $D_2^{pq}(S)=\Phi(X,Y^{t\ge
q}[p-1])$, 
$E_2^{pq}(S)=\Phi(X,Y^{t=q}[p])$ (we start $S$ from $E_2$). It naturally
 converges to $\Phi(X,Y[p+q])$ if $X\in \cu^b$.

\item\label{ispostn}

 Let  $T$ be the weight spectral
sequence given by Theorem \ref{pssw} for the functor $H:\ Z\mapsto
\Phi(Z,Y)$. Then for all $r\ge 2$ we have natural
isomorphisms $E_r^{pq}(T(H,X))\cong E_r^{pq}(S)$.  There is also an
equality $F^{-k}H^{m}(X)=\imm (\Phi(X,t_{\le k}Y[m]) \to H^m(X))$
(here we use the notation of part I4 of loc. cit.) compatible with this
isomorphism.

\item\label{ispost}
Suppose that $\Phi$ is also nice. Then the isomorphism mentioned
in the previous assertion extends naturally to the isomorphism of
of $T$  with $S$
(starting from $E_2$).

    \item Let $\dots\to X^{-j-1}\to X^{-j}\to
X^{1-j}\to\dots$ denote an arbitrary choice of the weight complex for
$X$. Then we have a functorial isomorphism
{\small
\begin{equation}\label{nfrestrw} \begin{gathered}
\Phi(X,Y^{t=i}[j])\cong \\ (\ke (\Phi(X^{-j},Y[i])\to
 \Phi(X^{-1-j},Y[i])) /\imm
(\Phi(X^{1-j},Y[i])\to \Phi(X^{-j},Y[i])).
\end{gathered}\end{equation}}

\end{enumerate}

\end{theo}
\begin{proof}

\begin{enumerate}

\item The theory of $t$-structures easily yields: $Y^{t\ge
q}$ and $Y^{t=q}$ can be functorially organized into a certain
Postnikov tower for $Y$. Hence the usual results on spectral
sequences coming from Postnikov towers (see \S IV2, Exercise 2, of
\cite{gelman}) yield the assertion easily.

\item Immediate from Proposition
\ref{pvirdu}(1)   and  Theorem \ref{pssw}(III). Note that the
latter assertion does not use the 'dimension shift' in (\ref{elovir}).

\item    Proposition
\ref{pvirdu}(2)  and  Theorem  \ref{pssw}(II) imply: there is
a natural isomorphism of the derived exact couple for $T$ with
the exact couple of $S$ ('at level 2'). The result follows immediately.

    \item This is just assertion \ref{ispostn} for $E_2$-terms.

\end{enumerate}

\end{proof}


\begin{rema}\label{rdual}
\begin{enumerate}

\item

So, we justified parts 4 and 5 of Remark 4.4.3 of \cite{bws}.

\item

 Note that the spectral sequence denoted by $S$
in (Remark 4.4.3(4) and \S6.4 of) ibid.  
 started from $E_1$; so it differed from our $S$ and
$T$ by a certain shift of indices.

\item
 So, we developed an 'abstract triangulated alternative'
 to the method of comparing similar spectral sequences that
 was developed by Deligne and Paranjape. The latter method used
 filtered complexes; it was applied in \cite{paran}, \cite{ndegl},
 and in \S6.4 of \cite{bws}. The disadvantage of this approach is
  that one needs extra information in order to construct the
  corresponding filtered complexes; this makes difficult to study
  the naturality of the isomorphism constructed. Moreover, in some
   cases the complexes required cannot exist at all; this is the case
    for the {\it spherical weight structure} and its adjacent
    Postnikov $t$-structure for $\cu=\du=SH$ (the topological
    stable homotopy category; see \S4.6 of \cite{bws}; yet  in
    this case one can compare the corresponding spectral sequences
    using topology).

\item

One could modify our reasoning to prove a version of the theorem
that does not mention weight and $t$-structures. To this end instead
 of considering a weight Postnikov tower for $X$ and the Postnikov
  tower coming from $t$-truncations of $Y$ one should just
  compare spectral
  sequences coming from some Postnikov towers for $X$ and $Y$
  in the case when these Postnikov
   towers  satisfy those 'orthogonality' conditions (with respect
    to a (nice) duality $\Phi$)
that are implied by the orthogonality of structures condition in our
 situation. Yet it seems difficult to obtain the naturality of the
 isomorphisms in  Theorem \ref{tdual}(\ref{ispost})  
 using this approach.

\item

 Even more generally, it suffices to have  an inductive system of
 Postnikov towers
in $\du$ and a projective system of Postnikov towers
in $\cu$ such that the orthogonality conditions required are satisfied
 in the (double) limit. Then the comparison statements for the double
  limits of the corresponding spectral sequences are valid also. A
  very partial (yet rather important) example of a reasoning of this
  sort is described in \S7.4 of \cite{bws}. Besides, this approach
  could possibly yield the comparison result of \S6 of \cite{ndegl} (even
  without assuming  $k$ to be  countable as we do here).

\item

  A simple (yet important) case of (\ref{nfrestrw}) is:
for any $i\in \z$
  \begin{equation}\label{ehomhw} X\in \cu^{w=i}\implies
\forall Y\in \obj \du\text{ we have } \Phi(X,Y)\cong \Phi(X,Y^{t=i}).\end{equation}

\end{enumerate}

\end{rema}

\subsection{'Change of weight structures';
comparing weight spectral sequences}\label{schws}

Now we compare weight decompositions, virtual $t$-truncations, and
weight spectral sequences corresponding to distinct weight
structures. In order make our
results more general (and to apply them below) we will assume that
these structures are defined on distinct triangulated categories;
yet  the case when both are defined on $\cu$ is also interesting.

So, till the end of the section we will assume: $\cu,\du$ are
triangulated categories endowed with weight structures $w$ and $v$,
respectively;
$F:\cu\to \du$ is an  exact functor.

\begin{defi}\label{dwefun}

1. We will say that $F$ is {\it right weight-exact} if
  $F(\cu^{w\ge 0})\subset \du^{v\ge 0}$.

2. If $F$ is  fully faithful and  right
weight-exact, we will say that {\it $v$ dominates $w$}. 

3. We will say that $F$ is {\it left weight-exact} if
  $F(\cu^{w\le 0})\subset \du^{v\le 0}$.

  4. $F$ will be called {\it weight-exact} if it is both
  right and left weight-exact.

We will say that $w$ {\it induces} $v$ (via $F$) if $F$ is  a
weight-exact localization functor.
\end{defi}


\begin{pr}\label{pcws}

Let $F$  be a right weight-exact functor; let $l\ge m\in \z$, 
 $X\in
\obj \du$, $X'\in \obj \cu$, $g\in \du(F(X'),X)$.

1. Let
 weight decompositions
        of $X[m]$ with respect to $v$ and $X'[l]$ with respect
         to $w$ be fixed. Then
$g$ can be
        completed to a morphism of  distinguished triangles
\begin{equation}\begin{CD} \label{d2x3n}
F(w_{\ge l+1}X')@>{}>>F(X') @>{}>> F(w_{\le l}X')
\\
@VV{a}V@VV{g}V @VV{b}V\\
v_{\ge m+1}X@>{}>>X @>{}>> v_{\le m}X
\end{CD}\end{equation}
This completion is unique if $l>m$. 


2.  For arbitrary weight
    Postnikov towers $Po_v(X)$ for $X$ (with respect to $v$)
    and  $Po_{w}X'$ for $X'$ (with respect to $w$),
    $g$
can be extended to a morphism $F_*(Po_{w}X')\to Po_v(X)$. 

3. Let $H:\du\to \au$ be any functor, $k\in \z$, $j>0$.
Denote $H\circ F$ by $G$. Then (\ref{d2x3n}) 
allows to extend $H(g)$ naturally to a diagram
$$
\begin{CD}
H_1^v(X)@>{}>>H(X) @>{}>> H_2^v(X)\\
@VV{}V@VV{H(g)}V @VV{}V\\
G_1^w(X')@>{}>>G(X') @>{}>> G_2^w(X')
\end{CD}
$$
here we add the weight structure chosen as an index to
the notation of Theorem \ref{trfun}(I).
\end{pr}

\begin{proof} 

1. Since $F$ is right weight-exact,
$\du(F(w_{\ge n+1}X'), v_{\le m}X)=\ns$
for any $n\ge m$. Hence the composition morphism
 $F(w_{\ge l+1}X') \to v_{\le m}X$ 
 is zero; if $l>m$
 then   $\du(F(w_{\ge l+1}X'), (v_{\le m}X)[-1])=\ns$.
 The result follows easily; see Proposition 1.1.9 of \cite{BBD}.

2. Assertion 1 (in the case $l=m$) yields that there exists a
system of morphisms $f_i
\in  \du(F(w_{\ge i}X'), v_{\ge i}X)$ compatible with $g$;
we fix such a system. Applying the same assertion for any pair
of $l,m:\ l>m$, we obtain that $f_l$ is compatible with $f_m$ 
(here we use  arguments similar to those described in
Remark \ref{rfunct}). Finally, since any commutative square
can be extended to a morphism of the corresponding
distinguished triangles (an axiom of triangulated categories),
 we obtain that we can complete (uniquely up to a non-canonical
  isomorphism) the data chosen to a morphism of Postnikov towers
  (i.e. choose a compatible system of morphisms $F(X'^i)\to X^i$).

3. Easy from assertion 1; note that for any commutative square in $\au$
$$\begin{CD}
X@>{f}>>Y \\
@VV{h}V@VV{}V \\
Z@>{g}>>T
\end{CD}$$
if we fix the rows then the morphism $g\circ h:X\to T$  completely
determines the morphism $\imm f\to \imm g$ induced by $h$.

\end{proof}

We easily obtain a comparison morphism of weight spectral sequences.

\begin{pr}\label{pcwspsq}
I Let $F,X',G$ be as in the previous proposition; suppose also
that $H$ is cohomological. Set $X=F(X')$, $g=\id_X$.

 1.  There exists some comparison morphism of the corresponding
 weight spectral sequences $M:T_v(H,X)\to T_w(G,X')$.
 Moreover, this morphism is unique and additively functorial (in $g$)
  starting from $E_2$.

2. Let there exist $D\subset \cu^{w=0}$ such that any $Y\in
\cu^{w=0}$ is a retract of some $Z\in D$, and that for any $Z\in D$
there exists a choice of $Z^{w\ge 1}$ such that
$E_2^{pq}T_v(H,F(Z^{w\ge 1}))=\ns$ for all $p,q\in \z$. Then (any
choice of) $M$ yields an isomorphism of the spectral sequence
functors starting from $E_2$.


3. Let $\eu$ be a triangulated category endowed with a  weight
structure $u$, $F':\du\to\eu$ a right weight-exact functor; suppose
that $H=E\circ F'$ for some cohomological functor $E:\eu\to \au$.
Then we have the following associativity property for comparison of
weight spectral sequences: the composition of $M$ with (any choice
of) a comparison morphisms $M':T_u(E,F'(X))\to T_v(H,X)$
constructed as in assertion 1, starting from $E_2$ is canonically
isomorphic to  (any choice of a similarly constructed) comparison
morphism $ T_u(E,F'(X))\to T_w(G,X')$.

II Let $H,X',X,G$ be as above, but suppose that $F:\cu\to \du$ is left
weight-exact. Then a method dual to the one for assertion I1 yields
a transformation $N:  T_w(G,X')\to T_v(H,X)$; this construction
satisfies the  duals for all properties of $M$ described in
assertion I.

\end{pr}
\begin{proof}
I 1. In order to construct some comparison morphism, it suffices  to
construct a morphism of the corresponding exact couples that is
compatible with $\id_X$. Hence it suffices to use 
Proposition \ref{pcws}(2) to obtain a morphism of the corresponding
Postnikov towers, and then apply $H$ to it.

Theorem \ref{pssw}(II) yields that weight spectral sequences
could be described in terms of the corresponding virtual
$t$-truncations. Hence Proposition \ref{pcws}(3) implies all
the functoriality properties of $M$ listed.

2. It suffices to prove that $M$ is an isomorphism on
$E_2^{**}T_w(G,Y)$ for all $Y\in \obj\cu$.

Since $D\subset \cu^{w\ge 0}$, this assertion is true for any $Y\in
D$. Since  $Z\mapsto E_2(T(G,Z))$ is a cohomological functor for
any weight structure  (see Theorem \ref{pssw} and the remark at
Definition \ref{rtrfun}),
the assertion is also true for any $Y\in \obj \cu^b$. To
conclude it suffices to note that for any $H$, any
$Y\in \obj \cu$, any finite 'piece' of $E_2^{**}T_w(G,Y)$
coincides with the corresponding piece of $E_2^{**}T_w(G,w_{[i,j]}Y)$  
(for any choice of $w_{[i,j]}Y$) if $i$ is small enough and
$j$ is large enough, and this isomorphism is compatible with $M$.

3. We recall that comparison morphisms for weight spectral
sequences were constructed using Proposition \ref{pcws}(1).
 It easily follows that $M'\circ M$ is one of the possible choices
 for a comparison morphism $ T_u(E,F'\circ F(X))\to T_w(G,X')$.
  It suffices to apply assertion I1 to conclude that this fixed
  choice of a comparison morphism coincides with any other possible
   choice starting from $E_2$.

II We obtain the assertion from assertion I immediately by
dualization (see  Theorem \ref{tbw}(\ref{idual}));
here one should consider the duals of $\cu$, $\du$, and $\au$
(and also 'dualize' the connecting functors).
\end{proof}

\begin{rema}\label{rdomin}

 $M$ is an isomorphism (starting from $E_2$) also
if:  there exists a
localization of $\du$ such that $H$ factorizes through it, $v$
 induces a weight structure $v'$ on it, $w$ induces a weight structure
  on the categorical image of $\cu$  that coincides
with the restriction of $v'$ to it (since both weight spectral sequences are isomorphic to the spectral sequence corresponding to  this new weight structure).

Yet this conditions are  somewhat restrictive since
 weight structures do not 'descend' to localizations
 in general (since for an exact $F':\cu\to\eu$ the
 classes $F'_*(\cu^{w\ge 1})$ and $F'_*(\cu^{w\le 0})$
 are not necessarily orthogonal in $\eu$).

\end{rema}

In order to simplify checking right and left weight-exactness
of functors, we will need the following easy statement.

\begin{lem}\label{lrweex}
Let $w$ be bounded.

1. An exact $J:\cu\to \du$ is a right weight-exact whenever
there exists a $D\subset \cu^{w=0}$ such that any $Y\in \cu^{w=0}$
is a retract of some $X\in D$, and that for any $X\in D$ we have
 $J(Y)\in \du^{v\ge 0}$.

2. An exact $J:\cu\to \du$ is a left weight-exact whenever there
exists a $D\subset \cu^{w=0}$ such that any $Y\in \cu^{w=0}$ is a
retract of some $X\in D$, and that for any $X\in D$ we have
$J(Y)\in \du^{v\le 0}$.

\end{lem}
\begin{proof}

It suffices to prove assertion 1, since assertion 2 is exactly its dual.

If $J$ is right weight-exact functor, then we can take $D=\cu^{w=0}$.

Now we prove the converse statement. 
Since $\du^{v\ge 0}$ is
Karoubi-closed and extension-stable in $\du$, 
Theorem \ref{tbw}(\ref{iwpost}) yields that $J(\cu^{w\ge 0})$ indeed belongs to
$\du^{v\ge 0}$.

\end{proof}


\section{Categories of comotives (main properties)}\label{scomot}

We embed $\dmge$ into a certain big triangulated motivic category
$\gd$; we will call it objects {\it comotives}. We will need several
properties of $\gd$; yet we will never use its description directly.
For this reason in \S\ref{comot} we only list the main properties of
$\gd$.

In \S\ref{spsch} we associate certain comotives to (disjoint unions of) 'infinite
intersections' of smooth varieties over $k$ (we call those
pro-schemes). We also introduce certain Tate twists for these
comotives.

In \S\ref{sprim} we recall the definition of a primitive scheme
(note that in the case of a finite $k$ we call a scheme primitive
whenever it is smooth semi-local). The main motivic
property of primitive
schemes (proved by M. Walker) is: $F(S)$ injects into $F(S_0)$ if
$S$ is primitive connected, $S_0$ is its generic point, and $F$ is a homotopy invariant presheaf
with transfers.

In \S\ref{smotprim} we study the relation of (the comotives of)
primitive schemes with the homotopy $t$-structure for $\dme$.

In \S\ref{smorprim} we prove that there are no $\gd$-morphisms
 of positive degrees between the comotives of primitive
schemes (and also certain Tate
 twists of those); this is also true for products of the comotives mentioned.

 In \S\ref{scgersten} we
 prove that one can pass to countable homotopy limits in Gysin distinguished triangles; this yields Gysin distinguished triangles for the comotives of pro-schemes. This allows to
 construct   certain
 Postnikov towers for the comotives of pro-schemes (and their Tate twists),
 whose factors are twisted products of
the comotives of function fields (over $k$).
 The construction of the tower is parallel to the
classical construction of
 coniveau spectral sequences (see \S1 of \cite{suger}).

\subsection {Comotives: an 'axiomatic description'} 
\label{comot}

We will define $\gd$ below as the derived category of differential
graded functors $J\to B(\ab)$;
 here $J$  yields a differential graded enhancement of $\dmge$
 (cf. \cite{bev} or \cite{le3}), 
  $B(\ab)$
is the differential graded category
 of complexes over $\ab$. We will also need some
 category $\gdp$ that projects to $\gd$ (a certain model of $\gd$). Derived
categories of differential
 graded functors were studied in detail in \cite{drinf}
and \cite{dgk}. 
We will define and study them in \S\ref{rdg} below; now
we will only list their properties
that are needed for the proofs of main statements.

Below we will also need certain (filtered) inverse limits several
times. $\gd$ is a triangulated category; so it is no wonder that
there are no nice limits in it. So we consider a certain additive
$\gdp$ endowed with an additive functor $p:\gdp\to \gd$.
 We will call  (the images of) inverse limits  from $\gdp$
homotopy limits in $\gd$.

The relation of $\gdp$ with $\gd$ is similar to the relation of
$C(\au)$ with
$D(\au)$. 
In particular, $\gdp$ is closed with respect to all (small
filtered) inverse limits; we have functorial cones of morphisms in
$\gdp$ that are compatible with inverse limits.

We will need some conventions and definitions introduced in Notation; in particular,  $I,L$  will be  projective systems;  $I$ is countable.

\begin{pr}\label{pprop}

\begin{enumerate}

    \item\label{ip1}   $\gd$ is a triangulated category; $\gd\supset
    \dmge$, and
    all objects
 of $\dmge$ are cocompact in $\gd$.

    \item\label{ip2}  There is an additive category
    $\gdp$ closed with respect to arbitrary
(small filtered) inverse limits,
    and an additive functor $p:\gdp\to \gd$ that preserves (small)
    products. Moreover, $p$ is surjective on objects.

    \item \label{ip3} $\gdp$ is endowed with a certain
invertible shift functor $[1]$
    that is compatible with the shift on $\gd$ and respects
 inverse limits.

 \item \label{ip7} There is a functorial cone of morphisms
in $\gdp$ defined;
    it is compatible with $[1]$ and respects  inverse limits.

    \item \label{ip8} Any triangle of the form
    $X\stackrel{f}\to Y\to \co(f)\to X[1]$
 in $\gdp$
    becomes distinguished in $\gd$.

    \item \label{ip4} The composition functor
$\mg:C^b(\smc)\to \dmge\to \gd$ can
     be canonically factorized through an additive functor
 $j:C^b(\smc)\to
     \gdp$. Shifts and cones of morphisms in $C^b(\smc)$ are
     compatible with those in $\gdp$ via $j$.

 \item \label{ipcf}
     For any $X\in \mg(C^b(\smc))\subset \obj \gd$, any $Y:L\to
     \gdp$ we have $\gd(p(\prli _{l\in L}Y_l),X)=\inli_{l\in L} \gd
     (p(Y_l),X)$.

    \item \label{ip5} $\dmge$ weakly cogenerates $\gd$
(i.e.
     we have ${}^\perp\dmge=\ns$, see Notation).

    \item \label{ip6} Let a sequence $i_n\in I$, $n>0$, be increasing (i.e. $i_{n+1}>i_n$ for any $n>0$)
unbounded (see Notation). Then
    for all functors $X:I\to \gdp$, 
    we have  functorial distinguished
    triangles in $\gd$:
\begin{equation}\label{elim} p(\prli_{i\in I} X_i)\to p(\prod X_{i_n})
\stackrel{e}{\to} p(\prod X_{i_n});\end{equation} 
$e$ is the product of
$\id_{X_{i_n}}\oplus - \phi_n:X_{i_{n+1}}\to X_{i_n}$; here
$\phi_n$ are the morphisms coming from $I$ via $X$.

\item\label{ienh} \label{ip9} There exists a {\it differential
graded enhancement}
 for $\gd$; see \S\ref{sbdedg} below.

\end{enumerate}

\end{pr}

\begin{rema}\label{raxiomdim}

1. Since below we will prove some statements for $\gd$ only using its 'axiomatics' (i.e. the properties listed in Proposition \ref{pprop}), these results would also be valid in any other category that fulfills these properties. This could be useful, since the author is not sure at all that all possible $\gd$ are isomorphic.

2. Moreover, one could modify the axiomatics of $\gd$ and
consider instead a category that would only contain the triangulated
subcategory of $\dmge$ generated by motives of smooth varieties of dimension
 $\le n$ (for a fixed $n>0$). Our results and arguments below can
 be easily carried over to this setting (with minor modifications;
it is also useful here to weaken condition \ref{ip5} in the Proposition).
This makes sense since these 'geometric pieces' of $\dmge$ are
self-dual with respect to Poincare duality (at least, if $\cha k=0$);
 see \S\ref{sdual}  below. See also  Remark \ref{rhocolim}(2).

Alternatively, 
we can 
weaken the condition for the functor $\dmge\to \gd$ to be a full
 embedding. For example, it could be interesting to consider the version of $\gd$ for which
 this functor kills $\dmge(n)$ (for some fixed $n>0$).

 Lastly note that we do not really need condition \ref{ip2} in its full generality (below).
\end{rema}

Now we derive some consequences from the axiomatics listed.

\begin{coro}\label{css}

\begin{enumerate}

\item For any $Z\in \obj \dmge\subset \obj \gd$, any $X:L\to \gdp$
we have
$\gd (p(\prli_{l\in L} X_l),Z) =\inli_{l\in L} \gd (p(X_l),Z)$.

\item For any $T\in \obj \gd$, all functors $Y:I\to \gdp$
we have functorial
 short exact sequences
$$\ns\to \prli^1 \gd(T,p(Y_i)[-1])\to  \gd (T,p(\prli Y_i))\to
\prli  \gd(T,p(Y_i)) \to \ns;$$
here $\prli^1$ is the (first) derived functor of
 $\prli=\prli_I$. 

\item For all functors $X:L\to C^b(\smc)$, $Y:I\to C^b(\smc)$,
 we have
functorial short exact sequences
\begin{equation}\label{lmor}
\begin{aligned} \ns\to \prli^1_{i\in I} (\inli_{l\in L} \gd(\mg(X_l),\mg(Y_i)[-1]))\to \\
 \gd (p(\prli_{l\in L} j(X_l)),p(\prli_{i\in I} j(Y_i)))\to \\ \prli_{i\in I}
 (\inli_{l\in L} \gd(\mg(X_l),\mg(Y_i)))\to \ns.\end{aligned}
\end{equation}

 \item \label{ipid} $\gd$ is idempotent complete.
\end{enumerate}
\end{coro}
\begin{proof}
\begin{enumerate}
\item If $Z\in \mg(C^b(\smc))$, then the assertion is exactly 
  Proposition \ref{pprop}(\ref{ipcf}).

It remains to note that any $Z\in \obj\dmge$ is a retract of some
object coming from $C^b(\smc)$.

\item Since inverse limits and their derived functors
do not change when we replace a projective
system by any unbounded subsystem, we can assume
that $L$ consists of some $i_n$ as in (\ref{elim}).

Now, (\ref{elim}) yields a long exact sequence
$$\begin{gathered}
\dots\to \prod_{i\in I}  \gd(T,p(Y_i)[-1]) \stackrel{f}{\to} \prod_{i\in I}  \gd(T,p(Y_i)[-1]) \to  \gd (T,p(\prli_{i\in I}  Y_i))\\
\to \prod_{i\in I}  \gd(T,p(Y_i))
\stackrel{g}{\to} \prod_{i\in I}  \gd(T,p(Y_i))\to \dots,\end{gathered}$$
here $f$ and $g$ are induced by $e$ in (\ref{elim}).

It is easily seen that $\ke g\cong \prli  \gd(T,\mg(Y_m))$.

Lastly,  Remark A.3.6 of \cite{neebook} allows to identify 
$\cok f$ with 
 $\prli^1 \gd(T,\mg(Y_m)[-1])$.

\item Immediate from the previous assertions.

\item Since $\gdp$ is closed with respect to all inverse limits, it is
 closed with respect to all (small) products. Then  
 Proposition \ref{pprop}(\ref{ip2}) yields that $\gd$ is also  closed with
respect to all products.
 Now, Remark 1.6.9 of \cite{neebook} yields the result
 (in fact, the proof uses only countable products).
\end{enumerate}
\end{proof}

We will often call the objects of $\gd$ {\it comotives}.

\subsection{Pro-schemes and their comotives}\label{spsch}

Now we have certain inverse limits for objects (coming from)
 $C^b(\smc)$; this allows
to define (reasonable) comotives for certain
schemes that are not (necessarily) of finite type over $k$ (and for their disjoint unions). 
We also define certain Tate twists of those.

We will call certain  formal disjoint unions of projective limits of smooth varieties over $k$  {\it pro-schemes}. Such a union  $V/k$ is a pro-scheme if it is a countable disjoint union of  projective limits of smooth varieties of dimension $\le c_V$ for some fixed $c_V\ge 0$ (in the category of schemes) with connecting morphisms being open dense embeddings (note that these projective limits are often $k$-schemes).  One may say that a pro-scheme is a countable disjoint union of 
 intersections of smooth varieties.
Note that  (the spectrum of) any function field over $k$ yields a pro-scheme; any smooth $k$-variety or any its (semi)-localizations is a
 pro-scheme also.
We have the operation of countable disjoint union for pro-schemes of bounded dimension.

Now, we would like to present a (not necessarily connected) pro-scheme $V$ as projective limits of smooth varieties $V_i$. This is easy if $V$ is connected (cf. Lemma 3.2.9 of \cite{deggenmot}). In the general case one should allow (formally) zero morphisms between connected components
of $V_i$ (for distinct $i$). So we consider a new category $\sv'$ containing the category of all smooth varieties as a (non-full!) subcategory. We take $\obj \sv'=\sv$; for any smooth connected varieties $X,Y\in \sv$ we have $\sv'(X,Y)=\mo_{\var}(X,Y)\cup\ns$; the composition of a zero morphism with any other one is zero; $\sv'(\sqcup_{i} X_i, \sqcup_{j} Y_j)=\sqcup_{i,j} \sv'(X_i,Y_j)$. 
$\sv'$ can be embedded into $\smc$ (certainly, this embedding is not full). 

We will write $V=\prli V_i$. 
Note that the set of connected components of $V$ is the inductive limit
of the corresponding sets for $V_i$. 

Now,  for any pro-scheme $V=\prli V_i$, any $s\ge 0$, we introduce the following notation: $\mg(V)(s)=p(\prli (j(V_i)(s)))\in \obj \gd$
(see Proposition \ref{pprop}); we will denote $\mg(V)(0)$ by  $\mg(V)$ and call $\mg(V)$ the comotif of $V$.
This notation should be considered
as formal i.e. we do not define Tate twists on $\gd$
 (till \S\ref{snpttw}).

 Obviously, if $V\in \sv$, its comotif (and its twists) coincides
 with its motif (and its twists), so we can use the same notation
 for them.

If $\au$ is a category closed with respect to filtered direct
limits, $H':\dmge\to \au$ is a functor, we can (formally) extend it
to co-motives in question; we set:
\begin{equation}\label{dfps} H(\mg(V)(s)[n])=\inli H'(
\mg(V_i)(s)[n]).
\end{equation}

\begin{rema}\label{rspsch}\label{rlprod}

1. For a general $H'$ this notation should be considered as formal.
Yet in the case $H'=(-,Y):\gd\to \ab$,  $Y\in \obj \dmge\subset \obj
\gd$, we have $H(\mg(V)(i)[n])=\gd (\mg(V)(i)[n],X)$; see  
Corollary \ref{css}(1), i.e. (\ref{dfps}) yields the value of a
well-defined functor $\gd\to \ab$ at $\mg(V)(s)[n]$. We will only
need $H'$ of this sort till \S\ref{sext}.

More generally, there exists such an $H$ if $\au$ satisfies AB5 and $H'$
is cohomological;  we will call the corresponding $H$ an {\it
extended cohomology theory}, see
 Remark \ref{rcohp} below.

2. Let $V^j$ be a countable set of pro-schemes (of bounded dimensions).
Then
$\mg(\sqcup V^j)=\prod \mg(V^j)$ by 
Proposition \ref{pprop}(\ref{ip2}).

Besides, for any $H'$ as in (\ref{dfps}) we have $H(\mg(\sqcup V^j)(s)[n])=\bigoplus H(\mg(V^j)(s)[n])$.

\end{rema}

Below we will need some conventions for pro-schemes.

For pro-schemes $U=\prli U_i$ and $V=\prli V_j$ we will call an
element of $\prli_{i\in I}(\inli_{j\in J} \smc (U_i,V_j))$
 an open embedding if it
 can be obtained as  a double limit of open
 embeddings $U_i\to V_j$ (as varieties).
 If $U=V\setminus W$ 
  for some pro-scheme $W$, we will say that $W$
  is a closed sub-pro-scheme of $V$. Note that in this
  case any connected component of  $W$ is a closed subscheme
  of some connected component of $V$; yet some components of $V$
   could contain an infinite set of connected components of $W$.

 For $V=\sqcup V^j$, $V^j$ are connected pro-schemes, we
  will call the maximum of the transcendence degrees of
  the function fields of $V^j$ the dimension of $V$ (note that
   this is finite). We will say that a sub-pro-scheme
 $U=\sqcup U^m$, $U^m$ are connected, is everywhere of
 codimension $r$ (resp. $\ge r$, for some fixed $r\ge 0$)
 in $V=\sqcup V^j$ if for every induced embedding $U^m\to V^j$
 the difference of their dimensions (defined as above) is $r$
 (resp. $\ge r$).

 We will call the inverse limit of the sets of
points of $V_i$ of a fixed codimension $s\ge 0$ the set of points of
$V$ of  codimension $s$ (note that any element of this set indeed
defines a point of some connected component of $V$).

\subsection{Primitive schemes: reminder}\label{sprim}

In \cite{walker} M. Walker proved that primitive schemes in the case
of an infinite $k$ have 'motivic' properties similar to those of smooth
semi-local schemes (in the sense of \S4.4 of \cite{3}). Since we don't want to discriminate the case of a
finite $k$, we will modify slightly the standard definition of
primitive schemes.

\begin{defi}

If $k$ is infinite then a (pro-)scheme is called primitive if all
of its connected components are affine $k$-schemes and their coordinate rings
$R_j$ satisfy the following primitivity criterion: for any $n>0$ 
every polynomial in $R_j[X_1,\dots,X_n]$ whose coefficients generate
$R_j$ as an ideal over itself, represents an $R_j$-unit.

If $k$ is finite, then we will call a pro-scheme primitive
whenever all of its connected components are 
semi-local (in the sense of \S4.4 of \cite{3}).

\end{defi}

\begin{rema}\label{rabprim}

Recall that in the case of infinite $k$ all semi-local $k$-algebras
satisfy the primitivity criterion (see Example 2.1 of
\cite{walker}).
\end{rema}

Below we will mostly use the following  basic property of
primitive schemes.

\begin{pr} \label{pwalk}
Let $S$ be a primitive pro-scheme, let $S_0$ be the collection of all of its generic
points; $F$ is a homotopy invariant presheaf with transfers. Then $F(S)\subset F(S_0)$;
here we define $F$ on pro-schemes as in (\ref{dfps}).
\end{pr}
\begin{proof}
We can assume that $S$ is connected (so it is a smooth primitive scheme).

Hence in the case of infinite $k$ our assertion follows from Theorem 4.19 of
\cite{walker}. Now, if $k$ is finite, then $S_0$ is
semi-local (by our convention);
so we may apply Corollary 4.18 of \cite{3} instead.

\end{proof}

\subsection{Basic motivic properties of primitive
schemes}\label{smotprim}

We will call a primitive pro-scheme just a primitive scheme.  We
prove certain motivic properties of primitive schemes (in the form
in which we will need them below).

\begin{pr}\label{porthog}
For $F\in \obj\dme$ we define $H'(-)=\dme(-,F)$ on $\dmge$; we also
define $H(\mg(V)(i)[n])$ as in (\ref{dfps}).
Let $S$ be a primitive scheme, $m\ge 0$, $i\in \z$.

1. Let $F\in \dme^{t\le -1}$ ($t$ is the homotopy $t$-structure, that we considered in \S\ref{dvoev}). Then $H(\mg(S)(m)[m])=\ns$.

2. More generally, for any $F\in \obj \dme$ we have
$H(\mg(S)(m)[m])\cong F^0_{-m}(S)$ where $F^0=F^{t=0}$, $F^0_{-m}$
is the $m$-th Tate twist of $F^0$ (see Definition \ref{dtw}).

\end{pr}
\begin{proof}

1. We consider the homotopy invariant presheaf with transfers
$F_{-m}:\ X\mapsto \dme(\mg(X)(m)[m],F)$. We should prove that
$F_{-m}(S)=0$ (here we extend $F_{-m}$ to pro-schemes in the usual
way i.e. as in (\ref{dfps})).

(\ref{dfps}) also yields that $F_{-m}(\sqcup S_i)=\bigoplus
F_{-m}(S_i)$.  Hence by Proposition \ref{pwalk}, it suffices to
consider the case of $S$ being (the  spectrum of) a function field
over $k$. Since $F_{-m}$ is represented by an object of $\dme^{t\le
-1}$ (see Proposition \ref{padj}(2)), it suffices to note
that any field is a Henselian scheme i.e. a point in the Nisnevich
topology.

2. By Proposition \ref{padj}, for any $X\in \sv$ we have
$\mg(X)(m)[m]\perp \dme^{t\ge 1}$. Hence we can
assume $F\in \dme^{t\le 0}$.

Next, using assertion 1, we can easily reduce the situation to the
case $F=F^{t=0}\in\obj HI$ (by considering the $t$-decomposition of
$F[-1]$). In this case the statement is immediate from 
Proposition \ref{padj}(1).


\end{proof}

\begin{lem}\label{lprdegl}

Let $U\to U'$ be an open dense embedding of smooth varieties.

1.  We have $\co(\mg(U)\to \mg(U'))\in  \dme^{t\le -1}$.

2. Let $S$ be primitive. Then for any $n,m,i\ge 0$ the map
$$\gd(\mg(S)(m)[m],\mg(U)(n)[n+i])\to
\gd(\mg(S)(m)[m],\mg(U')(n)[n+i])$$ is surjective.

\end{lem}
\begin{proof}
1. We denote $\co(\mg(U)\to\mg(U'))$ by $C$.
Obviously, $C\in \dme^{t\le 0}$. Let $H$ denote $C^{t=0}$
($H\in\obj HI$). By Corollary 4.19 of \cite{3}, we have $H(U)\subset
H(U')$. Next, from the long exact sequence $\ns(=\dme
(\mg(U)[1],H))\to \dme(C,H)\to \dme (U',H) \to \dme (U,H)\to \dots$
we obtain $C\perp H$. Then the long exact sequence $\dots\to
\dme(C^{t\le -1}[2],H)\to \dme(H,H)\to\dme (C,H)\to\dots$ yields
$H=0$.

2. It suffices to check that $\mg(S)(m)[m]\perp C(n)[n+i]$.
Since $\mg(U)(n)[n]$ is canonically a retract of
$\mg(U\times G_m^n)$, 
we can assume that $n=0$.

Now the claim
follows immediately from assertion 1 and 
Proposition \ref{porthog}(1). 

\end{proof}

\subsection{On morphisms between the comotives of
primitive schemes}\label{smorprim}

We will need the fact that certain 'positive' morphism groups are
zero.

Let $n,m,\ge 0$, $i>0$, $Y=\prli Y_l$ ($l\in L$), 
be any pro-scheme, $X$ be a primitive scheme.

\begin{pr}\label{maneg}

\begin{enumerate}

\item  The natural
homomorphism 
$$
\displaylines{\gd(\mg(X)(m)[m],\mg(Y)[n](n)) \to\hfill\cr
\hfill \to \prli_l(\inli_{X\subset Z,Z\in\sv}
\dmge(Z(m)[m],\mg(Y_l)(n)[n]))}$$ is surjective.   

\item $\mg(X)(m)[m]\perp \mg(Y)[n+i](n)$.

\end{enumerate}
\end{pr}

\begin{proof}

Note first that by the definition of the Tate twist $(1)$, it can be
lifted to $C^b(\smc)$.
\begin{enumerate}

\item  This is immediate from the short exact sequence (\ref{lmor}).

\item  By  Remark \ref{rlprod}(2), we may suppose that $Y$
is connected. 
Then we apply (\ref{lmor}) again.  The corresponding $\prli$-term is
zero by Proposition \ref{porthog}(1). Lastly, the surjectivity
proved in  Lemma \ref{lprdegl}(2) yields that
the corresponding $\prli^1$-term is zero. Indeed, the groups
 $\gd(\mg(X)(m)[m],\mg(Y_l)[n+i-1](n))$ obviously satisfy the
  Mittag-Leffler condition; see \S A.3 of \cite{neebook}.

In fact, one could 
easily deduce the assertion from the results of the previous
subsection and (\ref{elim}) directly (we do not need
much of the theory of higher limits in this paper).

\end{enumerate}

\end{proof}

\begin{rema}

In fact, this statement, as well as all other properties of
(primitive) pro-schemes that we need, are also true for not
necessary countable disjoint unions of (primitive) pro-schemes. This
observation could be used to extend the main results of the paper
 to a somewhat larger category; yet such an extension does not seem
to be important. 

\end{rema}

\subsection{The Gysin distinguished triangle for pro-schemes;
'Gersten' Postnikov towers for the comotives of pro-schemes}\label{scgersten}

We prove that we can pass to  countable homotopy limits in
  Gysin distinguished triangles.

\begin{pr}\label{pinfgy}

Let $Z,X$ be pro-schemes, $Z$ a closed subscheme 
 of $X$ (everywhere) of codimension $r$. Then for any
 $s\ge 0$ the natural morphism $\mg(X\setminus Z)(s)\to \mg(X)(s)$
  extends to a  distinguished triangle (in $\gd$):
  $\mg(X\setminus Z)(s)\to \mg(X)(s)\to \mg(Z)(r+s)[2r]$.

\end{pr}
\begin{proof}

First assume $s=0$.

We can assume $X=\prli X_i$, $Z=\prli Z_i$ for $i\in I$, where
$X_i,Z_i\in \sv$, $Z_i$ is closed everywhere of codimension $r$ in
$X_i$ for all $i\in I$.

We take $Y_i=j(X_i\setminus Z_i\to X_i)$, $Y=p(\prli_{i\in I} Y_i)$.
By parts \ref{ip7} and \ref{ip8} of Proposition \ref{pprop} we have a
distinguished triangle $\mg(X\setminus Z)\to \mg(X)\to Y$.

It remains to prove that $Y\cong \mg(Z)(r)[2r]$. Proposition 2.4.5 of
\cite{deggenmot} (a functorial form of the Gysin distinguished
triangle for Voevodsky's motives) yields that $p(Y_i)\cong
\mg(Z_i)(r)[2r]$; moreover, the connecting morphisms $p(Y_i)\to
p(Y_{i+1})$ are obtained from the corresponding morphisms
$\mg(Z_i)\to \mg(Z_{i+1})$ by  tensoring by $\z(r)[2r]$. It remains
to recall:  by Proposition \ref{pprop}(\ref{ip6}), the
isomorphism class of a homotopy limit in $\gd$ can be completely
described in terms of (objects and morphisms) of $\gd$ (i.e. we
don't have to consider the lifts of objects and morphisms to
$\gdp$). This yields the result.

Now, since $\mg(X\times G_m)=\mg(X)\bigoplus \mg(X)(1)[1]$ for any
$X\in \sv$ (hence this is also true for pro-schemes), the assertion
for the case $s=0$ yields the general case easily.
\end{proof}

Now we will construct a certain Postnikov tower $Po(X)$  for $X$
being the (twisted) comotif of a pro-scheme $Z$ that will be related
to the coniveau spectral sequences (for cohomology) of $Z$; our
method was described in \S\ref{sprom} above.
 Note that we consider the general case
of an arbitrary pro-scheme $Z$ (since in this paper pro-schemes play
an important role);  yet this case is not much distinct from the
(partial) case of $Z\in\sv$.

\begin{coro}\label{gepost}

We denote the dimension of $Z$ by $d$ (recall the
conventions of \S\ref{spsch}).

For all $i\ge 0$ we denote by $Z^i$ the set of
points of $Z$ of codimension $i$.

For any $s\ge 0$ there exists a Postnikov tower
for $X=\mg(Z)(s)[s]$ such that $l=0$, $m=d+1$, $X_i\cong \prod
_{z\in Z^i}\mg(z)(i+s)[2i+s]$.
\end{coro}
\begin{proof}

As above, it suffices to prove the statement for $s=0$. Since any
product of distinguished triangles is distinguished, we can assume
$Z$ to be connected.

We consider a projective system $L$
whose elements are sequences of closed subschemes
$\varnothing=Z_{d+1}\subset Z_d\subset Z_{d-1}\subset \dots \subset
Z_0$. Here $Z_0\in\sv$, $Z_l\in \var$ for $l>0$, $Z$ is open in
$Z_0$ (see \S\ref{spsch}; $Z_0$ is connected; in the case when $Z\in
\sv$
 we only take $Z_0=Z$); for all $j>0$ we have: $Z_{j}$
is everywhere of codimension  $\ge j$ in $Z_0$; all irreducible
components of all $Z_j$ are everywhere
  of codimension $\ge j$ in $Z_0$; and
 $Z_{j+1}$ contains the singular locus of $Z_j$ (for $j\le d$).
The  ordering in $L$ is given by open
 embeddings of varieties $U_j=Z_0\setminus Z_j$ for $j>0$.
For $l\in L$ we will denote the corresponding sequence by
$\varnothing=Z^l_{d+1}\subset Z^l_d\subset Z^l_{d-1}\subset
\dots \subset Z^l_0$. Note  that
$L$ is countable!

By the previous proposition, for any $j$ we have a distinguished
triangle $\mg (\prli(Z^l_0\setminus Z^l_j))\to
\mg(\prli(Z^l_0\setminus Z^l_{j+1}))\to \mg (\prli(Z^l_j\setminus
Z^l_{j+1})(j)[2j])$.

It remains to compute the last term; we fix some $j$.

We have $\prli_{l\in L'} (Z^l_j\setminus Z^l_{j+1}))\cong \prod _{z\in Z^i}\mg(z)$.
 Indeed, for all $l\in L$ the variety
 $Z^l_j\setminus Z^l_{j+1}$ is the disjoint
 union of some locally closed smooth subschemes of
$Z^l_0$ of codimension $j$;
 for any $z_0\in Z^j$ for $l\in L$ large enough $z_0$
is contained in
 $Z^l_j\setminus Z^l_{j+1}$ as an open sub-pro-scheme,
and the inverse limit
 of connected components of $Z^l_j\setminus Z^l_{j+1}$
containing $z_0$ is exactly $z_0$.
Now, we can apply the functor $X\mapsto \mg(X)(j)[2j]$ to this
isomorphism.
We obtain  $\mg (\prli(Z^l_j\setminus Z^l_{j+1})(j)[2j]) \cong
   \prod _{z\in Z^i}\mg(z)(i)$.
This yields the result.

\end{proof}

\begin{rema}\label{lger} 

 1. Alternatively, one could construct 
 $Po(X)$ for
 the (twisted) comotif of a pro-scheme $T=\prli T^l$ as the inverse
  limit of the Postnikov towers for $T^l$ (constructed as above yet
with fixed $Z_0^l=T^l$);
  certainly, to this end one should pass to the limit in $\gdp$. It is easily
 seen that one
   would get the same tower this way.

2. Certainly, if we shift a Postnikov tower for
 $\mg(Z)(s)[s]$ by $[j]$ for some $j\in \z$,
we obtain a Postnikov tower for $\mg(Z)(s)[s+j]$. We didn't
formulate assertion 2
 for these shifts only because we wanted $X^p$ to belong to $\gds^{w=0}$
 (see Proposition \ref{pexw} below).

3. Since the calculation of $X^i$ used Proposition \ref{pprop}(\ref{ip6}), our method cannot describe
connecting
morphisms between 
them (in $\gd$). Yet one can calculate the 'images' of the
connecting morphisms in $\gdn$;
 see \S\ref{sprom} and \S\ref{scompdegl}.

\end{rema}

\section{Main motivic results}  \label{sapcoh}

The results of the previous section combined with those of \S\ref{sbpws}   allow us
 to construct (in \S\ref{sdgersten}) a certain {\it Gersten
weight structure} $w$ on
 a certain triangulated $\gds$: $\dmge\subset\gds\subset \gd$.
Its main property is that
 the comotives of function fields over $k$
 (and their products) 
  belong to  $\hw$. It follows immediately that
the Postnikov tower $Po(X)$
  provided by Corollary
\ref{gepost} is a
  {\it weight Postnikov tower} with respect to $w$.
Using this, in \S\ref{stds}  we prove:
if $S$ is a primitive scheme,  $S_0$ is its dense sub-pro-scheme, then $\mg(S)$ is a direct summand of $\mg(S_0)$;
$\mg(K)$ (for a function field $K/k$)
 contains (as retracts)
    the comotives of primitive schemes whose generic point
is $K$,
    as well as the twisted comotives of residue
 fields of $K$ (for all geometric valuations).

In \S\ref{sext} we  (easily) translate these results to cohomology; in particular,
the cohomology of (the spectrum of)
    $K$ contains direct summands
    corresponding to the cohomology of primitive schemes whose
generic point is $K$,
    as well as twisted cohomology of residue
 fields of $K$. Here one can consider any cohomology
 theory $H:\gds\to \au$; one can obtain such an $H$ by
 extending to $\gds$ any cohomological $H':\dmge\to \au$ if
 $\au$ satisfies AB5 (by means of Proposition \ref{pextc}).
 Note: in this case the cohomology of
 pro-schemes mentioned
 is calculated in the 'usual' way.

In \S\ref{sdconi} we consider  weight spectral sequences corresponding
to (the Gersten weight structure) $w$. We observe that these
spectral sequences generalize naturally  the classical coniveau spectral
sequences. Besides, for a fixed $H:\gds\to \au$ our (generalized) coniveau
spectral sequence converging to $H^*(X)$ (where $X$ could be a
motif or just an object of $\gds$) is $\gds$-functorial in $X$ (i.e.
it is motivically functorial for objects of $\dmge$); this
 fact is non-trivial even when restricted to motives of smooth varieties.

In \S\ref{sconi} we prove that there exists a nice duality
$\gd^{op}\times \dme \to \ab$ (extending the bi-functor
$\dme(-,-):\dmge^{op}\times \dme\to \ab$); the Gersten weight
structure $w$ (on $\gds$) is left orthogonal to
 the homotopy $t$-structure $t$ on $\dme$ with respect to it.
 This allows to apply
  Theorem \ref{tdual}: in the case when $H$
 comes from $Y\in\obj\dme$ we prove the isomorphism
 (starting from $E_2$) of (the coniveau)  $T(H,X)$
with the spectral sequence corresponding to the  $t$-truncations of $Y$.
 We  describe
$\obj \dmge\cap \gds^{w\le i}$ in terms of $t$ (for $\dme$).
We also note that our results allow to describe
torsion motivic cohomology in terms of (torsion)
\'etale cohomology (see  Remark \ref{rconiv}(4)).

In \S\ref{sbchcon} we define the coniveau spectral sequence (starting from $E_2$)
for cohomology of a motif $X$ over a not (necessarily) countable perfect base field $l$ as the limit of the corresponding coniveau spectral sequences over countable perfect subfields of definition for $X$.
This definition is compatible with the classical one (for $X$ being the motif of a smooth variety); so we obtain motivic functoriality of classical coniveau spectral sequences over a general base field.

In \S\ref{stchow} we prove that the Chow weight structure for
$\dmge$ (introduced in \S6 of \cite{bws}) could be extended to $\gd$
(certainly, the corresponding weight structure $w_{\chow}$ differs
from $w$). We will call the corresponding weight spectral sequences {\it Chow-weight} ones; note that they are isomorphic to  classical (i.e. Deligne's) weight spectral sequences when the latter are defined.

In \S\ref{scchowco}
 we use the results \S\ref{schws}
 to compare Chow-weight spectral sequences with  coniveau  ones. We always have a comparison morphism; it is an isomorphism if $H$ is a {\it birational} cohomology theory.

 In \S\ref{sgdbr} we consider the category of birational comotives $\gdbr$ (a certain 'completion' of birational motives of \cite{kabir}) i.e. the localization of $\gd$ by $\gd(1)$.
 It turns our that $w$ and $w_{\chow}$ induce the same weight structure $w'_{bir}$ on $\gdbr$. Conversely, starting from $w'_{bir}$ one can glue 'from slices' the weight structures induced by
  $w$ and $w_{\chow}$ on $\gd/\gd(n)$ for all $n>0$. Furthermore, these structures belong to an interesting family of weight structures indexed by a single integral parameter; other terms of this family could be also interesting!

\subsection{The Gersten weight structure for
$\gds\supset\dmge$} \label{sdgersten}

Now we describe the main weight structure of this paper.
Unfortunately, the author does not know whether it is possible
 to define the Gersten weight structure (see below) on the whole $\gd$. Yet for
our purposes it is quite sufficient to define the corresponding
weight structure on a certain triangulated subcategory $\gds\subset
\gd$ containing $\dmge$ (and the comotives of all pro-schemes).

In order to make 
the choice of $\gds\subset \gd$ compatible with extensions of
scalars, we bound certain dimensions of objects of $\hw$.

We will denote by $H$ the full subcategory of $\gd$ whose
objects are all 
countable 
products $\prod_{l\in L} \mg(\spe K_l)(n_l)[n_l]$;
 here $K_l$ are
 (the spectra of) function fields over $k$, $n_l\ge
0$; we assume that the transcendence degrees of $K_l/k$ and $n_l$
are bounded.

\begin{pr}\label{pexw}
1. Let $\gds$ be the Karoubi-closure of $\lan H\ra$ in $\gd$.
Then $\cu=\gds$ can be endowed with a unique weight structure $w$
such that $\hw$ contains $H$.

2. $\hw$ is the idempotent completion of $H$.

3. $\gds$ contains $\dmge$ as well as all $\mg(Z)(l)$ for $Z$
 being a pro-scheme, $l\ge 0$.

4. For any primitive $S$, $i\ge 0$, we have $\mg(S)(i)[i]\in
\gds^{w=0}$.

5. Let $Z$ be a pro-scheme, $s\ge 0$. Then $\mg(Z)(s)[s]\in
\gds^{w\le 0}$;  the Postnikov tower for $\mg(Z)(s)[s]$ given by
Corollary \ref{gepost} is a weight Postnikov tower for it.

\end{pr}
\begin{proof}
1. By  Proposition \ref{maneg}(2), $H$ is negative (since any object of $H$ is a finite sum of $\mg(X_i)(m_i)$ for some primitive pro-schemes $X_i$, $m_i\in \z$). 
Besides, $\gd$ is idempotent complete (see Corollary \ref{css}(4)); hence $\gds$ and $\gds^{w=0}$ also are.
 Hence we can apply Theorem \ref{tbw}(\ref{igen}) (for $D=H$).

2. Also immediate from Theorem \ref{tbw}(\ref{igen}). 

3. $\mg(Z)(l)\in \obj \gds$ by Corollary \ref{gepost}; in
particular, this is true for $Z\in \sv$. It remains to note that
$\dmge$ is the Karoubi-closure of $\lan \mg(U):\ U\in \sv\ra$
in  $\gd$.

4. It suffices to note that $\mg(S)(i)[i]$ belongs both
to $\gds^{w\le 0}$ and to $\gds^{w\ge 0}$ by Theorem \ref{tbw}(\ref{iwgene}). Here we use  
 Proposition \ref{maneg}(2) again.

5. We have $X^i\in \gds^{w=0}$. Hence Theorem
\ref{tbw}(\ref{iwpostc}) yields the result. Note  here that we have $Y_{0}=0$
 in the notation  of Definition \ref{d2}(9). 

\end{proof}

We will call $w$ the {\it Gersten} weight structure, since it is
closely connected with Gersten resolutions of cohomology (cf.
\S\ref{sconi} below). 
By default, 
below $w$ will denote
the Gersten weight structure.

\begin{rema}\label{rcger}

1. $\hw$ is idempotent
 complete since $\gds$ is.

2. In fact, one could easily prove similar statements for $\cu$ being
just $\lan H\ra$ (instead of its Karoubi-closure in $\gd$). Certainly,
for this version of $\cu$ we will only have $\cu\supset \mg(K^b(\smc))$.

Besides, note that for any function field $K'/k$, any $r\ge 0$, there
exists a function field $K/k$ such that
$\mg(\spe K')(r)[r]$ is a retract of $\mg(\spe K)$ (see Corollary \ref{tds2} below).
Hence 
it suffices take $H$ being the full subcategory of $\gd$
whose objects are
  $\prod_{l\in L} \mg(\spe K_l)$ (for bounded transcendence degrees of $K_l/k$).

3. The proposition 
implies that $\gds$ is exactly the Karoubi-closure in $\gd$ of the
triangulated category generated by the comotives of all pro-schemes.

4. The author does not know whether one can 
describe weight
decompositions for arbitrary objects of
$\dmge$ explicitly. Still, one can say something about these weight
decompositions and
weight complexes using their functoriality properties.
In particular, knowing
weight complexes for $X,Y\in \obj \dmge$
(or just $\in \obj \gds$) and $f\in \gds(X,Y)$ one can
describe the weight complex of $\co(f)$  up to a homotopy
equivalence as the
corresponding cone (see Lemma \ref{lwc} below). Besides,
let $X\to Y\to Z$
be a distinguished triangle (in $\gd$). Then for any choice of
 $(X^{w\le 0}, X^{w\ge 1})$ and $(Z^{w\le 0}, Z^{w\ge 1})$
there exists a choice
of $(Y^{w\le 0}, Y^{w\ge 1})$ such that there exist
distinguished triangles
 $X^{w\le 0}\to Y^{w\le 0}\to Z^{w\le 0}$ and
$X^{w\ge 1}\to Y^{w\ge 1}\to Z^{w\ge 1}$;
 see Lemma 1.5.4 of \cite{bws}. In particular, we obtain
that $j$ maps complexes
(over $\smc$) concentrated in degrees $\le j$ into
$\gds^{w\le j}$
(we will prove a stronger statement in  Remark
\ref{rconiv}(4) below).
If $X\in \obj \dmge$ comes from a complex over $\smc$
whose connecting morphisms  satisfy certain
 codimension restrictions, these observations could be
extended to an explicit
description of a weight decomposition for it; cf. \S7.4 of \cite{bws}.

\end{rema}

\subsection{Direct summand results for  comotives}\label{stds}

Proposition \ref{pexw} easily implies the following
interesting result.

\begin{theo}\label{tds1n}

1. Let $S$ be a primitive scheme; let $S_0$ be its dense sub-pro-scheme.
Then $\mg(S)$ is a direct summand of $\mg(S_0)$.

2. Suppose moreover that $S_0=S\setminus T$ where  $T$ is a closed
subscheme of $S$ everywhere of codimension $r>0$. Then we have $\mg(S_0)\cong
\mg(S) \bigoplus \mg(T)(r)[2r-1]$.

 \end{theo}
\begin{proof}

We can assume that $S$ and $S_0$ are connected.

1. By Proposition \ref{pexw}(5), we have: $\mg(S_0),\mg(S)\in
\gds^{w\le 0}$; $\mg(\spe(k(S)))$ could be assumed to be the zeroth term
of their weight complexes for a choice of weight complexes compatible with some  negative Postnikov weight towers for them;
 the embedding $S_0\to S$ is compatible with
 $\id_{\mg(\spe(k(S)))}$ (since we have a commutative triangle $\spe k(S)\to S_0\to S$ of pro-schemes).
 Hence    Theorem \ref{tbw}(\ref{impostp}) yields the result.

2. By Proposition \ref{pinfgy} we have a distinguished triangle
$\mg(S_0)\to \mg(S)\to \mg(T)(r)[2r]$. By parts 4 and 5 of Proposition
\ref{pexw} we have $\mg(S_0)\in \gds^{w\le 0}$, $\mg(S)\in \gds^{w= 0}$,
$\mg(T)(r)[2r]\in \gds^{w\le -r}\subset \gds^{w\le -1}$. Hence 
 Theorem \ref{tbw}(\ref{isump}) yields the result.

\end{proof}

\begin{coro}\label{tds1} \label{tds2}

1. Let $S$ be a connected primitive scheme, let $S_0$ be its generic
point.
Then $\mg(S)$ is a retract of $\mg(S_0)$.

2. Let $K$ be a function field  over $k$. Let $K'$ be
the residue field for a geometric valuation $v$ of $K$ of rank $r$.
Then $\mg(\spe K')(r)[r]$ is a retract
of $\mg(\spe K)$.

\end{coro}
\begin{proof}
 1. This is just a partial case of part 1 of  the previous theorem.

2. Obviously, it suffices to prove the statement in the case $r=1$.
Next,  $K$ is the function field  of some normal projective variety over
$k$. Hence there exists a $U\in \sv$ such that: $k(U)=K$, $v$ yields
a non-empty closed subscheme of $U$ (since the singular locus has
 codimension
$\ge 2$ in a normal variety). It easily follows that there exists a
pro-scheme $S$ (i.e. an inverse limit of smooth varieties) whose
only points are the spectra of $K$ and $K_0$. So, $S$ is local,
hence it is primitive.

By part 2 of the theorem, we have
$$\mg(\spe K)=\mg(S)\bigoplus \mg(\spe K_0)(1)[1];$$
this concludes the proof.

\end{proof}

\begin{rema}\label{rstds} 
1. Note that we do not construct any explicit splitting morphisms in
the decompositions above. Probably, one cannot choose any canonical
splittings here (in the general case); so there is no (automatic)
compatibility for any pair of related decompositions. Respectively,
though the comotives of (spectra of) function fields contain tons of
direct summands, there seems to be no general way to decompose them
into indecomposable summands.

2. Yet Proposition \ref{pinfgy} easily yields that
$\mg(\spe k(t))\cong \z\bigoplus \prod_E\mg(E)(1)[1]$; here $E$
runs through all closed points of $\af^1$ (considered as a scheme over $k$). 

\end{rema}

\subsection{On cohomology
of pro-schemes, and its direct summands}\label{sext}

The results 
proved above immediately imply 
similar assertions for cohomology. 
We also construct a class of cohomology theories that respect
 homotopy limits. 

\begin{pr}\label{cdscoh}
Let $H:\gds\to \au$ be cohomological, $S$ be a primitive scheme.

1. Let $S_0$ be a dense sub-pro-scheme of  $S$. Then $H(\mg(S))$
  is a direct summand of $H(\mg(S_0))$.

2. Suppose moreover that $S_0=S\setminus T$ where  $T$ is a
closed subscheme of $S$ of codimension $r>0$. Then we have
$H(\mg(S_0))\cong H(\mg(S))
\bigoplus H(\mg(T)(r)[2r-1])$.

3. Let $S$ be connected, $S_0$ be the generic point of $S$.
Then $H(\mg(S))$ is a retract of $H(\mg(S_0))$ in $\au$.

4. Let $K$ be a function field  over $k$. Let $K'$
be the residue field for a geometric valuation $v$ of $K$ of rank
$r$.
Then $H(\mg(\spe K')(r)[r])$ is a retract of $H(\mg(\spe K))$ in $\au$.

5. Let $H':\dmge\to \au$ be
a cohomological functor, let $\au$ satisfy AB5. Then Proposition
 \ref{pextc}
allows to extend $H'$
to a cohomological functor $H:\gd\to \au$ that converts   inverse
limits  in $\gd'$ to the corresponding direct limits in $\au$.

\end{pr}
\begin{proof}

1. Immediate  from
Theorem \ref{tds1n}(1).

2. Immediate  from
 Theorem \ref{tds1n}(2).

3. Immediate  from
 Corollary \ref{tds1}(1).

4. Immediate  from
 Corollary \ref{tds2}(2).

5. Immediate from Proposition \ref{pextc}; note that $\dmge$
is skeletally small. Here in order to prove that $H$ converts
homotopy limits into direct limits we use part I2 of  loc. cit.
and  Proposition \ref{pprop}(\ref{ipcf}).

\end{proof}

\begin{rema}\label{rcohp}

1.  In the setting of assertion 5 we will call  $H$ an
{\it extended} cohomology theory.

Note that for $H'=\dmge(-,Y)$, $Y\in \obj \dmge$, we have $H=\gd(-,Y)$; see (\ref{ekrause}).


2. Now recall that for any pro-scheme $Z$, any  $i\ge 0$,
 $\mg(Z)(i)$ (by
definition) could be presented as a countable homotopy limit of geometric motives.
Moreover, the same is true for all  small countable products of $\mg(Z_l)(i)$.
Hence if $H$ is extended, then the cohomology of $\prod \mg(Z_l)(i)$
is the corresponding
direct limit; this coincides with the  definition given by (\ref{dfps}) (cf. Remark \ref{rspsch}).

In particular, one can apply the results of Proposition
 \ref{cdscoh} to
the usual \'etale cohomology of pro-schemes mentioned
(with values in $\ab$ or in some category of Galois modules).

3. 
 If $H'$ is also a tensor functor (i.e. it converts
tensor product in $\dmge$ into tensor products in $D(\au)$), then
certainly the cohomology of $\mg(\spe K')(r)[r]$ is the corresponding
tensor product of $H^*(\mg(\spe K'))$ with $H^*(\z(r)[r])$. Note that the
latter one is a retract of $H^*(G_m^r)$; 
we obtain the Tate
twist for cohomology this way.
\end{rema}

\subsection{Coniveau spectral sequences for cohomology of (co)motives}
\label{sdconi}

Let $H:\gds^{op}\to \au$ be a cohomological functor, $X\in \obj \gds$.


\begin{pr} \label{rwss}

1. Any choice of  a weight spectral sequence $T(H,X)$
(see Theorem \ref{pssw}) corresponding to the Gersten weight
 structure $w$ is canonical and $\gds$-functorial in
 $X$ starting from $E_2$.

2. $T(H,X)$ converges to $H(X)$.

3. Let $H$ be an extended theory (see Remark \ref{rcohp}),
 $X=\mg(Z)$ for $Z\in \sv$.
Then any choice of $T(H,X)$ starting from $E_2$ is canonically
isomorphic to the classical coniveau spectral sequence (converging
to the $H$-cohomology of $Z$; see  \S1 of \cite{suger}).

\end{pr}
\begin{proof}
1. This is just a partial case of  Theorem \ref{pssw}(I).

2. Immediate since $w$ is bounded;
see part I2 of loc. cit.

 3. Recall that in the proof of  Corollary \ref{gepost} a  
   Postnikov tower $Po(X)$ for $X$
 was obtained from certain 'geometric' Postnikov towers
(in $j(C^b(\smc))$) by passing to the homotopy limit. Now,  the
coniveau spectral sequence (for the $H$-cohomology of $Z$)
in \S1.2 of \cite{suger} was constructed by applying $H$ to
the same geometric towers 
and then passing to the inductive limit (in $\au$). Furthermore,
  Remark \ref{rcohp}(2) yields that the latter limit is
 (naturally) isomorphic to 
the spectral sequence obtained via  $H$ from $Po(X)$. 
Next, since $Po(X)$ is a weight Postnikov tower for $X$ (see 
Proposition \ref{pexw}(5)), we obtain that the latter spectral
 sequence  is one of the possible choices for $T(H,X)$.

 Lastly, assertion 1 yields that all other possible
$T(H,X)$ (they depend on the choice of a weight
 Postnikov tower for $X$) starting from $E_2$ are also
canonically isomorphic to the classical coniveau spectral
sequence mentioned.

\end{proof}

\begin{rema}\label{rrwss}

1. Hence we proved (in particular) that  classical 
 coniveau
spectral sequences (for cohomology theories that can be factorized through motives; this includes \'etale and singular cohomology of smooth varieties) are
$\dmge$-functorial (starting from $E_2$); we also obtain such
 a functoriality for the coniveau filtration for cohomology!
 These facts are far from being obvious from the usual definition of
 the coniveau filtration and spectral sequences, and seem to be
 new (in the general case). So, we justified the title of the paper.

We also obtain certain coniveau spectral sequences for cohomology
  of singular
varieties (for cohomology theories that can be factorized
through $\dmge$; in the case  $\cha k>0$ one also needs rational
 coefficients here).

2. Assertion 3 of the proposition yields a nice reason to
call (any choice of)   $T(H,X)$ a
{\it coniveau spectral
sequence} (for a general $H,\au$, and $X\in \obj \gds$); this
will also distinguish (this version of) $T$ from weight spectral
 sequences corresponding to other weight structures. We will
 give more justification for this term in Remark \ref{rconiv} below.
  So, the corresponding filtration
could be called the (generalized) coniveau filtration.

\end{rema}

\subsection{An extension of  results of Bloch and Ogus}\label{sconi}

Now we want  to relate  coniveau spectral sequences with the
homotopy $t$-structure (in $\dme$). This would be a
vast extension of the
seminal results of \S6 of \cite{blog} (i.e. of the calculation
by Bloch and
Ogus of the $E_2$-terms of coniveau spectral sequences)
and of \S6 of \cite{ndegl}.

We should relate $t$ (for $\dme$) and $w$; it turns out that
 they are orthogonal
 with respect to a certain quite natural nice duality.

\begin{pr}\label{pdualn}

 For any $Y\in \obj \dme$ we extend $H'=\dme(-,Y)$ from $\dmge$ to
$\gd\supset \gds$ by the method of Proposition \ref{pextc}; we
define $\Phi(X,Y)=H(X)$. Then the following statements are valid.

1. $\Phi$ is a nice duality (see Definition \ref{ddual}).

2 $w$
 is {\it left orthogonal} to the homotopy $t$-structure
  $t$ (on $\dme$) with respect to $\Phi$.

 3. $\Phi(-,Y)$ converts homotopy limits (in $\gdp$) into direct
 limits in $\ab$.

\end{pr}
\begin{proof}
1. By  Proposition \ref{pnice}(1), the restriction of $\Phi$
to $\dmge{}^{op}\times \dme$ is a nice duality. It remains to apply
part 3 of loc. cit.

2. In the notation of Proposition \ref{pdual}, we take for $D$ the
set of all small products
 $\prod_{l\in L} \mg(\spe K_l)(n_l)[n_l]\in \obj \gds$; here $\mg(\spe K_l)$
denote the comotives of (spectra of) some function fields  over $k$,
$n_l\ge 0$ and the transcendence degrees of $K_l/k$ 
are bounded (cf. \S\ref{sdgersten}). Then $D,\Phi$ satisfy
the assumptions of the proposition by  Proposition
\ref{porthog}(2) (see also Remark \ref{rcohp}(2)).

3. Immediate from Proposition \ref{cdscoh}(3).

\end{proof}

\begin{rema}\label{rhocolim}
1. Suppose that we have an inductive family $Y_i\in \obj\dme$ connected
by a compatible family of morphisms with some $Y\in \dme$ such that:
for any $Z\in \obj \dmge$ we have $\dme(Z,Y)\cong\inli \dme(Z,Y_i)$
(via these morphisms $Y_i\to Y$). In such a situation it is
reasonable to call $Y$ a homotopy colimit of $Y_i$.

The definition of $\Phi$ in the proposition easily implies: for any
$X\in \obj \gd$ we have $\Phi(X,Y)=\inli \Phi(X,Y_i)$. So, one may
say  that all objects of $\gd$ are 'compact with respect to $\Phi$',
whereas part 3 of the proposition yields that all objects of $\dme$
are 'cocompact with respect to $\Phi$'. Note that no analogues of
these nice properties can hold in the case of an adjacent weight and
$t$-structure (defined on a single triangulated category).

2. Now, we could have replaced $\dmge$ by $\dmgm$ everywhere in the 'axiomatics' of $\gd$ (in Proposition \ref{pprop}). Then the corresponding category $\gd_{gm}$ could be used for our purposes (instead of $\gd$), since our arguments work for it also. Note that we can extend $\Phi$ to a nice duality  $\gd_{gm}^{op}\times \dme\to \ab$; to this end it suffices for $Y\in \obj \dme$ to extend $H'$ to $\dmgm$ in the following way:  $H'(X(-n))=\dme (X,Y(n))$ for $X\in \obj \dmge\subset \obj \dmgm$, $n\ge 0$.

Moreover, the methods of \S\ref{snpttw} allow to define an invertible Tate twist functor on $\gd_{gm}$.

\end{rema}

\begin{coro}\label{ccompss}

1. If
$H$ is represented by a $Y\in \obj\dme$ (via our $\Phi$) then for a
(co)motif $X$ our coniveau spectral sequence $T(H,X)$ starting from
$E_2$ could be naturally expressed in terms of the cohomology of $X$
with coefficients in
$t$-truncations of $Y$ (as in Theorem \ref{tdual}). 

In particular, the coniveau filtration
for $H^*(X)$ could be described as in part \ref{ispostn}
of loc. cit.

2. For $U\in \obj \dmge$, $i\in \z$, we have
$U\in \gds^{w\le i}\iff U\in \dme^{t\le i}$.
\end{coro}

\begin{proof}

1. Immediate from Proposition \ref{pdualn}.

2. By    Theorem \ref{tbw}(\ref{iwgene}),
we should check whether $Z\perp U$
  for any
$Z=\prod_{l\in L} \mg(\spe K_l)(n_l)[n_l+r]$, where $K_l$
are function fields over $k$,
$n_l\ge 0$ and the transcendence degrees of $K_l/k$ 
are bounded, $r>0$
(see  Proposition \ref{pexw}(2)).  Moreover,
since $U$ is cocompact in $\gd$,
it suffices to consider $Z=\mg(\spe K')(n)[n+r]$ ($K'/k$
is a function field, $n\ge 0$).
Lastly,  Corollary \ref{tds2}(2) reduces the situation to the case $Z=\mg(\spe K)$
($K/k$ is a function field).

Hence (\ref{ehomhw}) implies: $U\in \gds^{w\le i}$
whenever for any $j>i$,
any function field $K/k$, the stalk of $U^{t=j}$ at
$K$ is zero. Now, if $U\in \dme^{t\le i}$ then
 $U^{t=j}=0$ for all $j>i$; hence all stalks
of $U^{t=j}$ are zero. Conversely, if all
stalks of $U^{t=j}$ at function fields are zero, then
Corollary 4.19 of \cite{3} yields $U^{t=j}=0$
(see also Corollary 4.20 of loc. cit.); if $U^{t=j}=0$
for all $j>i$ then $U\in \dme^{t\le i}$.

\end{proof}

\begin{rema}\label{rdualn} \label{rconiv} 

1. Our
comparison statement is  true for $Y$-cohomology of an
arbitrary $X\in \obj \dmge$; this  extends to motives Theorem
6.4 of \cite{ndegl} (whereas the latter essentially extends the
results of \S6 of \cite{blog}). We obtain one more reason  to call
$T$ (in this case) the coniveau spectral sequence for
(cohomology of) motives.

Note also that the methods of Deglise do not (seem to) yield the motivic functoriality of the isomorphism in question (cf. Remark \ref{rdual}).

2. If $Y\in \obj HI$, then $E_2(T)$ yields the
Gersten resolution
for $Y$  (when $X$ varies); this is why we called $w$ the
Gersten weight structure.

3. Now, let $Y$ represent \'etale cohomology with
coefficients in $\z/l\z$,
 $l$ is prime to $\operatorname{char}k$ ($Y$ is actually
unbounded from above,
  yet this is not important). Then the $t$-truncations of
 $Y$ represent $\z/l\z$-motivic
   cohomology by the (recently proved) Beilinson-Lichtenbaum
conjecture (see \cite{vobe}; this paper is not published at the moment).
   Hence  Proposition \ref{pvirdu}(1) yields some new formulae
    for
   $\z/l\z$-motivic    cohomology of $X$ and for the
   'difference' between \'etale and motivic cohomology. Note also
   that the virtual $t$-truncations (mentioned in loc. cit.) are exactly
   the $D_2$-terms of the alternative exact couple for $T(H,X)$
   and for the version of the exact couple used in the current paper respectively (i.e. we consider exact    couples coming from the two possible versions for a  weight Postnikov
    tower for $X$, as described in  Remark \ref{rpost}). See also
\S7.5 of \cite{bws}
   for more explicit results of this sort. 
   It could also be
interesting
   to study coniveau spectral sequences for singular
cohomology; this could
   yield a certain theory of 'motives up to algebraic equivalence';
see  Remark 7.5.3(3)
    of loc. cit. for more details.

    5. Assertion 2 of the corollary yields that
    $\gds^{w\le 0}\cap \obj\dmge$ is large enough to recover $w$ (in a certain sense); in
particular, this assertion is similar to the definition of adjacent
structures (see Remark \ref{rnice}). In contrast, $\gds^{w\ge 0}\cap
\obj\dmge$ seems to be too small.

\end{rema}

\subsection{Base field change for coniveau spectral sequences;
functoriality for an uncountable $k$}\label{sbchcon}

It can be easily seen (and well-known) that for any perfect field extension
$l/k$ there exist an extension of scalars functor  $\dmge{}_k\to
\dmge{}_l$ compatible with the extension of scalars for smooth
varieties (and for $K^b(\smc)$).
In \ref{sexre} below 
we will prove that this functor could be expanded to a functor
$\exts_{l/k}:\gd_k\to \gd_l$ that sends $\mg{}_{,k}(X)$ to
$\mg{}_{,l}(X_l)$ for a pro-scheme $X/k$; this extension procedure
is functorial with respect to embeddings of base fields. Moreover,
$\exts_{l/k}$ maps $\gds{}_k $ into $\gds{}_l$. Note the existence
of base change for comotives does not follow from the properties of
$\gd$ listed in Proposition \ref{pprop}; yet one can define base
change for our model of comotives (described in \S\ref{rdg} below)
and (probably) for any other possible reasonable version of $\gd$.

Now we prove that base change for comotives yields base change for
coniveau spectral sequences; it also allows to prove that these
spectral sequences are motivically functorial for not necessary
countable base fields.

In order to make the limit in Proposition \ref{punbf}(2) below well-defined, we
assume that for any $X\in \obj \dmge$ there is a fixed
representative $Y,Z,p$ chosen, where: $Z,Y\in C^b(\smc)$,
$\mg(Y)\cong \mg(Z)$, $p\in C^b(\smc)(Y,Z)$ yields a direct summand
of $\mg(Y)$ in $\dmge$ that is isomorphic to $X$. We also assume
that all the components of $(X,Y,p)$  have  fixed expressions in
terms of algebraic equations over $k$; so one may speak about  fields of definition for $X$.

\begin{pr}\label{punbf}

Let $l$ be a perfect field, $H:\gd_l\to \au$ be any cohomological
functor (for an abelian $\au$). For any perfect $k\subset l$ we denote $H
\circ \exts_{l/k}:\gd_k\to \au$ by $H_k$.

1. Let $l$ be  countable. 
Then for any $X\in \obj \gd_k$ the method of Proposition \ref{pcwspsq}(II) 
yields some morphism
$N_{l/k}:T_{w_k}(H_k,X)\to T_{w_l}(H,\exts_{l/k}(X))$; this morphism
is unique and $\gd_k$-functorial in $X$ starting from $E_2$.

The correspondence $(l,k)\mapsto N_{l/k}$ is associative with
respect to extensions of countable fields (starting from $E_2$); cf. part I3 of loc. cit.

2. Let $l$ be a not (necessarily) countable perfect field, let $\au$ satisfy AB5.

For $X\in \obj \dmge_l$ we define $T_w(H,
X)=\inli_{k}T_{w_k}(H_k,X_k)$. Here we take the limit with respect
to all  perfect $k\subset l$ such that $k$  is countable, $X$ is defined over $k$; the connecting morphisms are given by the maps $N_{-/-}$ mentioned in assertion 1;
 we
start our spectral sequences from $E_2$. Then $T_w(H, X)$ is a
well-defined spectral sequence that is  $\dmge_l$-functorial in $X$.

3. If $X=\mg{}_{,l}(Z)$, $Z\in \sv$, $H$ is as an extended theory,
and $\au$ satisfies AB5, the spectral sequence given by the previous
assertion is canonically isomorphic to  the
  classical coniveau spectral sequence (for $(H,Z)$; considered starting from $E_2$).


\end{pr}
\begin{proof}

1. By  Proposition \ref{pcwspsq}(II)
  it suffices to check that $\exts_{l/k}$ is left weight-exact
   (with respect to weight structures in question).
We take $D$ being the class of all small products $\prod_{l\in L}
\mg(\spe K_l)$,  where $\mg(\spe K_l)$ denote the comotives of (spectra of)
function fields over $k$ of bounded transcendence degree.
Proposition \ref{pexw} and   Corollary \ref{tds2}(2) yield
that any $X\in \gds{}_k^{w=0}$ is a retract of some element of $D$.
It suffices to check that  for any $X=\prod_{l\in L}
\mg{}_{,k}(K_l)$ we have  $\exts_{l/k}X\in  \gds{}_l^{w_l\le 0}$;
here we recall that $w_k$ is bounded and apply Lemma \ref{lrweex}.

Now, $X$ is the comotif of a certain pro-scheme, 
hence the same is true for $\exts_{l/k}X$. It remains to apply
  Proposition \ref{pexw}(5).

2. By the associativity statement in the previous assertion, the
limit is well-defined. Since $\au$ satisfies AB5, we obtain a
spectral sequence indeed.
  Since we have
  $k$-motivic functoriality of coniveau spectral sequences
  over each $k$, we obtain $l$-motivic functoriality in the limit. 

3. Again (as in the proof of  Proposition \ref{rwss}(3)) we
recall that the classical coniveau spectral sequence for this case is
defined by applying $H$ to 'geometric' Postnikov towers (coming
from elements of $L$ as in the proof of Corollary \ref{gepost}) and
then passing to the limit (in $\au$) with respect to $L$. Our
assertion follows easily, since each $l\in L$ is defined over some perfect
countable $k\subset l$; the limit of the spectral sequences with
respect to the subset of $L$ defined over a fixed $k$ is exactly
$T_{w_k}(H_k,X_k)$ since $H$ sends homotopy limits to inductive
limits in $\au$ (being an extended theory).

Here we certainly use the functoriality of $T$ starting from $E_2$.

\end{proof}

\begin{rema}\label{rconiunco}

1.  For a general $X\in \obj\dmge$ we only have a canonical choice
of base change maps (for $T(H_{k_l},X)$)
  starting from $E_2$; this is why we start our  spectral
  sequence from the $E_2$-level.

2. Assertion 2 of the proposition is also valid for any comotif
defined over a (perfect) countable subfield of $l$. Unfortunately, this does
not seem to include the comotives of function fields over $l$ (of
positive transcendence degrees, if $l$ is not countable).

\end{rema}

\subsection{The Chow weight structure for $\gd$}\label{stchow}

Till the end of the section, we will either assume  that
$\operatorname{char} k=0$, or that we deal with motives, comotives,
and cohomology with rational coefficients (we will use the same notation for motives with integral and rational coefficients;  cf. \S\ref{smco} below).

We prove that $\gd$ supports a weight structure that extends the
Chow weight structure of $\dmge$ (see \S6.5 and Remark 6.6.1 of
\cite{bws},  and also \cite{mymot}).

In this subsection we do not require $k$ to be countable.

\begin{pr}\label{pwchow}

1. There exists a Chow weight structure on $\dmge$ that is uniquely
characterized by the condition that all  $\mg(P)$ for $P\in \spv$ belong to its
heart; it could be extended to a weight structure $w_{\chow}$ on
$\gd$.

2. The heart of $w_{\chow}$ is the category $H_{\chow}$ of arbitrary
small  products of (effective) Chow motives.

3. We have $X\in \gd^{w_{\chow}\ge 0}$ if and only if
$\gd(X,Y[i])=\ns$ for any $Y\in \obj \chowe$, $i>0$.

4. There exists a $t$-structure $t_{\chow}$ on $\gd$ that is 
right adjacent to $w_{\chow}$ (see Remark \ref{rnice}). Its heart is the
opposite category to $\chowe{}^*$ (i.e. it is equivalent  to $(\adfu (\chowe,
\ab))^{op}$). 

5. $w_{\chow}$ respects products i.e.  $X_i\in \gd^{w_{\chow}\le
0}\implies \prod X_i\in \gd^{w_{\chow}\le 0}$ and $X_i\in
\gd^{w_{\chow}\ge 0}\implies \prod X_i\in \gd^{w_{\chow}\ge 0}$.

6. For $\prod X_i$ there exists a weight decomposition:  $\prod
X_i\to \prod X_i^{w\le 0} \to \prod X_i^{w\ge 1}$.

7. If $H:\gd\to \au$ is an extended theory, then the functor  that
sends $X$ to the  derived exact couple for $T_{w_{\chow}}(H,X)$
(see Theorem \ref{pssw}) converts all small products into direct sums.

\end{pr}
\begin{proof}
1. It was proved in (Proposition 6.5.3 and Remark 6.6.1 of) \cite{bws}
that there exists
 a unique weight structure $w'_{\chow}$ on $\dmge$ such that
$\mg(P)\in \gd^{w'_{\chow}=0}$ for
 all $P\in \spv$. 
 Moreover,
  the heart of this structure is exactly $\chowe\subset \dmge$.

Now, $\dmge$ is generated by $\chowe$.
It easily follows that
$\{\mg(P),\ P\in \spv\}$ weakly cogenerates $\gd$. Then the dual 
(see  Theorem \ref{tbw}(\ref{idual})) of 
Theorem 4.5.2(I2) of
 \cite{bws} yields that $w'_{\chow}$ could be extended to a weight structure $w_{\chow}$ for $\gd$.
Moreover, the dual to
  part II1 of loc. cit. yields that for this extension we
  have: $\hw_{\chow}$ is the
idempotent completion of $H_{\chow}$.

2. It remains to prove that $H_{\chow}$ is idempotent complete.
This is 
obvious since $\chowe$ is.

3. This is just the dual of (27) in loc. cit. 

4. The dual statement to part I2 of loc. cit. 
(cf. Remark \ref{rts}(1)) yields the existence of $t_{\chow}$.
Applying the dual
of  Theorem 4.5.2(II1) of \cite{bws} we obtain for
the heart of $t$: $\hrt_{\chow}\cong (\chowe_*)^{op}$.

5.   Theorem \ref{tbw}(\ref{iort}) easily yields that
$\gd^{w_{\chow}\le 0}$  is stable with respect to products. The
stability of $\gd^{w_{\chow}\ge 0}$   with respect to products
follows from assertion 3; here we recall that all objects of $\chowe$
are cocompact in $\gd$.

6. Immediate from the previous assertion; note that any small
product   of distinguished triangles is distinguished (see Remark
1.2.2 of \cite{neebook}).

7. Since $H$ is extended, it converts products in $\gd$ into direct
sums in $\au$. 
Hence for any $X_i\in \obj \gd$ there exist a choice of exact couples
for the corresponding weight spectral sequences for $X_i$ and $\prod
X_i$ that respects products i.e such that
$D_1^{pq}T_{w_{\chow}}(H,\prod X_i)\cong \bigoplus_i
D_1^{pq}T_{w_{\chow}}(H,X_i)$ and $E_1^{pq}T_{w_{\chow}}(H,\prod
X_i)\cong \bigoplus_i E_1^{pq}T_{w_{\chow}}(H,X_i)$ (for all $p.q\in
\z$; this isomorphism is also compatible with the connecting morphisms
of couples). Since $\au$ satisfies AB5, we obtain the isomorphism
desired for $D_2$ and $E_2$-terms (note that those are uniquely
determined by $H$ and $X$).

\end{proof}

\begin{rema}\label{rwchow}

1. In Remark 2.4.3 of \cite{bws} it was shown that 
weight spectral sequences corresponding to the Chow weight structure are isomorphic to the classical (i.e. Deligne's) weight spectral sequences  when the latter are defined (i.e. for singular or \'etale cohomology of varieties). Yet in  order to specify the choice of a weight structure here we will call these spectral sequences {\it Chow-weight} ones.

2. All the assertions of the Proposition could be extended to
arbitrary triangulated categories with negative families of
cocompact weak cogenerators (sometimes one should also demand all
products to exist; in  assertion 7 we only need $H$ to convert
all products into direct sums).

3. 
Since (effective) Chow motives are cocompact in $\gd$, 
$\hw_{\chow}$ is the category of 'formal products' of effective Chow motives
i.e. $\gd(\prod_{l\in L}X_l,\prod_{i\in I}Y_i)=\prod_{i\in
I}(\oplus_{l\in L}\chowe(X_l,Y_i))$  for $X_l,Y_l\in \obj \chowe\subset \obj \gd$
(cf.  Remark 4.5.3(2) of \cite{bws}).

4. Recall (see
\S7.1 of ibid.) that $\dme$ supports (adjacent) Chow weight and
$t$-structures (we will denote them by $w'_{\chow}$ and
$t'_{\chow}$, respectively). One could also check that these
structures are right orthogonal to the corresponding Chow structures
for $\gd$. Hence, applying  Proposition \ref{pvirdu}(1)
 repeatedly one could relate the compositions of
truncations (on $\gds\subset\gd$) via $w$ and via $t_{\chow}$ (resp.
via $w$ and via $w_{\chow}$) with truncations via $t$ and via
$w'_{\chow}$ (resp. via $t$ and via $t'_{\chow}$) on $\dme$; cf.
\S8.3 of \cite{bws}. One could also apply $w_{\chow}$-truncations
and then $w$-truncations (i.e. compose truncations in the opposite
order) when starting from an object of $\dmge$. Recall also that
truncations via $t_{\chow}$ (and their compositions with
$t$-truncations) are related with unramified cohomology; see Remark
7.6.2 of ibid.

\end{rema}

\subsection{Comparing Chow-weight and coniveau spectral sequences}\label{scchowco}

Now we prove that Chow-weight and coniveau spectral sequences are
naturally isomorphic for birational cohomology theories.

\begin{pr}\label{pcwcoss}

1. $w_{\chow}$ for $\gd$ dominates $w$ (for $\gds$) in the  sense of
\S\ref{schws}.

2. Let $H:\dmge\to \au$ be an extended cohomology theory in the
sense of Remark \ref{rcohp};  suppose that $H$ is {\it birational}
i.e. that $H(\mg(P)(1)[i])=0$ for all $P\in \spv,\ i\in \z$. Then
for any $X\in \obj \gds$ the Chow-weight spectral sequence
$T_{w_{\chow}}(H,X)$ (corresponding to $w_{\chow}$) is naturally
isomorphic starting from $E_2$ to (our) coniveau spectral sequence
$T_w(H,X)$ via the comparison morphism $M$ given by Proposition
\ref{pcwspsq}(I1).
\end{pr}

\begin{proof}
1. Let $D$ be the class of all countable products $\prod_{l\in L}
\mg(\spe K_l)$, where $\mg(\spe K_l)$ denote the comotives of (spectra of)
function fields over $k$ of bounded  transcendence degree.
Proposition \ref{pexw} and  Corollary \ref{tds2}(2) yield
that any $X\in \gds^{w=0}$ is a retract of some element of $D$. It
suffices to check that  any $X=\prod_{l\in L} \mg(\spe K_l)$ belongs to
$\gd^{w_{\chow}\ge 0}$; here we recall that $w$ is bounded and apply
Lemma \ref{lrweex}.

 By  Proposition \ref{pwchow}(5),
 we can assume that  $L$ consists of a single element.
 In this case we have $\gd(\mg(\spe K_l),\mg(P)[i])=0$
 (this is a trivial case of Proposition \ref{maneg});
 hence loc. cit. yields the result.

2. We take the same $D$ and $X$ as above.

Let $\cha k=0$. We choose $P_l\in \spv$ such that $K_l$  are their
function fields. Since all $\mg(P_l)$ are cocompact in $\gd$, we
have a natural morphism $X\to \prod \mg(P_l)$. By 
Proposition \ref{pcwspsq}(I2), it  suffices to check that $\co (X\to
\prod \mg(P_l))\in \gd^{w_{\chow}\ge 0}$, $H(X)\cong H(\prod \mg( P_l))$,
and $E_2^{**}T_{w_{\chow}}(H,\co (X\to \prod \mg(P_l)))=0$.

By  Proposition \ref{pwchow}(7) we obtain: it suffices again
to verify these statements in the case when $L$ consists of a single
element. Now, we have $\spe(K_l)=\prli \mg(U)$ for $U\in \sv,\
k(U)=K_l$. Therefore (\ref{elim}) yields:  it suffices to verify
assertions required for $Z=\mg(U\to P)$ instead, where $U\in \sv$,
$U$ is open in $P\in \spv$.

The Gysin distinguished triangle for Voevodsky's motives (see
Proposition 2.4.5 of \cite{deggenmot}) easily yields by induction that
$Z\in  \obj \dmge(1)$.

Since $\chowe$ is $-\otimes\z(1)[2]$-stable, we obtain that there
exists a $w_{\chow}$-Postnikov tower for $Z$ such that all of its terms
are divisible by $\z(1)$; this yields the vanishing of
$E_2^{**}T_{w_{\chow}}(H,Z)$. Lastly,  the fact that $Z\in
\dmge^{w'_{\chow}\ge 0}$ was  (essentially) proved by easy induction
(using the Gysin triangle) in the proof of Theorem 6.2.1 of
\cite{mymot}.

In the case $\cha k>0$, de Jong's alterations allow us to replace
$\mg(P_l)$ in the reasoning above by some Chow motives (with
rational coefficients); see Appendix B of \cite{hubka}; we will not
write down the details here.

\end{proof}

\begin{rema}
1. Hence for any $H:\gd\to \au$, $X\in \obj \gds$ there exists a comparison morphism $M:T_{w_{\chow}}(H,X)\to T_w(H',X)$, where $H'$ is the restriction of $H$ to $\gds$. It is canonical starting from $E_2$; see Proposition \ref{pcwspsq}(II).

2. Assertion 2  is not very actual for cohomology of smooth varieties
since any $Z\in \spv$ is birationally isomorphic to $P\in \spv$ (at
least for $\cha k=0$). Yet the statement becomes more interesting
when applied for  $X=\mg^c(Z)$.
\end{rema}

\subsection{Birational motives; constructing the Gersten weight
structure by gluing; other possible weight structures}\label{sgdbr}

An alternative way to prove Proposition \ref{pcwcoss}(2) is
to consider (following \cite{kabir}) the category of {\it birational
comotives}. It satisfies the following properties:

(i) All birational cohomology theories factorize through it.

(ii) Chow and Gersten weight structures induce the same weight
structure on it (see Definition \ref{dwefun}(4)).

(iii) More generally, for any $n\ge 0$ Chow and Gersten weight structures induce  weight
structures on the localizations $\gd(n)/\gd(n+1)\cong \gdbr$ (we call these localizations {\it slices}) that differ only by a shift.

Moreover, one could 'almost recover' original Chow and Gersten weight
structures starting from this single weight structure.

Now we describe the constructions and facts mentioned in more detail. We will
 be rather sketchy 
here, since we will not use the results of this subsection elsewhere
in the paper. Possibly, the  details will be written down
in another paper.

As we will show in \S\ref{snpttw} below,
the Tate twist functor could be extended (as an exact functor)
from $\dmge$ to $\gd$; this functor is compatible with (small) products. 

\begin{pr}\label{pmbr}

 I The functor $-\otimes \z(1)[1]$ is weight-exact with
 respect to $w$ on $\gds$;
   $-\otimes \z(1)[2]$ is weight-exact with respect to $w_{\chow}$
   on $\gd$ (we will say that $w$ is $-\otimes\z(1)[1]$-stable,
   and $w_{\chow}$ is $-\otimes\z(1)[2]$-stable).

II Let $\gdbr$ denote the localization of $\gd$ by $\gd(1)$;
$B$ is the localization functor. We denote  $B(\gds)$ by $\gdsbr$.

1. $w_{\chow}$ induces a weight structure $w'_{bir}$ on $\gdbr$. 
Besides, $w$ induces a weight structure $w_{bir}$ on $\gdsbr$.

2. We have  $\gdsbr^{w_{bir}\le 0}\subset \gdbr^{w'_{bir}\le 0}$,
$\gdsbr^{w_{bir}\ge 0}\subset \gdbr^{w'_{bir}\ge 0}$
(i.e. the embedding $(\gdsbr, w_{bir})\to (\gdbr, w'_{bir})$
is weight-exact).

3. For any pro-scheme $U$ we have $B(\mg(U))\in \gdsbr^{w_{bir}=0}$.

\end{pr}\begin{proof}
I This is easy, since the functors mentioned obviously map the
corresponding hearts (of weight structures) into themselves.

II 1. By assertion I, $w_{\chow}$ induces a weight structure on
$\gd(1)$ (i.e. $\gd(1)$ is a triangulated category,
$\obj \gd(1)\cap \gd^{w_{\chow}\le 0}$ and
$\obj \gd(1)\cap \gd^{w_{\chow}\ge 0}$ yield a weight structure on it).
 Hence by  Proposition 8.1.1(1) of \cite{bws} we obtain
 existence (and uniqueness) of $w'_{bir}$. The same argument 
also implies the existence of some $w_{bir}$ on $\gdsbr$. 

2. Now we  compare $w_{bir}$ with $w'_{bir}$. Since $w$ is bounded,
$w_{bir}$ also is (see loc. cit.). Hence it suffices to check that
$\hw_{bir}\subset \hw'_{{bir}}$ (see   Theorem
\ref{tbw}(\ref{iwpost})).

Moreover, it suffices to check that for $X=\prod_{l\in L} \mg(\spe K_l)$
we have $B(X)\in \gdbr^{w'_{bir}=0}$ (since  $\gdbr^{w'_{bir}=0}$ is
Karoubi-closed in $\gdbr$, here we also apply  Proposition
\ref{pwchow}(2)). As in the proof of Proposition
\ref{pcwcoss}(2), we will consider the case $\cha k=0$; the case $\cha
k=p$ is treated similarly.
Then we choose $P_l\in \spv$ such that $K_l$ are their function fields; 
we have a natural morphism $X\to \prod \mg(P_l)$. It remains to
check  that $\co (X\to \prod \mg(P_l))\in \gds(1)$. Now, since
$\gds(1)$ and the class of distinguished triangles are closed with
respect to small products, it suffices to consider the case when $L$
consists of a single element. In this case the statement is
immediate from Corollary \ref{gepost}.

3. Immediate from Corollary \ref{gepost}.


\end{proof}

\begin{rema}\label{rbir}

1. Assertion II easily implies
 Proposition \ref{pcwcoss}(2).

Indeed, any extended birational $H$ (as in loc. cit.) can be
factorized as  $G\circ B$ for a cohomological $G:\gdbr\to \au$.
Since $B$ is weight-exact with respect to $w_{\chow}$ (and its
restriction to $\gds$ is weight-exact with respect to $w$), (the
trivial case  of) Proposition \ref{pcwspsq}(I2) implies
 that for any $X\in \obj \gd$ (any choice) of $T_{w'_{bir}}(G,B(X))$
 is naturally isomorphic starting from $E_2$ to any choice of
 $T_{w_{\chow}}(H,X)$; for any  $X\in \obj \gds$ (any choice) of
  $T_{w_{bir}}(G,B(X))$ is naturally isomorphic starting from $E_2$
  to any choice of $T_{w}(H,X)$.

It is also easily seen that the isomorphism $T_{w_{\chow}}(H,X) \to
T_{w}(H,X)$ is compatible with the comparison  morphism $M$ (see
loc. cit.).

2. The proof of existence of $w_{bir}$ and of assertion 3 works with integral coefficients even if $\cha k>0$. Hence we obtain that that the category image $B(\mg(U))$, $U\in \sv$, is negative. We can apply this statement in $\cu$ being the idempotent completion of $B(\dmge)$ i.e. in the category of birational comotives. Hence   Theorem \ref{tbw}(\ref{igen})
 yields: there exists a weight structure for $\cu$ whose heart is the category of birational Chow motives (defined as in \S5 of \cite{kabir}). Note also that one can pass to the inductive limit with respect to base change in this statement (cf. \S\ref{sbchcon}); hence one does not need to require $k$ to be countable.

\end{rema}

Now we explain that $w$ and $w_{\chow}$ could be 'almost recovered'
from $(\gdbr,\, w'_{bir})$. Exactly the same reasoning as above shows
that for any $n>0$ the localization of $\gd$ by $\gd(n)$ could be
endowed with a weight structure $w'_n$ compatible with $w_{\chow}$,
whereas the localization of $\gds$ by $\gds(n)$ could be endowed
with a weight structure $w_n$ compatible with
 $w$.

 Next, we have a short exact sequence of triangulated categories
  $\gd/\gd(n)\stackrel{i_*}{\to} \gd/\gd(n+1)\stackrel{j^*}{\to} \gdbr$.
   Here the notation for functors comes from the 'classical' gluing
   data setting (cf. \S8.2 of \cite{bws}); $i_*$
  could be given by $-\otimes\z(1)[s]$ for any $s\in\z$, $j^*$ is just
   the localization.
  Now, if we
  choose $s=2$ then  $i_*$ 
 is weight-exact with respect to $w'_n$ and $w'_{n+1}$;
 if we choose $s=1$ then the restriction of $i_*$ to $\gds/\gds(n)$
  is weight-exact with respect to  $w_n$ and $w_{n+1}$.

 Next, an argument similar to the one used in
  \S8.2 of \cite{bws} shows: for any
  short exact sequence $\du\stackrel{i_*}{\to} \cu\stackrel{j^*}{\to} \eu$ of triangulated categories,
 if $\du$ and $\eu$ are endowed with weight structures, then there
  exist at most one weight structure on $\cu$ such that both $i_*$
  and $j^*$ are weight-exact (see also Lemma 4.6 of \cite{beilnmot} for
    the proof of a similar statement for $t$-structures). Hence one
    can recover $w_n$ and $w'_n$  from (copies of) $w'_{bir}$; the main
    difference between them is that the first one is
    $-\otimes \z(1)[1]$-stable, whereas the second one is
    $-\otimes \z(1)[2]$-stable. It is quite amazing that weight
    structures corresponding to spectral sequences of quite
    distinct geometric origin differ just by $[1]$ here! If one
    calls the filtration of $\gd$ by $\gd(n)$ the {\it slice filtration}
(this term was already used by A. Huber, B. Kahn, M. Levine,
V. Voevodsky, and other authors for other 'motivic' categories), then
one may say that $w_n$ and $w'_n$ could be recovered from slices; the
difference between them is 'how we shift the slices'.

 Moreover, Theorem 8.2.3 of \cite{bws} shows: if both adjoints to $i_*$
  and $j^*$ exist,
  then one can use this gluing data in order to glue (any pair) of
   weight structures for $\du$ and $\eu$ into a weight structure for
   $\cu$. So, suppose that we have a weight structure $w_{n,s}$ for
$\gd/\gd(n)$ that is $-\otimes (1)[s]$-stable and compatible with
$w'_{bir}$ on all slices (in the sense described above; so
$w'_n=w_{n,2}$,  $w_n$ is the restriction of $w_{n,1}$ to $\gds/\gds(n)$, and  all $w_{1,s}$ coincide with $w'_{bir}$). General homological algebra (see Proposition 3.3 of
\cite{krauno}) yields that all the adjoints required do exist in our case.
Hence one can
construct $w_{n+1,s}$ for $\gd/\gd(n+1)$ that satisfies similar
properties.
    So, $w_{n,s}$ exist for all $n>0$ and all $s\in \z$.
     Hence Gersten and
Chow weight structures (for $\gds/\gds(n)\subset \gd/\gd(n)$) are
members of a rather natural family of weight structures indexed by a
single integral parameter. It could be interesting to study other
members of it (for example, the one that is $-\otimes
\z(1)$-stable), though possibly  $w'_n$ is the only
member of this family whose heart is cocompactly cogenerated.

 This approach 
could  allow  constructing $w$ in the case of a not
necessarily countable $k$. Note here that the system of
$\gds/\gds(n)$ yields a fine approximation of $\gds$.
Indeed, if $X\in \spv$, $n\ge \dim X$, then Poincare duality yields: for any $Y\in \obj \dmge$ we have $\dmge(Y(n),\mg(X))\cong \dmge(Y\otimes X(n-\dim X)[-2\dim X],\z)$; this is zero if $n>\dim X$ since $\z$ is a birational motif. Hence (by Yoneda's lemma) for any $n>0$ the full subcategory of $\dmge$ generated by motives of varieties of dimension less than $n$ fully embeds into $\dmge/\dmge(n)\subset \gd/\gd(n)$. 

It follows that the restrictions of $w_{n,s}$ to a certain series of (sufficiently small) subcategories of $\gd/\gd(n)$ are induced by a single $-\otimes (1)[s]$-stable weight structure $w_s$
  for the corresponding subcategory of $\gd$. Here for  the corresponding subcategory of  $\gd/\gd(n)$ (or $\gd$) one can take the union  of the subcategories of  $\gd/\gd(n)$ (resp. $\gd$) cogenerated (in an appropriate sense) by the comotives  of (smooth) varieties of dimension $\le r$ (with $r$ running through all natural numbers). Note that this subcategory of $\gd$ contains $\dmge$. 

 We also relate briefly our results with the (conjectural) picture
 for $t$-structures described in \cite{beilnmot}. There another
 (geometric) filtration for motives was considered; this filtration
 (roughly) differs from the filtration considered above by
  (a certain version of) Poincare duality. Now, conjecturally
  the $gr_n$ of the category of birational motives with rational
  coefficients (cf. \S4.2 of ibid.) should be (the homotopy
  category of complexes over) an abelian semisimple category.
  Hence it supports a $t$-structure which is simultaneously
  a weight structure. This structure should be the building block
  of all relevant weight and $t$-structures for (co)motives.
  Certainly, this picture is quite conjectural at the present moment.

\begin{rema}\label{rrela}
The author also hopes to carry over (some of) the results of the
current paper to relative motives (i.e. motives over a base scheme that
is not a field), relative comotives, and their cohomology. One of the
possible methods for this is the usage of  gluing of weight structures
(see  \S8.2 of \cite{bws}, especially  Remark 8.2.4(3) of loc. cit.). Possibly for this situation the 'version of $\gd$' that uses motives with compact support (see \S\ref{sdual} below) could be more appropriate.

\end{rema}

\section{The construction of $\gd$ and $\gdp$; base change and
Tate twists} 
\label{rdg}

Now we construct our categories $\gdp$ and $\gd$ using
the differential graded
categories formalism.

In \S\ref{sbdedg} we recall the definitions of differential graded
categories, modules over them, shifts and cones (of morphisms).

In \S\ref{dgmod} we recall  main properties of the derived
category of (modules over) a differential graded category.

In \S\ref{scgd} we define $\gdp$ and $\gd$ as the categories opposite
to the corresponding categories of modules; then we prove that they
satisfy the properties required.

In \S\ref{sbttw} we use the differential graded
modules formalism to define base change
for motives (extension and restriction of scalars).
This yields: our results on direct summands of the comotives
(and cohomology) of function fields (proved above)
could be carried over to pro-schemes
obtained from 
them via base change.

We also define tensoring of comotives by motives,
as well as a certain
'co-internal Hom' (i.e. the corresponding left adjoint functor to
$X\otimes -$
for $X\in \obj \dmge$). These results do not require $k$ to
 be countable.

\subsection{DG-categories and modules over them}\label{sbdedg}

 We recall some basic definitions; cf. \cite{dgk}
and \cite{drinf}. 

An additive category $A$ is called graded if for any $P,Q\in\obj A$
there is a canonical decomposition $A(P,Q)\cong \bigoplus_i A^i(P,Q)$
defined;
 this decomposition satisfies
$A^i(*,*)\circ A^j(*,*)\subset A^{i+j}(*,*)$.
 A differential graded category (cf. \cite{drinf}) is a graded category
 endowed with an additive operator
$\delta:A^i(P,Q)\to A^{i+1}(P,Q)$ for all $ i\in \z, P,Q\in\obj A$.
$\delta$ should satisfy  the equalities $\delta^2=0$ (so $A(P,Q)$ is
a complex of abelian groups); $\delta(f\circ g)=\delta f\circ
g+(-1)^i f\circ \delta g$ for any $P,Q,R\in\obj A$, $f\in A^i(Q,R)$,
$g\in A(P,Q)$. In particular, $\delta (\id_P)=0$.

We denote $\delta$ restricted to morphisms of degree $i$ by
$\delta^i$.

Now we give a simple example of a differential graded category.

For an additive category $B$ we consider the category $B(B)$ whose
objects are the same as for $C(B)$ whereas for $P=(P^i)$, $Q=(Q^i)$
we define $B(B)^i(P,Q)=\prod_{j\in \z} B(P^j, Q^{i+j})$.  Obviously
$B(B)$ is a graded category. We will also consider a full
subcategory $B^b(B)\subset B(B)$ whose objects are bounded
complexes.

We set $\delta f=d_Q\circ f-(-1)^i f \circ d_P$, where $f\in
B^i(P,Q)$, $d_P$ and $d_Q$ are the differentials in $P$ and $Q$.
Note that the kernel of $\delta^0(P,Q)$ coincides with $C(A)(P,Q)$
(the morphisms of complexes); the image of $\delta^{-1}$ are the
morphisms homotopic to $0$.

Note also that the opposite category to a differential graded
category becomes  differential graded also (with the same gradings
and differentials) if we define $f^{op}\circ g^{op}=(-1)^{pq}(g\circ
f)^{op}$ for $g,f$ being composable homogeneous morphisms of degrees
$p$ and $q$, respectively.

For any differential graded $A$ we define an additive category
$H(A)$ (some authors denote it by $H^0(A)$); its objects are the
same as for $A$; its morphisms are defined as
$$H(A)(P,Q)=\ke \delta^0_A(P,Q)/\imm \delta^{-1}_A(P,Q).$$
In the case when $H(A)$ is triangulated (as a full subcategory of
the category $\dk(A)$ described below) we will say that $A$ is a
(differential graded) {\it enhancement} for $H(A)$.

We will also need $Z(A)$: $\obj Z(A)=\obj A$; $Z(A)(P,Q)=\ke
\delta^0_A(P,Q)$. We have an obvious functor $Z(A)\to H(A)$. Note
that $Z(B(B))=C(B)$; $H(B(B))=K(B)$.

Now we define (left differential graded) modules over  a small
differential graded category $A$ (cf. \S3.1 of \cite{dgk} or \S14 of
\cite{drinf}): the objects $\dgm(A)$ are those additive functors of
the underlying additive categories $A\to B(\ab)$ that preserve
gradings and differentials for morphisms. We define $\dgm(A)^i(F,G)$
as the set of transformations of additive
functors 
of degree $i$; for $h\in \dgm(A)^i(F,G)$ we define
 $\delta^i(h)= d_G\circ f-(-1)^i f \circ d_F$. We have a
natural Yoneda embedding
   $Y:A^{op}\to \dgm(A)$ (one should apply Yoneda's lemma
for the underlying additive
    categories); it is easily seen to be a full embedding
of differential graded
     categories. 

Now we define shifts and cones in $\dgm(A)$ componentwisely.
For $F\in \obj \dgm(A)$ we set $F[1](X)=F(X)[1]$. 
For $h\in \ke \delta^0_{\dgm(A)}(F,G)$ 
we define the object $\co(h)$:
$\co (h)(X)=\co(F(X)\to G(X))$ for all $X\in \obj A$.

Note that for $A=B(B)$ both of these definitions are compatible
with the corresponding notions for complexes (with respect to
the Yoneda embedding).

We have a natural triangle of morphisms in $\delta^0_{\dgm(A)}$:
 \begin{equation}\label{dgtri}
P\stackrel{f}{\to} P'\to \co(f)\to P[1].
\end{equation}

\subsection{The derived category of a differential graded
category}\label{dgmod}

We define $\dk(A)=H(\dgm(A))$. It was shown in \S2.2 of \cite{dgk}
that $\dk(A)$ is a triangulated category with respect to shifts and
cones of morphisms that were defined above (i.e. a triangle is
distinguished if it is isomorphic to those of the form
(\ref{dgtri})).

We will say that $f\in \ke \delta^0_{\dgm(A)}(F,G)$ is a {\it
quasi-isomorphism} if for any $X\in \obj A$ it yields an isomorphism
$F(X)\to F(Y)$.  We define $\dd(A)$ as the localization of $\dk(A)$
with respect to quasi-isomorphisms; so it is a triangulated
category. Note that quasi-isomorphisms yield a localizing class of
morphisms in $K(A)$. Moreover, the functor $X\to H^0(F(X)):\dk(A)\to
\ab $ is corepresented by $\dgm(A)(X,-)\in \obj \dk(A)$; hence for
any $X\in
\obj A$, $F\in \obj \dk(A)$
we have \begin{equation}\label{emorloc} 
\dd(A)(Y(X),F)\cong
\dk(A)(Y(X),F).\end{equation} Hence we have an
embedding $H(A)^{op}\to \dd(A)$.

We define $\dc(A)$ as $Z(\dgm(A))$. It is easily seen that  $\dc(A)$
is closed with respect to (small filtered) direct limits, and $\inli
F_l$  is given by $X\to \inli F_l(X)$.

Now we recall (briefly) that differential graded modules admit
certain 'resolutions' (i.e. any object is
quasi-isomorphic to a {\it semi-free}
one in the terms of \S14 of \cite{drinf}).

\begin{pr}\label{pmodstr}
There exists a full triangulated $K'\subset \dk(A)$
such that 
the projection $K(A)\to D(A)$ induces an equivalence $K'\approx
\dd(A)$. $K'$ is  closed with respect to all (small) coproducts.
\end{pr}
\begin{proof}
See \S14.8 of \cite{drinf}
\end{proof}

\begin{rema}\label{rmodel}
In fact, there exists a (Quillen) model structure for
$\dc(A)$ such that $\dd(A)$ its
homotopy category; see Theorem 3.2 of \cite{dgk}.
Moreover (for the first  model
structures mentioned in loc. cit.)  all objects of
$\dc(A)$ are fibrant,
all objects coming from $A$ are cofibrant. For this
model structure two morphisms
are homotopic whenever they become equal in $\dk(A)$.
So, one could take $K'$
 whose objects are the cofibrant objects of $\dc(A)$.

Using these facts, one could verify most of Proposition \ref{pprop}
(for $\gdp$ and $\gd$ described below).
\end{rema}

\subsection{The construction of $\gdp$ and $\gd$; the proof
of Proposition \ref{pprop}}\label{scgd}

It was proved in \S2.3 of $\cite{bev}$ (cf. also \cite{le3}
that $\dmge$ could be described as $H(A)$,  where $A$
is a certain (small) differential graded category. Moreover,
 the
functor $K^b(\smc)\to \dmge$ could be presented as $H(f)$, where $f:
B^b(\smc)\to A$ is a naturally defined differential graded functor (note still: it seems difficult to construct such a functor for the 'model' of $\dmge$ constructed in \cite{mymot}).
We will not
describe the details for (any of) these constructions since we will
not need
them.

We define $\gdp=\dc(A)^{op}$, $\gd=\dd(A)^{op}$, $p$ is the natural
projection. We verify that these categories satisfy Proposition
\ref{pprop}. Assertion \ref{ienh} follows from the fact that any
localization of a triangulated category that possesses an
enhancement is enhanceable also (see \S\S3.4--3.5 of \cite{drinf}).

The embedding $H(A)^{op}\to \dd(A)$ yields $\dmge\subset \gdp$.
Since all objects coming from $A$ are cocompact in $\dk(A)^{op}$,
Proposition \ref{pmodstr} yields that the same is true in $\gd$. We
obtain assertion \ref{ip1}.

$\gdp$ is closed with respect to inverse limits since $\dc(A)$ is
closed with respect to direct ones. Since the projection $\dc(A)\to
\dk(A)$ respects coproducts (as well as all other (filtered)
colimits), Proposition \ref{pmodstr} yields that $p$ respects
products also. We obtain assertion \ref{ip2}.

 The descriptions
of $\dc(A)$ and $\dd(A)$ yields all the properties of
shifts and cones
required. This yields assertions \ref{ip3}, \ref{ip7},
 and \ref{ip4}. Since $\dd(A)$ is a localization of $\dk(A)$, we also
 obtain assertion \ref{ip8}.

Next, since $\dd(A)$ is a localization of $\dk(A)$ with respect to
quasi-isomorphisms, we obtain assertion \ref{ip5}.

Recall that  filtered direct limits of exact sequences of abelian
groups are exact. Hence for any $X\in \obj A\subset \obj \gdp$,
$Y:L\to \dgm(A)$ we have
$$\begin{gathered} \dk(A)(\dgm(A)(X,-), \inli_l Y_l)
=H^0((\inli Y_l)(A))\\ =H^0(\inli(Y_l(A)))=\inli H^0(Y_l(A))
=\inli_l\dk(A)(\dgm(A)(X,-),Y_l) .\end{gathered}$$
Applying (\ref{emorloc}) we obtain assertion \ref{ipcf}.

It remains to verify assertion \ref{ip6} of loc. cit. Since the
inverse limit with respect to a projective system is isomorphic to
the inverse limit with respect to any its unbounded subsystem, and
the same is true for $\prli_1$ in the countable case, we can assume
that $I$ is the category of natural numbers, i.e. we have
 a sequence of $F_i$ connected by
morphisms.

In this case we have functorial morphisms $\prli F_i
\stackrel{f}{\to} \prod F_l\stackrel{g}{\to} \prod F_i$ as in
(\ref{elim}). Hence it suffices to check that these morphisms yield
a distinguished triangle in $\gd$. Note that $g\circ f=0$; hence $g$
can be factorized through a morphism $h:\co f\to \prod F_i$
in $\gdp$. Since
for any $X\in \obj A$ the morphism
$h^*: \prod_{\gdp} F_i(X)\to \co F(X)$ is a
quasi-isomorphism, $h$ becomes an isomorphism in $\gd$. This
finishes the proof.

\begin{rema}\label{run}
 1. Note that the only part of our argument when we  needed $k$ to
be countable in the proof of assertion \ref{ip6} of loc. cit.

2. The constructions of $A$ (i.e. of the 'enhancement' for $\dmge$
 mentioned above)
  that were described in \cite{bev} and  \cite{le3} 
  are easily
 seen to be functorial
 with respect to base field
  change (see below). Still, the constructions mentioned
are distinct and
 far from being the only ones possible;
   the author does not know whether all possible $\gd$
are isomorphic.
   Still, this makes no difference for cohomology
(of  pro-schemes); see Remark \ref{rcohp}.

Moreover, note that assertion \ref{ienh} of Proposition \ref{pprop}
was not very important for us (without if we would only have to consider
a certain {\it weakly exact} weight complex functor in \S\ref{scompdegl}
 below;
see \S3 of \cite{bws}).
 The
author doubts that this  condition  follows from
the other parts of Proposition \ref{pprop}.

\end{rema}


\subsection{Base change and Tate twists for comotives}\label{sbttw}

Our differential graded formalism yields certain functoriality of
comotives with
respect to embeddings
 of base fields. We construct both extension and restriction of scalars
 (the latter one for  the case of a finite extension of fields only).
 The construction of base change functors uses induction 
 for differential graded modules. This method
also allows  defining certain
 tensor products and  $\cihom$ for comotives.
In particular, we obtain a Tate twist
  functor which is compatible with (\ref{dfps})
(and a left adjoint to it).

We note that the results of this subsection (probably) could not
be deduced from the 'axioms' of $\gd$ listed in Proposition
\ref{pprop}; yet they are quite natural.

\subsubsection{Induction and restriction for differential
graded modules: reminder}\label{sindres}

We recall  certain results of \S14 of \cite{drinf} on functoriality
 of differential graded modules. These extend the corresponding
(more or less standard) statements for modules over differential graded
algebras (cf. \S14.2 of ibid.).

If $f:A\to B$ is a functor of differential graded categories,
we have an obvious {\it restriction} functor $f^*:\dc(B)\to \dc(A)$.
It is easily seen that $f^*$ also induces functors $\dk(B)\to \dk(A)$
and $\dd(B)\to \dd(A)$.
Certainly, the latter functor respects homotopy colimits
(i.e. the direct limits from $\dc(B)$).

Now, it is not difficult to construct an {\it induction}
functor $f_*:\dgm(A)\to \dgm(B)$ which is left adjoint to $f^*$;
 see \S14.9 of ibid.
By Example 14.10 of ibid, for any $X\in \obj A$ this
functor sends $X^*=A(X,-)$ to $f(X)^*$.

$f_*$ also induces  functors $\dc(A)\to \dc(B)$ and $\dk(A)\to \dk(B)$.
Restricting the latter one to the category of semi-free modules $K'$
(see Proposition \ref{pmodstr}) one obtains a functor
$Lf_*:\dd(A)\to \dd(B)$
 which is also left adjoint to the corresponding $f^*$;
see \S14.12 of \cite{drinf}.
 Since all functors of the type $X^*$ are semi-free by
 definition, we have
$Lf_* (X^*)=A(X,-)=Lf(X)^*$. 
It can also be shown that $Lf_*$ respects direct limits of
 objects of $A^{op}$
(considered as $A$-modules via the Yoneda embedding).
In the case of
countable limits this follows easily from the definition
of semi-free
modules and the expression of the homotopy  colimit in $\dd(A)$ as
 $\inli X_i=\co (\coprod X_i\to \coprod X_i)$ (this is just the
dual to (\ref{elim})). For uncountable limits, one could prove the
fact using a 'resolution' of the direct limit similar to those
described in \S A3 of \cite{neebook}.

\subsubsection{ Extension and restriction of scalars
for comotives}\label{sexre}

Now let $l/k$ be an extension of perfect fields.

Recall that $\gdp$ and $\gd$ were described (in \S\ref{scgd}) in terms
 of modules over a certain differential graded category $A$.
It was shown in
  \cite{le3} that the corresponding version of $A$ is a tensor
 (differential graded)
   category;  we also have an extension of scalars functor $A_k\to A_l$.
    It is most probable that both of these properties hold
for the version of $A$
    described in \cite{bev} (note that they obviously  hold
for $B^b(\smc)$).
    Moreover, if $l/k$ is finite, then  we have the functor of
restriction of scalars
    in inverse direction.

So, the induction for the corresponding differential graded
modules yields
an exact functor of extension of scalars  $\exts_{l/k}:\gd_k\to \gd_l$.
The reasoning above shows that $\exts_{l/k}$ is compatible
with the 'usual' extension of scalars  for smooth
varieties (and complexes of smooth correspondences).
Moreover, since $\exts_{l/k}$
respects homotopy limits, this compatibility extends to
 the comotives of pro-schemes
and their products. It can also be easily shown that
$\exts_{l/k}$ respects Tate twists.

We immediately obtain the following result.

\begin{pr} Let $k$ be countable (and perfect), let
$l\supset k$ be a perfect field.

    1. Let $S$ be a connected primitive scheme over $k$, let
$S_0$ be its generic point.
    Then $\mg{}(S_l)$ is a retract of $\mg{}(S_{0l})$
in $\gd_l$.

    2. Let $K$ be a function field  over $k$. Let $K'$
       be the residue field for a geometric valuation $v$ of $K$
       of rank
       $r$. 
Then $\mg{}(K'_l(r)[r])$ is a retract
       of $\mg{}(K_l)$ in $\gd_l$.

\end{pr}

As in  \ref{sext}, this result immediately implies similar
 statements for any cohomology of
pro-schemes mentioned (constructed from a cohomological
$H:\dmge_l\to A$ via Proposition
\ref{pextc}).

Next, if $l/k$ is finite, induction for differential
graded modules
 applied to the restriction of
scalars for $A$'s also yields
a restriction of scalars functor $\res_{l/k}:\gd_l\to \gd_k$.
Similarly to $\exts_{l/k}$,
this functor is compatible with restriction of scalars for
 smooth
 varieties, pro-schemes, and complexes of smooth correspondences;
it also respects Tate
 twists.

 It follows:  $l/k$ is finite, then $\exts_{l/k}$ maps
 $\gds{}_k$ to $\gds{}_l$; $\res_{l/k}$ maps $\gds{}_l$ to
 $\gds{}_k$. Besides, if we also
 assume
  $l$ to be countable, then  both of these functors respect
 weight structures
  (i.e. they map $\gds{}_k^{w\le 0}$ to $\gds{}_l^{w\le 0}$,
$\gds{}_k^{w\ge 0}$ to $\gds{}_l^{w\ge 0}$,
  and vice versa).

\begin{rema}\label{rkahn}
It seems that one can also define restriction of scalars via restriction
of differential graded modules (applied to the extension of
scalars for $A$'s).
 To this end one needs to check the corresponding
 adjunction for $\dmge$; the corresponding (and related)
 statement for the motivic homotopy categories was proved by J. Ayoub.
 This would allow  defining $\res_{l/k}$ also in the case when
 $l/k$ is infinite;
  this seems to be rather interesting if $l$ is a function
field over $k$.
  Note that $\res_{l/k}$ (in this case) would (probably)
also  map $\gds{}_l^{w\le 0}$
   to $\gds{}_k^{w\le 0}$ and $\gds{}_l^{w\ge 0}$ to
$\gds{}_k^{w\ge 0}$
(if $l$ is countable).
\end{rema}

\subsubsection{Tensor products and 'co-internal Hom' for comotives;
 Tate twists}\label{snpttw}

Now, for $X\in \obj A$ we apply restriction and induction of
differential graded modules for the functor $X\otimes -:A\to A$.
Induction yields a certain functor $ X\otimes -:\gd\to \gd$,
whereas
restriction yields its left adjoint which we will denote by
 $\cihom(X,-):\gd\to \gd$.
Both of them respect homotopy limits. Also,  $ X\otimes -$
is compatible with
tensoring by $X$ on $\dmge$. Besides,  the isomorphisms classes of these
functors only depend on the quasi-isomorphism class of $X$ in $\dgm(A)$.
Indeed, it is easily seen that both $ X\otimes Y$ and $\cihom(X,Y)$ are
exact with respect to $X$ if we fix $Y$; since they are
obviously zero
for $X=0$, it remains to note that quasi-isomorphic objects
 could be connected
by a chain of quasi-isomorphisms.

Now suppose that $X$ is a Tate motif i.e. $X=\z(m)[n]$, $m>0$, $n\in \z$.
Then we obtain that
 the formal Tate twists defined by (\ref{dfps}) are the true
Tate twists i.e. they are
 given by tensoring by $X$ on $\gd$.
Then recall the Cancellation Theorem for motives:
(see Theorem 4.3.1 of \cite{1},
and \cite{voevc})): $ X\otimes -$
 is a full embedding of $\dmge$ into itself.
Then one can deduce that $X\otimes -$ is
  fully faithful on $\gd$ also
(since all objects of $\gd$ come from semi-free modules over $A$).
 Moreover,
$\cihom(X,-)\circ (X\otimes -)$ is easily seen to be
isomorphic to the identity on $\gd$ (for such an $X$).

\section{Supplements}\label{ssupl}

We describe some more properties of comotives, as well as certain
possible variations of our methods and results. We will be somewhat
sketchy sometimes.

In \S\ref{scompdegl} we  define an additive
 category $\gdg$ of generic motives (a variation of those
studied in \cite{deggenmot}).
 We also prove that the exact conservative
{\it weight complex} functor
 (that exists by the general theory of weight structures)
could be modified
 to an exact conservative $WC:\gds\to K^b(\gdg)$.
Besides, 
 we prove assertions on retracts of the pro-motif of
a function field $K/k$, that are
 similar to
 (and follow from) those for its comotif.

In \S\ref{shi} we prove that $HI$ has a nice
description in terms of $\hw$.
This is a sort of Brown representability: a cofunctor
 $\hw\to \ab$ is representable by a
(homotopy invariant) sheaf with transfers whenever
it converts all small products into
direct sums.
This result is similar to the corresponding results
of \S4 of \cite{bws} (on the connection
between the hearts of adjacent structures).

In \S\ref{smco} we note that our methods could be used for
motives (and comotives) with coefficients in an
arbitrary commutative
unital ring $R$; the most
important cases are rational (co)motives and 'torsion' (co)motives.

In \S\ref{sdual} we note that there exist natural motives of
pro-schemes  with compact support in $\dme$.
It seems that one could
 construct alternative $\gd$ and $\gdp$ using
this observation (yet
this probably would not affect our main results significantly).

We conclude the section by studying which of our arguments could be
extended to the case of an uncountable $k$.

\subsection{The weight complex functor; relation with
generic motives}\label{scompdegl}

We recall that the general formalism of weight structures yields a
 conservative exact weight complex functor  $t:\gds\to K^b(\hw)$;
it is compatible with  Definition
\ref{d2}(\ref{idwpost}). Next we
prove that one can compose it with a certain 'projection' functor
without losing the conservativity.

\begin{lem}\label{lwc}

 There exists an exact conservative functor
$t:\gds\to K^b(\hw)$  that sends $X\in \obj \gds$
to  a choice of its weight complex (coming from  any choice of a weight Postnikov tower for it).
\end{lem}
\begin{proof}
Immediate from Remark 6.2.2(2)  and
Theorem 3.3.1(V) of \cite{bws}  (note that $\gds$  has a
 differential graded enhancement by 
Proposition \ref{pprop}(\ref{ip9})).

\end{proof}

Now, since all objects of $\hw$ are retracts of those
 that come via $p$ from inverse
limits of objects of $j(C^b(\smc))$, we have a natural additive
functor $\hw\to \gdn$ (see \S\ref{sprom}). 
 Its categorical image will be denoted by
$\gdg$; this is a slight modification of Deglise's category of
generic motives. We will denote the 'projection'
$\hw\to \gdg$ and $K^b(\hw)\to K^b(\gdg)$ by
$pr$.

\begin{theo}\label{twcg}
1. The  functor $WC=pr\circ t:\gds\to K^b(\gdg)$
is exact and conservative.

2.  Let $S$ be a connected primitive scheme, let $S_0$ be its generic point.
Then $pr(\mg(S))$ is a retract of $pr(\mg(S_0))$ in $\gdg$.

3. Let $K$ be a function field  over $k$. Let $K'$
be the residue field for some geometric valuation $v$ of $K$ of rank
$r$. 
Then $pr(\mg(\spe K')(r)[r])$ is a retract
 of $pr(\mg(\spe K))$ in $\gdg$.

\end{theo}

\begin{proof}

1. The exactness of $WC$ is obvious (from Lemma \ref{lwc}).
Now we check that $WC$ is conservative.

By Proposition \ref{pprop}(\ref{ip5}), it suffices to check:
if $WC(X)$ is acyclic for some $X\in \obj\gds$, then $\gd(X,Y)=0$
for all $Y\in \obj \dmge$. We denote the terms of $t(X)$ by $X^i$.

We consider the coniveau spectral sequence $T(H,X)$ for the functor
$H=\gd(-,Y)$ (see Remark \ref{rrwss}). Since $WC(X)$ is acyclic,
we obtain that the complexes $\gd(X^{-i},Y[j])$ are acyclic for all
$j\in \z$. Indeed, note that the restriction of a functor
$\gd(X^{-i},-)$ to $\dmge$ could be expressed in terms of
$pr(X^{-i})$; see Remark \ref{rspsch}. Hence
 $E_2(T)$ vanishes. Since 
$T$ converges (see  Proposition \ref{rwss}(2)) we obtain the
claim.

2. Immediate from  Corollary  \ref{tds1}(1).

3. Immediate from Corollary  \ref{tds2}(2).

\end{proof}


\begin{rema}\label{rdegl}

For $X=\mg(Z)$, $Z\in \sv$, it easily seen that $WC(X)$ could be described as a 'naive' limit of complexes of motives; cf. \S1.5.

 Now, the terms of $t(X)$ are just the
factors of (some possible) weight
  Postnikov tower
for $X$; so one can calculate them (at least, up to an isomorphism) for $X=\mg(Z)$.
Unfortunately, it seems difficult to describe the boundary for
$t(X)$ completely since $\hw$ is finer than $\gdg$.

\end{rema}

\subsection{The relation of the heart of $w$ with
$HI$ ('Brown representability')}\label{shi}

In Theorem 4.4.2(4) of \cite{bws}, for a
pair of adjacent structures $(w,t)$
for $\cu$ (see Remark \ref{rnice}) it was proved that $\hrt$ is a full
subcategory of $\hw_*(=\adfu(\hw^{op},\ab))$.
This result cannot be extended
to arbitrary orthogonal
structures since our definition of a duality did not
include any non-degenerateness
conditions (in particular, $\Phi$ could be $0$). Yet
for our main example
of orthogonal structures the statement is true;
moreover, $HI$ has a natural description
in terms of $\hw$. This statement is very similar
to a certain Brown representability-type result (for
adjacent structures)
proved in  Theorem 4.5.2(II.2) of ibid.

Note that $\hw$ is closed with respect to arbitrary small products; see
 Proposition \ref{pexw}(2).

\begin{pr}\label{pht}
$HI$ is naturally isomorphic to a full abelian
 subcategory $\hw_*'$ of $\hw_*$
 that consists of functors that convert  all products in
 $\hw$ into direct sums (of the corresponding abelian groups).
\end{pr}
\begin{proof}

First, note that for any $G\in \obj \dme$ the functor
 $\gd \to\ab $ that sends
$X\in \obj \gd$ to $\Phi(X,G)$ ($\Phi$ is the duality constructed in
 Proposition \ref{pdualn}) is cohomological. 
Moreover,
 it  converts homotopy limits
 into injective limits (of the corresponding abelian groups);
hence its restriction to
$\hw$ belongs to $\hw_*'$. We obtain an additive functor
 $\dmge\to \hw_*'$.
In fact, it factorizes through $HI$ (by (\ref{ehomhw})).
 For $G\in \obj HI$ we denote the functor $\hw\to \ab$
obtained by $G'$.

Next, for any (additive) $F:\hw^{op}\to \ab$ we define
 $F':\gds\to \ab$
by: \begin{equation}\label{efhi} F'(X)=(\ke (F(X^{0})\to
 F(X^{-1})) /\imm
(F(X^{1})\to F(X^{0}));\end{equation} here $X^i$ is a
weight complex for $X$.
 It easily seen from  Lemma \ref{lwc} that  $F'$ is a well-defined
cohomological functor. Moreover,   Theorem \ref{tbw}(\ref{iwgen})
 yields that $F'$ vanishes on $\gds^{w\le -1}$ and on $\gds^{w\ge 1}$
(since it vanishes on $\gds^{w=i}$ for all $i\neq 0$).

Hence $F'$ defines an additive functor
$F''=F'\circ \mg: \smc^{op}\to \ab$
 i.e. a presheaf with transfers. Since  $\mg(Z)\cong \mg(Z\times\af^1)$
 for any $Z\in \sv$, $F''$ is homotopy invariant.
 We should check that $F''$ is actually a (Nisnevich) sheaf. By
Proposition 5.5 of \cite{3}, it suffices to check that $F''$ is a
 Zariski sheaf. Now, the  Mayer-Vietoris
triangle for motives (\S2 of \cite{1}) yields: to any Zariski
 covering $U\coprod V\to U\cup V$ there corresponds a long exact
sequence 
{\small $$\dots\to F'(\mg(U\cap V)[1])\to F''(U\cup V)\to
F''(U)\bigoplus F''(V)\to F''(U\cap V)\to \dots $$}
Since $\mg(U\cap V)\in \gds^{w\le 0}$ by part 5 of Proposition \ref{pexw},
 we have $F'(\mg(U\cap V)[1])=\ns$; hence $F''$ is a sheaf indeed.

 So, $F\mapsto F''$ yields an additive functor $\hw_*\to HI$.

 Now we check that the functor $G\mapsto G'$ (described above)
and the restrictions of $F\mapsto F''$ to  $\hw'_*\subset \hw_*$
yield mutually inverse equivalences of the categories in question.

 (\ref{nfrestrw}) immediately yields that
the functor $HI\to HI$ that sends
 $G\in \obj HI$ to $(G')''$ is isomorphic to $\id_{HI}$.

Now for $F\in \obj \hw_*'$ we should check: for any
 $P\in \gds^{w=0}$ we have a natural
 isomorphism $(F'')'(P)\cong F(P)$. Since $\hw$ is the idempotent
completion of $H$, it suffices to consider  $P$ being of the form
$\prod_{l\in L} \mg(\spe K_l)(n_l)[n_l]$ (here $K_l$ are function
 fields over $k$, $n_l\ge 0$; $n_l$ and the transcendence degrees of $K_l/k$ are bounded);
see part 2 of Proposition \ref{pexw}.  Moreover, since $F$
converts products into direct sums,
it suffices to consider $P=\mg(\spe K')(n)[n]$ ($K'/k$ is a
function field, $n\ge 0$).
Lastly, part 2 of Corollary  \ref{tds2} reduces the situation to the
case $P=\mg(\spe K)$
($K/k$ is a function field).
 Now, by the definition of the functor $G\mapsto G'$,
 we have $(F'')'(\mg(\spe K))=\inli_{l\in L} F''(\mg(U_l))$,
 where $K=\prli_{l\in L}U_l$, $U_l\in \sv$.
We have
$F''(U_l)=\ke F(\mg(\spe K))\to F(\prod_{z\in U_l^1}\mg(z)(1)[1])$;
 here $U_l^1$ is the
set of points of $U_l$ of codimension $1$. Since
$F(\prod_{z\in U_l^1}\mg(z)(1)[1])=
\oplus_{z\in U_l^1}
F(\mg(z)(1)[1])$; 
we have $\inli_{l\in L}F(\prod_{z\in U_l^1}\mg(z)(1)[1])=\ns$;
this yields the result.

\end{proof}

\subsection{Motives and comotives with rational and
torsion coefficients}\label{smco}

 Above we considered (co)motives with integral coefficients. Yet, as
 was shown in \cite{vbook}, one could do the theory of motives with
 coefficients in an arbitrary commutative associative ring
with a unit $R$.
  One should start with the naturally defined category of
$R$-correspondences:
  $\obj(\smc_R)=\sv$; for $X,Y$ in $\sv$ we set
 $\smc_R (X,Y)=\bigoplus_U R$ for
all integral closed $U\subset X\times Y$  that are finite over $X$
and dominant over a connected component of $X$.
Then one obtains a theory of
motives that would satisfy all properties that are required
in order to deduce
the main results of this paper. So, we can define
$R$-comotives and
extend our results to them.

A well-known case of motives with coefficients are
the motives with
rational  coefficients (note that $\q$ is a flat $\z$-algebra). Yet,
one could also take $R=\z/n\z$ for any $n$ prime to
$\operatorname{char} k$.

So, the results of this paper are also valid for rational
(co)motives and 'torsion' (co)motives.

Still, note that there could be idempotents for $R$-motives that do
not come from integral ones.  In particular, for the naturally
defined rational motivic categories we have $\dmge\q\neq
\dmge\otimes \q$; also $\chowe\q \neq \chowe \otimes\q$  (here
$\chowe\q\subset \dmge\q$ denote the corresponding $R$-hulls).
Certainly, this does not matter at all in the current paper.

\subsection{Another possibility for $\gd$;  motives
with compact support of pro-schemes}\label{sdual}

In the case $\operatorname{char}k=0$, Voevodsky developed a nice
theory of motives with compact support that is compatible with
Poincare duality; see Theorem 4.3.7 of \cite{1}. Moreover, the
explicit constructions of
\cite{1} yield that the
functor of motif with compact support $\mgc:\sv^{op}\to \dmge$
is compatible with a certain
$j^c: \sv_{fl}^{op}\to C^-(\ssc)$ (which sends $X$ to the Suslin
complex of $L^c(X)$, see \S4.2 loc. cit.);
this observation was kindly communicated to the author by
 Bruno
Kahn). This allows us to define $j^c(V)$ for a pro-scheme $V$ as the
corresponding direct limit (in $C(\ssc)$).

Starting from this observation, one could try to develop an analogue
of  our  theory using the functor $\mgc$.
  One could consider $\gd=\dme^{op}$; then it
would contain $\dmge^{op}$ as the full category of cocompact
objects. It seems that our arguments could be carried over to this
context. One can construct some $\gdp$ for this $\gd$ using certain
differential graded categories.

 Though motives with compact support are Poincare dual to ordinary
 motives of smooth varieties (up to a certain Tate twist),
 we do not have a covariant embedding $\dmge\to \gd$
 (for this 'alternative' $\gd$), since (the whole) $\dmge$ is not
  self-dual. Still, $\dmge$ has a nice embedding into (Voevodsky's)
 self-dual
  category $\dmgm$; it contains an  exhausting system of self-dual subcategories.
Hence this alternative $\gd$ would yield
a theory that
is compatible with (though not 'isomorphic' to)
   the theory developed above.

 Since the alternative version of $\gd$ is closely related with
$\dme^{op}$, it  seems reasonable to call its objects
 comotives (as we
did for the objects of 'our' $\gd$).

These observations show that one can dualize all the direct summands
results of \S\ref{sapcoh} to obtain their natural analogues for
motives of pro-schemes with compact support. Indeed, to  prove them
we may apply the duals of our arguments in \S\ref{sapcoh} without any
problem; see part 2 of Remark \ref{raxiomdim}. Note that we obtain
certain direct summand statements for objects of $\dme$ this way.
This is an advantage of our 'axiomatic' approach in \S\ref{comot}.

One could also take $\gd^{op}=\cup_{n\in \z} \dmge(-n)$ (more
precisely, this is the direct limit of copies of $\dmge$ with
connecting morphisms being $-\otimes \z(1)$). Then we have a
covariant embedding $\dmge\to\dmgm\to \gd$.

 Note that both of these alternative versions of $\gd$ are
 not closed
 with respect to all (countable) products, and so not closed
 with respect to all (filtered countable) homotopy limits;
 yet they contain all products and homotopy limits that are
 required for our main arguments.


\subsection{What happens if $k$ is uncountable}\label{unck}

 We describe which of the arguments above could be
applied in the case of an uncountable
 $k$ (and for which of them the author has no idea
how to achieve this).
 The author warns that he didn't check the details thoroughly here.

As we have already noted above, it is no problem to define $\gd$,
$\gdp$, or even $\gds$ for any $k$.
 The main problem here that (if $k$ is uncountable) the comotives of
 generic points of varieties
 (and of other pro-schemes) can usually be presented
only as uncountable homotopy limits
 of motives of varieties.
 The general formalism of inverse limits (applied to
the categories of modules
 over a differential graded category) allows us to extend to this case
all parts of
 Proposition \ref{pprop} except part \ref{ip6}.
This actually means that instead of the short
  exact sequence (\ref{lmor}) one obtains a spectral
sequence whose $E_1$-terms
  are certain $\prli^j$; here $\prli^j$ is the $j$'s
  derived functor of $\prli_I$; cf. Appendix A of
\cite{neebook}. This does not
  seem to be catastrophic; yet the author has absolutely
no idea how to control
  higher projective limits in the proof of Proposition
\ref{maneg}; note that part
   2 of loc. cit. is especially important for the
construction of the Gersten weight structure.

  Besides, the author does not know how to pass to an
  uncountable homotopy limit in the Gysin distinguished triangle.
  It seems that to this end one either needs to lift the functoriality
  of the (usual) motivic Gysin triangle to $\gdp$, or to
 find a way to describe the isomorphism class of an uncountable homotopy
 limit in $\gd$
 in terms of
$\gd$-only (i.e. without fixing any lifts to $\gdp$; this seems to be
impossible in general).
  So, one could define the 'Gersten' weight tower for the comotif of a
  pro-scheme   as
  the homotopy limit of 'geometric towers'
(as in the proof of Corollary
  \ref{gepost}); yet it seems
  to be rather difficult to calculate factors of such a tower.
It seems that the problems mentioned do not become simpler for the
alternative versions of $\gd$ described  in \S\ref{sdual}.
     So, currently
the author does not know how to prove the direct summand results of
\S\ref{stds}  if $k$ is uncountable (they even could be wrong). The
problem here that the  splittings of \S\ref{stds} are not canonical
(see Remark \ref{rstds}), so one cannot apply a limit argument (as
in \S\ref{sbchcon}) here.

  It seems that constructing the Gersten weight structure is easier
  for $\gds/\gds(n)$ (for some $n>0$); 
  see \S\ref{sgdbr}.

 Lastly, one can avoid the problems with homotopy limits completely
 by restricting attention to the subcategory of Artin-Tate motives
 in $\dmge$ (i.e. the triangulated category generated by Tate twists
 of motives of finite extensions of $k$, as considered in \cite{wildat}).
 Note that coniveau spectral sequences for cohomology of such motives
 (could be chosen to be) very 'economic'.


\end{document}